\documentclass[11pt, a4paper, reqno]{amsart}
\usepackage[T1]{fontenc}
\usepackage[latin1]{inputenc}
\usepackage[english]{babel}
\usepackage[reqno]{amsmath}
\usepackage{amsthm}
\usepackage{latexsym}
\usepackage{amssymb}
\usepackage{mathrsfs}
\usepackage{mathtools}
\usepackage{mathabx}
\usepackage{lipsum}
\usepackage{titlesec}
\usepackage{tikz-cd}
\usepackage{dsfont} %identity operator \mathds{1}

\usepackage[totalwidth=15cm,totalheight=20cm, hmarginratio=1:1]{geometry}
\usepackage{mathtools}
\usepackage{color}
\usepackage{physics}
\usepackage{faktor}
\usepackage[shortlabels]{enumitem}
\usepackage{tikz} %for diagrams
\usepackage{hyperref}
\usepackage{graphicx}
\usepackage{tabularx}
\usepackage{cases} %cases with distinct labels
\newcolumntype{C}{>{\centering\arraybackslash}X} %colonna di tipo X con testo centrato
\usepackage{csquotes}
\usepackage{url}

\usepackage[style=alphabetic]{biblatex}

\usepackage{accents}

\makeatletter
\newtheorem*{rep@theorem}{\rep@title}
\newcommand{\newreptheorem}[2]{%
\newenvironment{rep#1}[1]{%
 \def\rep@title{#2 \ref{##1}}%
 \begin{rep@theorem}}%
 {\end{rep@theorem}}}
\makeatother

\makeatletter
\newtheorem*{rep@cor}{\rep@title}
\newcommand{\newrepcor}[2]{%
\newenvironment{rep#1}[1]{%
 \def\rep@title{#2 \ref{##1}}%
 \begin{rep@cor}}%
 {\end{rep@cor}}}
\makeatother

\makeatletter
\newtheorem*{rep@prop}{\rep@title}
\newcommand{\newrepprop}[2]{%
\newenvironment{rep#1}[1]{%
 \def\rep@title{#2 \ref{##1}}%
 \begin{rep@prop}}%
 {\end{rep@prop}}}
\makeatother

\newtheorem{theorem}{Theorem}[section]
\newreptheorem{theorem}{Theorem}
\numberwithin{theorem}{section}

\newtheorem{manualtheoreminner}{Theorem} % manually labelled theorem
\newenvironment{manualtheorem}[1]{%
    \renewcommand\themanualtheoreminner{#1}%
    \manualtheoreminner
}{\endmanualtheoreminner}

\newtheorem{manualcorollaryinner}{Corollary}
\newenvironment{manualcorollary}[1]{%
  \renewcommand\themanualcorollaryinner{#1}%
  \manualcorollaryinner
}{\endmanualcorollaryinner}

\newtheorem{lemma}[theorem]{Lemma}
\newtheorem{corollary}[theorem]{Corollary}

\newrepcor{corollary}{Corollary}

\newtheorem{proposition}[theorem]{Proposition}

\newrepprop{prop}{Proposition}

\theoremstyle{definition}
\newtheorem*{definition*}{Definition}
\newtheorem{definition}[theorem]{Definition}

\theoremstyle{remark}
\newtheorem{remark}[theorem]{Remark}

\makeatletter
\def\paragraph{\@startsection{paragraph}{4}%
  \z@\z@{-\fontdimen2\font}%
  {\normalfont\bfseries}}
\makeatother

\numberwithin{equation}{section}
\allowdisplaybreaks

\patchcmd{\subsection}{-.5em}{.5em}{}{}

\makeatletter

\renewcommand\section{\@startsection{section}{1}%
  \z@{.7\linespacing\@plus\linespacing}{.5\linespacing}%
  {\normalfont\scshape\centering}}

\renewcommand\subsection{\@startsection{subsection}{2}%
  \z@{-.5\linespacing\@plus-.7\linespacing}{.5\linespacing}%
  {\bfseries}}
  
\renewcommand\subsubsection{\@startsection{subsubsection}{3}%
  \z@{-.5\linespacing\@plus-.7\linespacing}{.5\linespacing}%
  {\itshape}}
  
\def\l@paragraph{\@tocline{4}{0pt}{1pc}{7pc}{}}

\makeatother

\newcommand{\C}{\mathbb{C}}
\newcommand{\R}{\mathbb{R}}

\newcommand{\Z}{\mathbb{Z}}

\newcommand{\Hyp}{\mathbb{H}^2}

\newcommand{\Teichc}{\mathcal{T}^\mathfrak{c}(\Sigma)}
\newcommand{\Teichh}{\mathcal{T}^\mathfrak{h}(\Sigma)}
\newcommand{\Teichrep}{\mathcal{T}^\mathfrak{rep}(\Sigma)}

\newcommand{\GH}{\mathcal{GH}}

\newcommand{\Lsl}{\mathfrak{sl}}

\newcommand{\Lsymp}{\mathfrak{S}}
\newcommand{\Lham}{\mathfrak{H}}

\newcommand{\diag}{\mathrm{diag}}

\newcommand{\vl}{|}
\newcommand{\Ker}{\mathrm{Ker}}

\newcommand{\PSL}{\mathbb{P}\mathrm{SL}}

\newcommand{\Id}{\mathrm{Id}}

\newcommand{\dive}{\mathrm{div}}

\newcommand{\Ree}{\mathcal{R}e}

\newcommand{\Hit}{\mathrm{Hit}}

\renewcommand{\i}{\mathbf{I}}
\renewcommand{\j}{\mathbf{J}}
\renewcommand{\k}{\mathbf{K}}

\newcommand{\g}{\mathbf{g}}

\newcommand{\Sg}{\Sigma}
\newcommand{\deft}{\mathcal{B}_0(T^2)}

\newcommand{\defg}{\mathcal{B}(\Sigma)}
\newcommand{\defgp}{\mathcal{B}_0(\Sigma)}
\newcommand{\RP}{\mathbb{R}\mathbb{P}^2}

\newcommand{\ome}{\boldsymbol{\omega}}

\newcommand{\cubicg}{Q^3\big(\Teichc\big)}
\newcommand{\cubic}{Q^3\big(\mathcal{T}(T^2)\big)}
\newcommand{\almost}{\mathcal{J}(\R^2)}
\newcommand{\almostg}{\mathcal{J}(\Sigma)}
\newcommand{\pick}{D^3(\mathcal{J}(\R^2))}
\newcommand{\pickg}{D^3(\mathcal{J}(\Sigma))}

\newcommand{\normpick}{\vl\vl A\vl\vl_0^2}

\newcommand{\dx}{\mathrm{d}x}
\newcommand{\dy}{\mathrm{d}y}

\newcommand{\charvar}{\mathfrak{R}\big(\Sg, \SL(3,\R)\big)}
\newcommand{\hitc}{\mathrm{Hit}_3(\Sigma)}
\newcommand{\bigslant}[2]{{\raisebox{.2em}{$#1$}\left/\raisebox{-.2em}{$#2$}\right.}}

\newcommand{\haffrhozero}{\mathcal{H}\mathcal S_0(\Sg,\rho)}
\newcommand{\haffrhozerotilde}{\widetilde{\mathcal{H}\mathcal S}_0(\Sg,\rho)}
\newcommand{\defgtilde}{\widetilde{\mathcal{B}}(\Sigma)}
\newcommand{\momentmap}{\boldsymbol{\mu}}
\newcommand{\haff}{\mathcal{H}\mathcal S(\Sg)}
\newcommand{\haffzero}{\mathcal{H}\mathcal S_0(\Sg)}
\newcommand{\liederivative}{\mathcal{L}}
\newcommand{\neighborhood}{\mathcal{N}_{\mathcal{F}(\Sg)}}

\DeclarePairedDelimiterX{\scal}[2]{\langle}{\rangle}{#1 \mid #2}
\DeclarePairedDelimiterX{\scall}[2]{\langle}{\rangle}{#1, #2}

\DeclareMathOperator{\End}{End}

\DeclareMathOperator{\SL}{\mathrm{SL}}

\DeclareMathOperator{\divr}{div}

\DeclareMathOperator{\Symp}{Symp}
\DeclareMathOperator{\Diff}{Diff}
\DeclareMathOperator{\Ham}{Ham}

\DeclareMathOperator{\Ad}{Ad}

\begin{document}

\setcounter{secnumdepth}{3}
\setcounter{tocdepth}{2}

\title[Pseudo-K\"ahler structure on $\Hit_{3}(\Sigma)$ and Goldman symplectic form]{Pseudo-K\"ahler structure on the $\mathrm{SL}(3,\R)$-Hitchin component and Goldman symplectic form}

\author[Nicholas Rungi]{Nicholas Rungi}
\address{NR: Universit\'e Grenoble Alpes (Institut Fourier), Grenoble, France.} \email{nicholas.rungi@univ-grenoble-alpes.fr} 

\author[Andrea Tamburelli]{Andrea Tamburelli}
\address{AT: Department of Mathematics, University of Pisa, Italy.} \email{andrea.tamburelli@libero.it}

\date{\today}

\begin{abstract}
    The aim of this paper is to show the existence and give an explicit description of a pseudo-Riemannian metric and a symplectic form on the $\SL(3,\R)$-Hitchin component, both compatible with Labourie and Loftin's complex structure. In particular, they give rise to a mapping class group invariant pseudo-K\"ahler structure on a neighborhood of the Fuchsian locus, which restricts to a multiple of the Weil-Petersson metric on Teichm\"uller space. By comparing our symplectic form with Goldman's $\ome_G$, we prove that the pair $(\ome_G, \i)$ cannot define a K\"ahler structure on the Hitchin component. %As a consequence, we obtain that Teichm\"uller space of the surface embeds as a totally geodesic submanifold of the $\PSL(3,\R)$-Hitchin component, with respect to the Weil-Petersson K\"ahler metric.
\end{abstract}

\maketitle

\tableofcontents

\section{Introduction}
Historically, the deformation space $\defg$ of convex $\RP$-structures on a smooth, closed and orientable surface $\Sg$ of genus $g\ge 2$, was the very first example of a higher rank Teichm\"uller space parameterizing geometric structures on the surface (see \cite{Wienhard_intro} for a survey on this topic). It generalizes the notion of hyperbolic structures and, because of this, it contains a copy of Teichm\"uller space $\Teichh$. In the early nineties, Goldman proved that $\defg$ is homeomorphic to a cell of real dimension $-8\chi(\Sg)$ (\cite{goldman1990convex}) and by using works of Koszul (\cite{koszul1965varietes},\cite{koszul1968deformations}), it was shown (\cite{goldman1990convex}) that $\defg$ embeds as an open submanifold in the representation space $\mathrm{Hom}\big(\pi_1(\Sg), \SL(3,\R)\big)/\SL(3,\R)$. Three years later, Choi and Goldman (\cite{choi1993convex}) proved that $\defg$ is indeed homeomorphic to a connected component Hit$_3(\Sg)$ of the $\SL(3,\R)$-character variety, found by Hitchin using techniques from Higgs bundles theory (see \cite{hitchin1992lie} for a more general result). It consists of representations obtained as deformation of the \emph{Fuchsian} ones, i.e. those arising as holonomy of a hyperbolic structure, and nowadays it is called the \emph{Hitchin component}. In the early 2000s, Labourie (\cite{Labourie_cubic}) and Loftin (\cite{loftin2001affine}) proved independently, using the theory of hyperbolic affine spheres and harmonic maps in symmetric spaces, that the space $\defg$ can be endowed with a mapping class group invariant complex structure $\i$. Such a complex structure comes from the identification between $\defg$ and the total space of the holomorphic bundle of cubic differentials over $\Teichh$. \\ \\ A natural question, after the parameterization by Labourie and Loftin, is to look for a Riemannian metric (or symplectic form) on $\defg$ that gives rise to a K\"ahler structure, extending the Weil-Petersson metric, once coupled with $\i$. The first attempt was made by Loftin (\cite{loftin2001affine}) with the Goldman symplectic form $\ome_G$ (\cite{goldman1984symplectic}), defined using the algebraic description $\defg\cong$ Hit$_3(\Sg)$ found by Choi and Goldman. Nowadays, it is still unclear whether $\ome_G(\i\cdot,\cdot)$ is a Riemannian metric. Later on, three more Riemannian metrics on the $\SL(3,\R)$-Hitchin component were defined: one by Darvishzadeh and Goldman (\cite{darvishzadeh1996deformation}), one by Li (\cite{li2013teichm}) and another by Bridgeman-Canary-Labourie-Sambarino (\cite{bridgeman2015pressure}) called \emph{pressure metric} (defined also on much more general spaces). Regarding the first two it has been shown that they restrict to a multiple of the Weil-Petersson metric on Teichm\"uller space, which is totally geodesic in Hit$_3(\Sg)$ with respect to the metric found by Li. As far as pressure metric is concerned, very little is known and this is partly due to its complicated expression (\cite{labourie2018variations},\cite{dai2019geodesic}). In all three cases the relation with Labourie and Loftin's complex structure is unknown. \\ \\ Recently, Kim and Zhang (\cite{kim2017kahler}), using various notions of positivity for holomorphic bundles on K\"ahler manifolds, have succeeded in showing the existence of a K\"ahler metric on Hit$_3(\Sg)$, which restricts to a multiple of the Weil-Petersson one on the Fuchsian locus. Even if this metric is natural, namely invariant under the action of the mapping class group, the relation of its complex structure with the one found by Labourie and Loftin is still mysterious (see the discussion in \cite[\S 1.2 and \S 1.3]{labourie2017cyclic}). \\ 

\noindent In this paper we study the symplectic and pseudo-Riemannian geometry of Hit$_3(\Sg)$ using a different approach coming from the theory of symplectic reduction in an infinite-dimensional context. Recall that a \emph{pseudo-Riemannian} metric $\g$ on a smooth $n$-manifold $M$ is an everywhere non-degenerate, smooth, symmetric $(0,2)$-tensor. Let $\mathbf{I}$ be a complex structure on $M$, then $(\g,\mathbf{I})$ is a \emph{pseudo-Hermitian structure} if \begin{equation*}
    \g(\mathbf{I}X,\mathbf{I}Y)=\g(X,Y), \qquad\forall X,Y\in T_pM, p\in M \ . 
\end{equation*} The \emph{fundamental 2-form} $\ome$ of a pseudo-Hermitian manifold $(M,\g,\mathbf{I})$ is defined by: \begin{equation*}
    \ome(X,Y):=\g(X,\mathbf{I}Y), \qquad\forall X,Y\in T_pM, p\in M \ . 
\end{equation*}
A pseudo-Hermitian manifold $(M,\g,\mathbf{I},\ome)$ is called \emph{pseudo-K\"ahler} if the fundamental $2$-form is closed, namely if $\mathrm{d}\ome=0$. In this case the corresponding metric is called \emph{pseudo-K\"ahler}.\\

\noindent The main result of the paper is the following: 

\begin{manualtheorem}A \label{thmA}
There exists a bi-linear, alternating and closed $2$-tensor $\ome$ on $\hitc$ such that $\g(\cdot,\cdot):=\ome(\i \cdot,\cdot)$ defines an indefinite symmetric bi-linear form. Moreover, the triple $(\g,\ome,\i)$ gives rise to a mapping class group invariant pseudo-K\"ahler structure on a neighborhood of the Fuchsian locus in the Hitchin component, and it restricts to a multiple of the Weil-Petersson K\"ahler metric on Teichm\"uller space. The signature of the pseudo-Riemannian metric $\g$ in the above neighborhood is $(6g-6,10g-10)$.
\end{manualtheorem}

\noindent There is a natural well-defined action of the mapping class group MCG$(\Sg)$ on the spaces $\defg$ and $\hitc$ %given respectively by: \begin{equation}\label{MCGaction}[\psi]\cdot[f,M]:=[f\circ\psi, M],\qquad [\psi]\cdot[\rho]:=[\rho\circ\psi_*]\end{equation}for $[\psi]\in\textrm{MCG}(\Sg)$, $[f,M]\in\defg$ and $[\rho]\in\hitc$.
for which the monodromy map is an equivariant isomorphism (see Theorem \ref{teo:choigoldman}). Therefore, one gets an induced pseudo-K\"ahler structure on a neighborhood $\mathcal{N}_{\Teichh}$ of Teichm\"uller space in $\defg$. The quotient $\mathcal{C}(\Sg)$ of the deformation space by the mapping class group is a complex orbifold and it is smooth at generic points (\cite[Proposition 4.1.2]{loftin2001affine}). The above discussion together with Theorem \ref{thmA} implies the following:
\begin{manualcorollary}B\label{cor:B}
    There exists an orbifold neighborhood $\mathcal{N}_{\mathcal M_g}$ of the moduli space of Riemann surfaces of genus $g\ge 2$ inside $\mathcal C(\Sg)$ endowed with a pseudo-K\"ahler orbifold structure. Such a structure restricts to a multiple of the Weil-Petersson orbifold K\"ahler structure on $\mathcal M_g$.
\end{manualcorollary}

\noindent The tensors $\ome$ and $\g$ will be explicit and defined on the whole $\hitc$ but, because we cannot exclude that $\ome$ might be degenerate outside the Fuchsian locus, the triple $(\ome,\g,\i)$ defines a-priori a pseudo-K\"ahler structure only on a neighborhood of it. We will also give an interpretation of our pseudo-K\"ahler structure in terms of hyperbolic affine spheres. \\ 

\noindent The above theorem is inspired by a similar result obtained in the case of maximal globally hyperbolic anti-de Sitter three-manifolds (\cite{mazzoli2021parahyperkahler}), where the authors developed part of the techniques the we used in our work (see Section \ref{sec:1.6}). \\

\noindent The identification between $\mathcal{B}(S)$ and the  holomorphic vector bundle of cubic differentials over Teichm\"uller space induces a natural circle action on $\defg$ that corresponds to the rotation along the fibres on $\cubicg$, i.e. $e^{i\theta}\cdot([J],q):=([J],e^{-i\theta}q)$. The pseudo-K\"ahler structure we found in Theorem \ref{thmA} behaves well with respect to this circle action:

\begin{manualtheorem}C \label{thmC}
Let $\rho$ be a fixed area form on $\Sg$, then the circle action on $\hitc$ is Hamiltonian with respect to $\ome$ and it satisfies: \begin{equation}\label{isometriaazionecircolare}\Psi_\theta^*\g=\g,\quad\forall\theta\in\R \ .\end{equation}The Hamiltonian function is given by: $$H(J,q):=\frac{2}{3}\int_\Sg f\bigg(\frac{\vl\vl q\vl\vl^2_{g_J}}{2}\bigg)\rho,$$ where $f:[0,+\infty)\to(-\infty,0]$ is the smooth function defined by (\ref{definitionf}) and $g_J(\cdot,\cdot):=\rho(\cdot,J\cdot)$ is a Riemannian metric on the surface.
\end{manualtheorem}
\noindent In the theorem above, we improperly used the notion of "Hamiltonian action" since we know that the symplectic form $\ome$ is non-degenerate only on a neighborhood of the Fuchsian locus in $\hitc$. Nevertheless, the relation (\ref{isometriaazionecircolare}) and the moment map equation defined by $H(J,q)$ continue to hold on the whole Hitchin component.\\

\noindent By work of Goldman, the Hitchin component $\Hit_3(\Sigma)$ has a natural symplectic form $\ome_G$. One of the main open question about the geometry of $\Hit_3(\Sigma)$ is whether the complex structure $\i$ that the Hitchin component inherits from the Labourie-Loftin parameterization is compatible with Goldman symplectic form and the pair $(\ome_G, \i)$ gives rise to a K\"ahler structure on the Hitchin component. By comparing our symplectic form $\ome$ with $\ome_G$, we give a negative answer to the above question:

\begin{manualtheorem}D \label{thmGoldman}
The symplectic forms $\ome$ and $\ome_G$ coincide on the Fuchsian locus along all directions, tangent and transverse to the Fuchsian locus. In particular, even if the complex structure $\i$ were compatible with $\ome_G$, the metric $\ome_G(\i\cdot, \cdot)$ would be indefinite with signature $(6g-6,10g-10)$, thus the pair $(\ome_G, \i)$ cannot define a K\"ahler structure on the Hitchin component. 
\end{manualtheorem}

\subsection{Strategy of the proofs}\label{sec:1.5}
In this section we explain with some details the techniques used to construct the Hitchin component as an infinite dimensional symplectic quotient. Throughout the exposition we will state other fundamental and non-immediate results needed for the proof of Theorem \ref{thmA}, Theorem \ref{thmC} and Theorem \ref{thmGoldman}. On the one hand the overall strategy follows the lines of the anti-de Sitter case (\cite[\S 1.7]{mazzoli2021parahyperkahler}), but on the other we encountered more difficulties during some steps of the proof that will be explained on a case-by-case basis (see also Section \ref{sec:1.6} for a more general discussion). \vspace{1em}\newline \underline{\emph{The genus one case}}:\vspace{1em}\\ The case when $\Sg=T^2$ was studied by the authors in a previous work (\cite{rungi2021pseudo}). Let us give a brief summary of what has been proven, as the methods used for the genus one case will later be a key part of the general case. Let $\almost$ be the space of (almost) complex structures on $\R^2$ compatible with the standard area form $\rho_0=\dx\wedge\dy$, namely all the endomorphisms $J:\R^2\to\R^2$ such that $J^2=-\mathds{1}$ and for which $\{v, Jv\}$ is a positive basis, whenever $v\neq 0$. For any such $J$, let $g_J^0(\cdot,\cdot):=\rho_0(\cdot,J\cdot)$ be the associated scalar product on $\R^2$. There is an identification between this space and the hyperbolic plane $\Hyp$, so that the action of $\SL(2,\R)$ on $\Hyp$ by M\"obius transformations results in an action by conjugation on $\almost$. We introduced an explicit family of $\SL(2,\R)$-invariant pseudo-K\"ahler metrics $(\hat\g_f,\hat\i,\hat\ome_f)$, parametrized by a smooth function $f:[0,+\infty)\to(-\infty,0]$, on a vector bundle $\pick$ over $\almost$, whose fibre over a point $J$ is given by $\{A\in \End_0(\R^2, g_J^0)\otimes T^*\R^2 \ | \ A(J\cdot)=A(\cdot)J, \ A(X)Y=A(Y)X \}$, where $A(X)$ is a trace-less and $g_J^0$-symmetric endomorphism of $\R^2$, for any $X\in T\R^2$. Each element of the pseudo-K\"ahler structure preserves the complement of the zero section in $\pick$, which can be identified with the deformation space $\deft$. Moreover, under the above identification the mapping class group action on $\deft$ corresponds to the restricted $\SL(2,\Z)<\SL(2,\R)$-action on $\pick$ minus the zero section (see Theorem 3.13 in \cite{rungi2021pseudo}). \vspace{1em}\newline %In the end, the main result we get is:\begin{theorem}[\cite{rungi2021pseudo}]Let $f:[0,+\infty)\to(-\infty,0]$ be a smooth function such that:\begin{itemize} \item[(i)]$f'(t)<0, \ \forall t\ge0$ \item[(ii)]$\displaystyle\lim_{t\to+\infty}f(t)=-\infty$ \item[(iii)] $f(0)=0$. \end{itemize}Then the space $\deft$ admits a mapping class group invariant pseudo-K\"ahler metric $(\g_f,\ome_f,\mathbf{I})$. The pseudo-Riemannian metric $\g_f$ and the symplectic form $\ome_f$ both depend on $f$ and $\mathbf{I}$ is the complex structure coming from the identification of $\deft$ with the complement of the zero section in $\pick$. \end{theorem}
\underline{\emph{The general case}}:\vspace{1em}\newline
Now let $\Sg$ be a smooth closed connected and oriented surface of genus $g\ge 2$. The crucial step in moving from the genus one case to the higher genus case, consists in the following construction, which will be developed in full details in Section \ref{sec:4.1}. Let $\rho$ be a fixed area form on $\Sg$, then for any (almost) complex structure $J$ on $\Sg$, let $g_J:=\rho(\cdot,J\cdot)$ be the associated Riemannian metric. Now consider the space formed by pairs $(J,A)$, where $J$ is an (almost) complex structure on $\Sg$ compatible with the given orientation and $A$ is a $1$-form with values in the bundle of trace-less and $g_J$-symmetric endomorphisms such that $A(J\cdot)=A(\cdot)J$ and $A(X)Y=A(Y)X, \ \forall X,Y\in\Gamma(T\Sg)$. This space, denoted by $\pickg$, is of infinite dimension and it carries a pseudo-K\"ahler structure as its analogue $\pick$. In fact, one can choose an area-preserving linear isomorphism from $\R^2$ to $T_x\Sg$, which induces an identification between $\pick$ and $D^3\big(\mathcal J(T_x\Sg)\big)$. Since the pseudo-K\"ahler metric on $\pick$ is $\SL(2,\R)$-invariant, the induced structure does not depend on the chosen area-preserving linear isomorphism. Then, one can (formally) integrate each element of the pseudo-K\"ahler structure on $\Sg$, evaluated on first-order deformations $(\dot J,\dot A)$. Slightly more in detail, let $P$ be the $\SL(2,\R)$-frame bundle over $\Sg$ whose fibres over a point $x\in\Sg$ are linear maps $F:\R^2\to T_x\Sg$ such that $F^*\rho_x$ is the standard area form on $\R^2$. Let us define the fibre bundle $$P\big(\pick\big):=\bigslant{P\times \pick}{\SL(2,\R)}$$where $\SL(2,\R)$ acts diagonally on the two factors. The space $\pickg$ can be identified with the space of smooth sections of such fibre bundle. Hence, as explained above, one can introduce the following formal symplectic form $$(\ome_f)_{(J,A)}\big((\dot J,\dot A),(\dot J',\dot A')\big):=\int_{\Sg}\hat\ome_f\big((\dot J,\dot A),(\dot J',\dot A')\big)\rho$$ and the formal pseudo-Riemannian metric $$(\g_f)_{(J,A)}\big((\dot J,\dot A),(\dot J',\dot A')\big):=\int_{\Sg}\hat\g_f\big((\dot J,\dot A),(\dot J',\dot A')\big)\rho$$where $\hat\ome_f$ and $\hat\g_f$ are both induced on each fibre of $T^{\text{vert}}P\big(\pick\big)$. Here we denoted by $T^{\text{vert}}P\big(\pick\big)$  the vertical sub-bundle of $TP\big(\pick\big)$ with respect to the projection map $P\big(\pick\big)\to\Sg$. Similarly, a complex structure $\i$ is obtained on the infinite-dimensional space $\pickg$ of smooth sections, by applying point-wise $\hat\i$, which is defined on $\pick$. It should be noted that the symplectic form $\hat\ome_f$ and the pseudo-Riemannian metric $\hat\g_f$ both depend on the choice of a smooth function $f$, as they arise from the construction on $\pick$. In particular, the expression for $\ome_f$ and $\g_f$ combined with $\i$ effectively gives us a (formal) family of pseudo-Kahler metrics on the space of smooth sections $\pickg$. Instead, we are interested in inducing such structures on a certain submanifold, whose elements $(J,A)$ will be identified with the set of embedding data of hyperbolic affine spheres in $\R^3$ (see Section \ref{sec:2.3}). In order to do so, a particular choice of the function $f$ appearing in the expression of $\ome_f$ and $\g_f$ has to be made. Let $F:[0,+\infty)\to\R$ be the unique smooth function such that $ce^{-F(t)}-2te^{-3F(t)}+1=0$, where $c$ is a constant depending only on the topology and the area of $(\Sg, \rho)$. Let us define a new metric in the same conformal class of $g_J$ by the formula $h:=e^{F(t)}g_J$, where the function $F$ is computed in $\vl\vl A\vl\vl^2_{g_J}$ (the norm of the tensor $A$ with respect to $g_J$) divided by $8$. Then, imposing equations (\ref{Gausscodazzi}) on the pair $(h,A)$ we get a $\Ham(\Sg,\rho)$-invariant submanifold $\haffrhozerotilde$ of the space of smooth sections $(J,A)$, whose quotient $\defgtilde$ by $\Ham(\Sg,\rho)$, is a smooth manifold of dimension $16g-16+2g$. This will be a consequence of a simple application of Moser's trick in symplectic geometry, of the particular choice of the function $f$ in terms of $F$, and finally of the existence and uniqueness of hyperbolic affine sphere immersions in $\R^3$ (see the discussion in Section \ref{sec:2.3} for more details). It turns out that such a manifold is not diffeomorphic to the deformation space of convex $\RP$-structures, as its dimension exceeds that of $\defg$ by $2g$. As we shall see later, the tangent space to this manifold splits as the $\g_f$-orthogonal direct sum of the tangent space to $\defg$ and the tangent to the orbit of harmonic vector fields (Lemma \ref{prop:orthogonaldecompositionW}). For this reason, the further (finite-dimensional) quotient of $\defgtilde$ by $\Symp_0(\Sg,\rho)/\Ham(\Sg,\rho)\cong H^1_{\text{dR}}(\Sg,\R)$ gives us the desired deformation space.\vspace{1.0em}\newline
\underline{\emph{The candidate for the tangent space to the Hitchin component}}:\vspace{1.0em}\newline
As explained in the previous paragraph, in order to actually obtain the Hitchin component from the space $\haffrhozerotilde$, and thus induce a pseudo-K\"ahler structure $(\g_f,\i,\ome_f)$ on it, we need to perform two quotients: the first by $\Ham(\Sg,\rho)$ and the second by $\Symp_0(\Sg,\rho)/\Ham(\Sg,\rho)$. The idea is to define a distribution $\{W_{(J,A)}\}_{(J,A)}$ of $\Ham(\Sg,\rho)$-invariant subspaces inside the tangent space to $\haffrhozerotilde$. Each vector space $W_{(J,A)}$ of this distribution will be defined by a system of partial differential equations and will be point-wise isomorphic to the tangent space of $\defgtilde$. The first result that conceals a number of technical difficulties, developed in Section \ref{sec:4.3}, shows, using an argument from the theory of elliptic operators on compact manifolds, that the dimension of each $W_{(J,A)}$ is bounded below by the expected dimension of the quotient manifold.\begin{manualtheorem}E \label{thmD}
Let $(J,A)$ be a point in the infinite-dimensional space $\haffrhozerotilde$. Let $W_{(J,A)}$ be the vector space of solutions of the following system:\begin{equation}\label{differentialequations}
\begin{cases}
\mathrm d\big(\divr\big((f-1)\dot J\big)+\mathrm d\dot f\circ J-\frac{f'}{6}\beta\big)=0  \\ \mathrm d\big(\divr\big((f-1)\dot J\big)\circ J+\mathrm d\dot{f}_0\circ J-\frac{f'}{6}\beta\circ J\big)=0 \\ \mathrm d^\nabla\dot A_0(\bullet,\bullet)-J(\divr\dot J\wedge A)(\bullet,\bullet)=0
\end{cases}
\end{equation}where $\beta(\bullet):=\langle(\nabla_\bullet A)J, \dot A_0\rangle$ is a $1$-form and $\dot{f}_0=-\frac{f'}{4}\langle A, \dot A_0J \rangle$ is a smooth function on $\Sg$. Then, $\dim W_{(J,A)}\ge 16g-16+2g$.
\end{manualtheorem}
\noindent The second difficult statement, which will be consequence of the above theorem, also involves a large number of technical details (carried out in Section \ref{sec:4.4}). It allows us to identify each subspace $W_{(J,A)}$ with the tangent space to the first quotient space $\defgtilde$ at the point $(J,A)$.
\begin{manualtheorem}F \label{thmE}
For every element $(J,A)\in\haffrhozerotilde$, the vector space $W_{(J,A)}$ is contained inside $T_{(J,A)}\haffrhozerotilde$ and it is invariant by the complex structure $\i$. Moreover, the collection $\{W_{(J,A)}\}_{(J,A)}$ defines a $\Ham(\Sg,\rho)$-invariant distribution on $\haffrhozerotilde$ and the natural projection $\pi:\haffrhozerotilde\to\defgtilde$ induces a linear isomorphism $$\mathrm d_{(J,A)}\pi:W_{(J,A)}\longrightarrow T_{[J,A]}\defgtilde$$
\end{manualtheorem}
\noindent In particular, we can restrict the pseudo-K\"ahler structure $(\g_f,\i,\ome_f)$ from the ambient space to the finite dimensional manifold $\defgtilde$. Since the pseudo-metric $\g_f$ is not positive-definite, it is not immediate that it is still non-degenerate when restricted to the subspaces $W_{(J,A)}$. This will be the biggest issue to be addressed and will be discussed in the last paragraph of this section. At this point, one can proceed by performing the finite-dimensional quotient $\defgtilde/H$, where $H:=\Symp_0(\Sg,\rho)/\Ham(\Sg,\rho)\cong H^1_{\text{dR}}(\Sg,\R)$. Such a quotient is isomorphic to $\defg$ (see Section \ref{sec:6.1}), hence to the Hitchin component. In particular, there is a $\g_f$-orthogonal decomposition $W_{(J,A)}=V_{(J,A)}\oplus S_{(J,A)}$, where $V_{(J,A)}$ is the tangent to $\hitc$ and $S_{(J,A)}$ is a copy of $H$ (Lemma \ref{lem:complexsymplecticsubspace}). %Again, however, we need to ask the question whether the pseudo-metric is indeed non-degenerate once induced on the $\g_f$-orthogonal to the space of harmonic vector fields inside $W_{(J,A)}$. This will be achieved through a nontrivial integration by parts that we will develop in Section ??. 

\begin{manualtheorem}G \label{thmF}
The $H$-action on $\defgtilde$ is free and proper, with complex and symplectic $H$-orbits. Moreover, the pseudo-K\"ahler structure $(\g_f,\i,\ome_f)$ descend to the quotient which is identified with $\hitc$. Finally, the complex structure $\i$ induced on the $\SL(3,\R)$-Hitchin component coincides with the one found by Labourie and Loftin.\end{manualtheorem}  %Moreover, there exists a neighborhood $\widetilde{\mathcal{N}}_{\mathcal{F}(\Sg)}$ of the submanifold given by those points $(J,A)\in\defgtilde$ with $A=0$ where the metric $\g_f$ is non-degenerate. In particular, the $H$-orbits in $\widetilde{\mathcal{N}}_{\mathcal{F}(\Sg)}$ are complex symplectic submanifolds and the pseudo-K\"ahler structure $(\ome_f,\g_f,\i)$ descends to the quotient $\widetilde{\mathcal{N}}_{\mathcal{F}(\Sg)}/H\cong \mathcal{N}_{\mathcal F(\Sg)}\subset\hitc$.

 \vspace{1.0em}\noindent\underline{\emph{The relation with moment maps and symplectic reduction}}\vspace{1.0em}\newline
While Theorem \ref{thmD} and Theorem \ref{thmE} can be proven with self-contained arguments, it is not clear how to obtain the differential equations (\ref{differentialequations}) defining the subspace $W_{(J,A)}$. In fact, their origin must be sought in the context of moment maps and symplectic reductions, but in an infinite-dimensional context. For this reason, we will briefly explain how to characterize the subspaces $W_{(J,A)}$ in these terms and how the presence of isotropic vectors for $\g_f$ generates further difficulties. \newline In a previous paper, where we studied the same problem on the torus, we showed that the action of $\SL(2,\R)$ on $\pick$ is Hamiltonian with respect to the symplectic form $\hat\ome_f$ and we computed explicitly the moment map $\widehat\mu:\pick\to\mathfrak{sl}(2,\R)^*$ (\cite{rungi2021pseudo}). A general theorem of Donaldson (see Section \ref{sec:5.1}), allows us to promote the previous result to a Hamiltonian action of $\Ham(\Sg, \rho)$ on $\pickg$, with respect to the symplectic form $\ome_f$. In this case, the moment map $\momentmap$  associates to each pair $(J,A)\in\pickg$ an element in the dual Lie algebra of Hamiltonian vector fields on the surface. It turns out that to obtain an honest moment map $\widetilde{\momentmap}$ for the action of the group of Hamiltonian diffeomorphisms, one has to add a scalar multiple of the area form $\rho$. %Since the group $\Ham(\Sg,\rho)$ is normal in $\Symp_0(\Sg,\rho)$, it is natural to ask whether $\widetilde\momentmap$ can be promoted to a moment map for the group of all surface symplectomorphisms. %It will be seen, from the explicit construction, that it cannot be equivariant for $\Symp_0(\Sg,\rho)$ but, in our particular case, we will be able to show that it will satisfy the other property required in the moment map definition (see ??). 
At this point, it can be shown that the submanifold $\widetilde\momentmap^{-1}(0)$ intersected with the set $\mathcal M_\mathrm C=\{(J,A) \in \pick \ | \ d^{\nabla}A=0\}$ (see Section \ref{sec:2.3} and Section \ref{sec:5.2}) is equal to $\haffrhozerotilde$. Inspired by classical symplectic reduction theory, one is tempted to induce the pseudo-Riemannian metric $\g_f$ and the symplectic form $\ome_f$ on the quotient $\big(\widetilde{\momentmap}^{-1}(0)\cap\mathcal M_\mathrm C\big)/\Ham(\Sg,\rho)$. The issue is that, in our case, the tangent space $T_{(J,A)}\pickg$ is not a Hilbert space, and there is no canonical way to %a Krein space. Roughly speaking, a Krein space is a (real or complex) infinite-dimensional vector space endowed with an indefinite inner product which admits an orthogonal direct sum decomposition in positive and negative part. Moreover, the pseudo-metric restricted to both the positive and negative part induces a complete norm (see Definition \ref{def:Kreinspaces}). The presence of the indefinite metric does not allow us, like in the Hilbert case (\cite[Theorem 1.3.2]{tromba2012teichmuller}), to
identify the $\g_f$-orthogonal to the $\Ham(\Sg,\rho)$-orbit inside $\haffrhozerotilde$ with the $\i$-invariant distribution tangent to the finite dimensional manifold $\defgtilde$. Despite that, by imitating the reduction in the positive-definite case, we are able to give a characterization of the subspace $W_{(J,A)}$ as follows:\begin{manualtheorem} H \label{thmG}
For any $(J,A)\in\haffrhozerotilde$, the vector space $W_{(J,A)}$ is the largest subspace in $T_{(J,A)}\haffrhozerotilde$ that is: \begin{itemize}
    \item[$\bullet$] invariant under the complex structure $\i$;
    \item[$\bullet$] $\g_f$-orthogonal to the orbit $T_{(J,A)}\big(\Ham(\Sg,\rho)\cdot(J,A)\big)$.
\end{itemize}
\end{manualtheorem}
\noindent The proof of this theorem will be developed in Section \ref{sec:5.3} and it is independent of the other results. On the other hand, it serves as a motivation for defining the subspace $W_{(J,A)}$ as the solution of a system of partial differential equations (\ref{differentialequations}). \vspace{1.0em}\newline\underline{\emph{Regarding the non-degeneracy of the pseudo-metric.}}\vspace{1.0em}\newline Theorem \ref{thmE} together with Theorem \ref{thmF} allow us to induce the pseudo-K\"ahler structure $(\g_f,\i,\ome_f)$ from the infinite-dimensional manifold $\pickg$ to the Hitchin component, but, a-priori, it may be degenerate. We will show in the proof of Theorem \ref{thmA} (Section \ref{sec:4.2}) that, at least on the Fuchsian locus, there are no non-zero degenerate vectors. As for the tangent directions to points away from $\mathcal F(\Sg)$, the analysis becomes very complicated. On the one hand, we know the exact expression of $\g_f$, but on the other hand, the model $V_{(J,A)}$ of the tangent space to the Hitchin component is described by very complicated PDEs (Remark \ref{rem:equationsHitchincomponent}), whose solution is far from being explicit. The idea, is to look for a subspace of $T_{(J,A)}\pickg$ (possibly of infinite dimension) whose elements have a treatable description for our purpose. This is the tangent space $T_{(J,A)}\mathcal M_\mathrm C$ to the set of pairs $(J,A)\in\pickg$ satisfying the Codazzi-like equation $\mathrm d^\nabla A=0$ for hyperbolic affine spheres (see (\ref{Gausscodazzi})). We will show that the set $T_{(J,A)}\mathcal M_\mathrm C$ contains the tangent space to the $\Diff(\Sg)$-orbit, which in turn splits as a direct sum of three subspaces (Lemma \ref{lem:vectorfieldssurface}). Then, using the relation between the PDEs describing $V_{(J,A)}$ and the theory of symplectic reduction, the following $\g_f$-orthogonal decomposition of $T_{(J,A)}\mathcal M_\mathrm C$ can be obtained: 
$$V_{(J,A)}\overset{\perp_{\g_f}}{\oplus} S_{(J,A)}\overset{\perp_{\g_f}}{\oplus} T_{(J,A)}\big(\Ham(\Sg,\rho)\cdot(J,A)\big)\overset{\perp_{\g_f}}{\oplus}\i\Big(T_{(J,A)}\big(\Ham(\Sg,\rho)\cdot(J,A)\big)\Big) \ . $$
The existence of the moment map $\widetilde\momentmap$ for the action of $\Ham(\Sg,\rho)$ on $\pickg$ and an explicit calculation allow us to conclude that $\g_f$ restricted to the Hamiltonian orbit is non-degenerate. Moreover, using a highly non-trivial integration by parts (Proposition \ref{prop:integrazioneperparti}) we prove that $\g_f$ is non-degenerate even when restricted to the subspace $S_{(J,A)}$. Finally, using the relation $\g_f(\i\cdot,\i\cdot)=\g_f(\cdot,\cdot)$ one gets the following equivalence: $\g_f$ is non-degenerate on $T_{(J,A)}\mathcal M_\mathrm C$ if and only if it is non-degenerate on $V_{(J,A)}$ (Corollary \ref{cor:nondegeneratemetriconVandMc}).

\vspace{1.0em}

\noindent \underline{\emph{The relation with Goldman symplectic form}}\vspace{1.0em}

\noindent In \cite{Goldman_affine}, Goldman gave a gauge theoretic formula for his symplectic form on the $\mathrm{SL}(3,\mathbb{R})$-character variety: by identifying a Hitchin representation $\rho:\pi_1(\Sigma)\rightarrow \mathrm{SL}(3,\mathbb{R})$ with the holonomy of a flat connection $\nabla$, which is nothing but the Blaschke connection of the $\rho$-equivariant affine sphere, the pairing 
\[
    \ome_G(\dot{\nabla}, \dot{\nabla}') = \int_\Sigma \tr(\dot{\nabla} \wedge \dot{\nabla}') -\frac{1}{3} \tr(\dot{\nabla}) \wedge \tr(\dot{\nabla}')
\]
defines a non-degenerate symplectic form on the Hitchin component. The techniques we develop in Section \ref{sec:5.2} will allow us to compute explicitly the variations $\dot{\nabla}$ in terms of variations $\dot{J}$ and $\dot{A}$ of the embedding data of the corresponding affine spheres, at least at points on the Fuchsian locus, in \textit{all} directions, tangent or transverse to the Fuchsian locus. By direct inspection, we show that $\ome_G(\i \cdot, \cdot)$ coincides with our pseudo-K\"ahler metric $\g$ at the Fuchsian locus in all directions, thus showing that the pair $(\ome_G, \i)$ cannot induce a positive definite K\"ahler metric on the Hitchin component.

%In this regard, in Section \ref{sec:6.3} we introduce the notion of Krein space and some useful results that may lead to a better understanding of (possible) degenerate vectors for $\g_f$ away from the Fuchsian locus (Section \ref{sec:6.4}).%\newline On the other hand, using the theory of symplectic reduction, one is tempted to say that proving the metric is non-degenerate on the orbit is sufficient to get the same conclusion on the quotient space. Unfortunately, in our case, the Hitchin component $\hitc$ is cut by two equations (see (\ref{Gausscodazzi}) and Proposition \ref{prop:mosertrick}) and only one has an interpretation as a moment map. It is not clear to use, whether the same holds for $\mathrm d^\nabla A=0$, or equivalently, $\bar\partial_Jq=0$ for a cubic differential $q$ on the surface (see the discussion in Section \ref{sec:6.4}).
\subsection{Comparison with the anti-de Sitter case}\label{sec:1.6} As mentioned earlier, part of the techniques we use to construct the $\SL(3,\R)$-Hitchin component as a symplectic quotient and the definition of the pseudo-K\"ahler metric are based on a previous work (\cite{mazzoli2021parahyperkahler}), where the authors defined a para-hyperk\"ahler structure on the deformation space of maximal globally hyperbolic anti-de Sitter $3$-manifolds, denoted with $\mathcal{M}\GH(\Sg)$. Such a deformation space can be identified with a maximal component in the $\PSL(2,\R)\times\PSL(2,\R)$ character variety, which consists entirely of discrete and faithful representations. In particular, such a space is parameterized by two copies of Teichm\"uller space (\cite{mess2007lorentz},\cite{krasnov2007minimal}) and it is isomorphic to the cotangent bundle $T^*\mathcal{T}(\Sg)$ (\cite{krasnov2007minimal}).\newline As can be seen, the first major difference lies in the fact that $\hitc$ cannot be isomorphic to $\mathcal{T}(\Sg)\times\mathcal{T}(\Sg)$, since its real dimension is equal to $16g-16$. This does not allow, unlike the anti-de Sitter case, to define a natural para-complex structure $\j$ that together with $\i$ gives rise to another para-complex structure $\k:=\i\j$. Moreover, the parameterization $\mathcal M\GH(\Sg)\cong T^*\mathcal{T}(\Sg)$ as a holomorphic vector bundle, gives rise to a complex symplectic structure on $\mathcal{M}\GH(\Sg)$, which is missing for the Hitchin component. This is the reason why with our construction we only obtain a pseudo-K\"ahler metric. \newline The different descriptions as holomorphic vector bundles over $\mathcal{T}(\Sg)$ lead to different computations along the way. In fact, in the $\PSL(2,\R)\times\PSL(2,\R)$ setting one has to work with a pair given by a complex structure and a holomorphic quadratic differential on $\Sg$, the real part of which corresponds to the second fundamental form of the immersion as a maximal surface in AdS $3$-manifolds, namely it is an endomorphism of $T\Sg$. In our case, the real part of a holomorphic cubic differential is, up to the contraction with the metric, an End$(T\Sg)$-valued $1$-form. On the one hand, the additional $1$-form part makes the analysis more difficult, but on the other we still succeed in obtaining similar results in regard to some key steps in the construction (Proposition \ref{prop:differentialourmomentmap} and Proposition \ref{prop:integrazioneperparti}). \newline The presence of other two moment maps in the AdS seeting (\cite[Theorem 6.5]{mazzoli2021parahyperkahler}), allowed the authors to obtain $\mathcal M\GH(\Sg)$ as the quotient of an infinite-dimensional space by the group of all symplectomorphisms of the surface isotopic to the identity. In our case, not knowing whether the equation $\mathrm{d}^\nabla A=0$ can be interpreted as a moment map, we had to resort to the use of two quotients, which led to further difficulties developed in Section \ref{sec:6.1}. It is also worth mentioning that since the PDEs defining the distribution tangent to the deformation space are much more complicated in our setting, it was necessary to employ a deep analysis of the associated differential operators (Section \ref{sec:4.3}). \newline Finally, the most relevant part: the pseudo-metric is non-degenerate on the deformation space. In \cite{mazzoli2021parahyperkahler}, the authors were able to identify the three symplectic forms they defined on $\mathcal{M}\GH(\Sg)$ with already known symplectic forms (thus non-degenerate), in terms of the various parameterizations given above. In our setting, we know that, at least on the Fuchsian locus, our symplectic form coincides with Goldman's one both on horizontal and vertical directions. Away from the Fuchsian locus, the relation is still unknown but we nonetheless gain some new information about Goldman's symplectic form, as explained in the introduction.  %do not know what the relation between our symplectic form $\ome$ and Goldman's one is, because of the particular choice of function $f$ that must be made and on which $\ome$ depends. This led us to a careful analysis of the involved infinite-dimensional spaces and to obtain some partial results towards the non-existence of degenerate vectors for $\g$ away from the Fuchsian locus (Section \ref{sec:6.2} and \ref{sec:6.4}).\NR{Da cambiare}
\section*{Acknowledgements}\noindent
The authors are grateful to Filippo Mazzoli and Andrea Seppi for reading an early version of this paper and for their recommendations that improved the exposition. Part of this work has been carried out while the first author was visiting the second author at the University of Pisa.

\section{Preliminaries}
\subsection{The \texorpdfstring{$\SL(3,\R)$}{PSL(3,R)}-Hitchin component}\label{sec:1.1} We briefly introduce the Hitchin component for the group $\SL(3,\R)$ and we explain the topological type of the associated character variety. For a more detailed and general discussion on Hitchin components and character varieties see \cite{goldman1984symplectic}, \cite{hitchin1992lie},\cite{franccois2006anosov},\cite{fock2006moduli} and \cite{collier2019studying} for a survey.\\ \\
Let $\Sg$ be a closed, connected smooth and oriented surface of genus $g\ge 2$ and consider the space Hom$(\pi_1(\Sg),\SL(3,\R))$ of all representations from $\pi_1(\Sg)$ to $\SL(3,\R)$. This set has a topology induced by the inclusion \begin{align*}\mathrm{Hom}(\pi_1(\Sg),  & \ \SL(3,\R))\hookrightarrow\SL(3,\R)^{2g} \\ &\rho\longmapsto\big(\rho(a_1),\dots,\rho(b_g)\big)\end{align*}where $a_1,\dots,b_g$ are generators of $\pi_1(\Sg)$ subject to the relation $\prod_{i=1}^g\big[a_i,b_i\big]=1$. There is a natural action of $\SL(3,\R)$ on this space given by conjugation: for $\gamma\in\pi_1(\Sg)$ and $P\in\SL(3,\R)$ \begin{equation}\label{PSL3action}
    (P\cdot\rho)(\gamma):=P^{-1}\rho(\gamma)P \ .
\end{equation}In order to get a Hausdorff quotient space, one needs to restrict to the \emph{completely reducible} representations, i.e. those $\rho:\pi_1(\Sg)\to\SL(3,\R)$ which split as a direct sum of irreducible representations. If we denote with Hom$^+(\pi_1(\Sg),\SL(3,\R))$ the space of such representations, the quotient space $$\charvar:=\bigslant{\mathrm{Hom}^+(\pi_1(\Sg),\SL(3,\R))}{\SL(3,\R)}$$
is called the $\SL(3,\R)$-\emph{character variety}.
\begin{theorem}[Hitchin \cite{hitchin1992lie}]
The topological space $\charvar$ has three connected components: the one containing the class of the trivial representation, the one consisting of representations whose associated flat $\R^3$-bundles have non-zero second Stiefel-Whitney class and the one consisting of representations connected to those arising as uniformization. Moreover, the third one is contained in the smooth locus of $\charvar$ and it is diffeomorphic to $\R^{-8\chi(\Sg)}$.
\end{theorem}

%It must be noted that the there is no topological invariant which distinguishes the first component to the third one, as they are both formed by representations whose associated flat $\R^3$-bundle has zero second Stiefel-Whitney class. 
The most interesting component in the above list is the last one, which will be denoted by Hit$_3(\Sg)$ throughout the discussion. It can also be defined as the connected component of $\charvar$ containing the \emph{Fuchsian representation} $\tau\circ j$, where $j:\pi_1(\Sg)\to\PSL(2,\R)$ is discrete and faithful and $\tau$ is the unique (up to conjugation) irreducible representation from $\PSL(2,\R)$ to $\SL(3,\R)$. %In Hitchin's original paper (\cite{hitchin1992lie}) it was called the "Teichm\"uller component" since it seemed to be a natural generalization of the Teichm\"uller component $\Teichrep$ for $\PSL(2,\R)$, which is actually contained in Hit$_3(\Sg)$. %Nowadays it is known as the \emph{Hitchin component} and for our particular case (also for $\PSL(n,\R)$) there is a quite explicit description of its construction and of the inclusion $\Teichrep\hookrightarrow\mathrm{Hit}_3(\Sg)$ which we now describe. Let us identify $\R^3$ with the space of homogeneous polynomials in two variables $x,y$ of degree 2, i.e. $\R^3\cong\mathrm{Span}_\R\{x^2,xy,y^2\}$. There is an action of $\SL(2,\R)$ on such space: $$\begin{pmatrix} a & b \\ c & d \end{pmatrix}\cdot x^{2-i}y^i:=(ax+cy)^{2-i}(bx+dy)^i, \qquad i=0,1,2 \ $$which induces a (unique up to conjugation) representation $\tau_3:\SL(2,\R)\to\SL(3,\R)$ given by: $$\tau_3\bigg(\begin{pmatrix}  a & b \\ c & d \end{pmatrix}\bigg)=\begin{pmatrix} a^2 & ab & b^2 \\ 2ac & ad+bc & 2bd \\ c^2 & cd & d^2\end{pmatrix} \ .$$ 
 %It is immediate to see that one gets an induced representation $\PSL(2,\R)\to\PSL(3,\R)$ still denoted by $\tau_3$. For any discrete and faithful representation $j:\pi_1(\Sg)\to\PSL(2,\R)$, the composition $\tau_3\circ j:\pi_1(\Sg)\to\PSL(3,\R)$ is discrete and faithful as well. The Hitchin component can be defined as the connected component of $\charvar$ containing $\tau_3\circ j$, i.e. it is formed by all the representations obtained as deformations of the \emph{Fuchsian} ones. 
In particular, the composition $\tau\circ j$ induces an inclusion of $\Teichrep$ in $\hitc$, whose image is called the \emph{Fuchsian locus} and it will be denoted by $\mathcal{F}(\Sg)$. \\ \\ There is a geometric interpretation for representations in $\hitc$, found by Choi and Goldman, which we now recall. %Let $\Sg$ be a smooth, closed and orientable surface. 
A \emph{(properly) convex} $\RP$-\emph{structure} on $\Sg$ is a pair $(\phi,M)$, where $\phi:\Sg\to M$ is a diffeomorphism (called the \emph{marking}) and $M\cong \Omega/_{\Gamma}$ is a (properly) convex $\RP$-surface, namely $\Omega$ is a (bounded) convex subset of $\RP$ preserved by a properly discontinuous action of a discrete subgroup $\Gamma<\SL(3,\R)$. 
%\end{definition}
One can define an equivalence relation on such pairs: $(\phi_1,M_1)\sim (\phi_2,M_2)$ if and only if there exists an element $\Psi\in\Diff_0(\Sg)$ such that the map $\phi_2\circ\Psi\circ \phi_1^{-1}:M_1\to M_2$ is a projective isomorphism. The \emph{deformation space of (properly) convex $\RP$-structures} is defined as:
\begin{align*}
    &\defg:=\bigslant{\{(f,M) \ \text{convex} \ \RP-\text{structure on} \ \Sg\}}{\Diff_0(\Sg)} \\ &\defgp:=\bigslant{\{(f,M) \ \text{properly convex} \ \RP-\text{structure on} \ \Sg\}}{\Diff_0(\Sg)} \ .
\end{align*}
When $g\ge 2$ these two spaces coincide and therefore every convex $\RP$-structure is actually properly convex (\cite{kuiper1953convex},\cite{benzecri1960varietes}). To any equivalence class of convex $\RP$-structures on $\Sg$ there is an associated class of representations $[\rho]$, with $\rho:\pi_1(\Sg)\to\SL(3,\R)$. In particular, this association defines the so-called \emph{monodromy map} $\mathfrak{hol}:\defg\to\charvar$ whose image is contained in the space of discrete and faithful representations.  \begin{theorem}[\cite{goldman1990convex}, \cite{choi1993convex}]\label{teo:choigoldman} The map $\mathfrak{hol}:\defg\to\charvar$ induces an isomorphism between $\defg$ and $\mathrm{Hit}_3(\Sg)$. In particular, any deformation of a Fuchsian representation $\tau\circ j:\pi_1(\Sg)\to\SL(3,\R)$ can be realized as the holonomy of a convex $\RP$-structure on $\Sg$. \end{theorem}

\subsection{Hyperbolic affine spheres}\label{sec:1.2}
Here we briefly introduce the theory of hyperbolic affine spheres in $\R^3$ (see \cite{loftin2013cubic},\cite{nomizu1994affine}) and their relation with convex $\RP$-structures. %These tensors will play an essential role in the costruction of the pseudo-K\"ahler metric. In the discussion we will deal with the case where the immersed surface is closed and of genus $g\ge 2$, but everything can also be done in much more generality.
\\ \\
Let $\Sg$ be a closed surface of genus $g\ge 2$ with universal cover $\widetilde\Sg$ and let $f\!: \widetilde\Sg \to \R^{3}$ be an immersion with $\tilde\xi\!:\widetilde\Sg\to\R^{3}$ a transverse vector field to $f(\widetilde\Sg)$. This means that for all $x\in\widetilde\Sg$ we have a splitting: $$T_{f(x)}\R^{3} = f_* T_x\widetilde\Sg + \R\tilde\xi_x \ . $$ 
Let $D$ be the standard flat connection on $\R^3$ and suppose the structure equations of the immersed surface are given by:
\begin{equation}\begin{aligned}\label{structurequations}
D_XY &= \nabla_XY + h(X,Y)\xi \\
D_X\xi &= -S(X)
\end{aligned}\end{equation}
where $\nabla$ is a torsion-free connection on $\widetilde\Sg$ called the \emph{Blaschke connection}, $\xi$ is the \emph{affine normal} of the immersion (see \cite[\S 3.1]{loftin2001affine} for example), $h$ is a metric on $\widetilde\Sg$ called the \emph{Blaschke metric} and $S$ is an endomorphism of $T\widetilde\Sg$ called the \emph{affine shape operator}.
\begin{definition}
Let $N$ be an immersed hypersurface in $\R^{3}$ with structure equations given by (\ref{structurequations}). Then $N$ is called a \emph{hyperbolic affine sphere} if $S=-\Id_{TN}$. 
\end{definition}
The properties of the global geometry of hyperbolic affine spheres were conjectured by Calabi (\cite{calabi1972complete}) and proved by Cheng-Yau (\cite{cheng1977regularity}, \cite{cheng1986complete}) and Calabi-Nirenberg (with clarifications by Gigena (\cite{gigena1981conjecture}) and Li (\cite{li1990calabi}, \cite{li1992calabi})). The most important result (stated only in $\R^3$ but true in arbitrary $\R^n$) is the following: \begin{theorem}[Cheng-Yau-Calabi-Nirenberg]\label{thmchengyau}
Given a constant $\lambda<0$ and a convex, bounded domain $\Omega\subset\R^2$, there is a unique properly embedded hyperbolic affine sphere $N\subset\R^{3}$ with affine shape operator $S=\lambda\cdot\Id_{TN}$ and center $0$ asymptotic to the boundary of the cone $\mathcal{C}(\Omega):=\{(tx,t) \ | \ x\in\Omega, t>0\}\subset\R^{3}$. For any immersed hyperbolic affine sphere $f: N\rightarrow\R^{3}$, properness of the immersion is equivalent to the completeness of the Blaschke metric, and any such $N$ is a properly embedded hypersurface asymptotic to the boundary of the cone given by the convex hull of $N$ and its center.
\end{theorem}

We can use the above theorem to describe a $\Diff(\Sg)$-equivariant one-to-one correspondence between convex $\RP$-structures and hyperbolic affine spheres. In fact, given a convex $\RP$-structure $\phi:\Sg\to M\cong\Omega/\Gamma$, where $\Omega\subset\R^2$ is bounded, there exists a unique hyperbolic affine sphere $\mathcal H\subset\R^3$ asymptotic to the boundary of the cone $\mathcal C(\Omega)\subset\R^3$ (Theorem \ref{thmchengyau}). Such a hyperbolic affine sphere $\mathcal H$ is invariant under automorphisms of $\mathcal{C}(\Omega)$, seen as a subgroup of $\SL(3,\R)$. The restriction of the projection $\pi:\mathcal C(\Omega)\to\Omega$ induces a diffeomorphism of $\mathcal H$ onto $\Omega$. By equivariance, the tensor $h$ and the connection $\nabla$ descend to the quotient $\Omega/\Gamma\cong M$. Viceversa, given an embedding of the universal cover $\widetilde\Sg\hookrightarrow\R^3$ as a $\widetilde\Gamma$-equivariant hyperbolic affine sphere, with $\widetilde\Gamma\cong\pi_1(\Sg)$, one gets an identification of $\widetilde\Sg$ with a domain $\Omega\subset\RP$, via the developing map. Then, Theorem \ref{thmchengyau} implies that $\widetilde\Sg$ is asymptotic to a cone over $\Omega$. The action of $\widetilde\Gamma$ on $\widetilde\Sg\subset\R^3$ corresponds to an action of a group $\Gamma<\SL(3,\R)$, isomorphic to $\pi_1(\Sg)$, on the domain $\Omega$ so that $\Sg\cong\Omega/\Gamma$. 
%\begin{corollary}\label{cor:affineconvexproj}Let $\Sg$ be a closed surface of genus $g\ge 2$, then the following are equivalent: \newline $\bullet$ the pair $(\phi,M)$, where $\phi:\Sg\to M$ is a diffeomorphism and $M\cong\Omega/\Gamma$ is a convex $\RP$-surface, defines a convex $\RP$-structure on $\Sg$; \newline $\bullet$ there is an immersion $f:\widetilde\Sg\to\R^3$ as a hyperbolic affine sphere invariant under the action of a discrete group $\Gamma$ consisting entirely of unimodular affine transformations and isomorphic to $\pi_1(\Sg)$. \newline Moreover, the above correspondence is invariant under the action of $\Diff(\Sg)$.\end{corollary}
%In particular, there is a bijective correspondence between the deformation space of convex $\RP$-structures on $\Sg$ and the space of immersions $f:\Sg\to\R^3$ as a hyperbolic affine sphere, equivariant for the action of a discrete subgroup $\Gamma$ consisting entirely of unimodular affine transformations and isomorphic to $\pi_1(\Sg)$.  
\\ \\
Let $f:\!(\widetilde\Sg, h, \nabla)\rightarrow\R^{3}$ be an immersed hyperbolic affine sphere, where $h$ is the Blaschke metric and $\nabla$ is the Blaschke connection. If $\nabla^h$ denotes the Levi-Civita connection with respect to $h$, then $\nabla=\nabla^h+A$, where $A$ is a section of $T^*(\Sg)\otimes\End(T\Sg)$ called the \emph{Pick form}. In particular, for every $X\in\Gamma(T\Sg)$ the quantity $A(X)$ is an endomorphism of $T\Sg$.
%\begin{proposition}[{\cite[Lemma 4.3, Lemma 4.4]{benoist2013cubic}}]\label{prop:pickform}The section $A$ has the following properties:\begin{itemize}\item[(1)] $A(X)Y=A(Y)X, \qquad\forall X,Y\in\Gamma(T\Sg)$\item[(2)] The endomorphism $A(X)$ is trace-free and $h$-symmetric, $\forall X\in\Gamma(T\Sg)$ \item[(3)] $\mathrm{d}^{\widehat\nabla}A=0$, where $\mathrm{d}^{\widehat\nabla}A(X,Y)=\big(\widehat\nabla_XA\big)(Y)-\big(\widehat\nabla_YA\big)(X),\qquad X,Y\in\Gamma(T\Sg) \ .$\end{itemize}\end{proposition}

\begin{definition}\label{def:picktensor}
The \emph{Pick tensor} is the $(0,3)$-tensor defined by \begin{equation}\label{picktensorandpickform} C(X,Y,Z):=h(A(X)Y,Z), \quad\forall X,Y,Z\in\Gamma(T\Sg) \ .\end{equation}
\end{definition}

\begin{corollary}\label{picktensoraffinesphere}
If $f:\!(\widetilde\Sg, h, \nabla)\hookrightarrow\R^3$ is an immersed hyperbolic affine sphere, then the Pick tensor is totally symmetric, namely in index notation $C_{ijk}$ we have $$C_{ijk}=C_{\sigma(ijk)},\quad\forall\sigma\in\mathfrak S_3 \ .$$
In particular, this is equivalent to the requirement that the endomorphism $A(X)$ is $h$-symmetric for all $X\in\Gamma(T\Sg)$ and \begin{equation}
    A(X)Y=A(Y)X, \quad \forall X,Y\in\Gamma(T\Sg) \ .\end{equation}
\end{corollary}
%Since the manifold $\Sg$ is of dimension $2$, one can relate the Pick tensor to the conformal structure on the surface:
\begin{theorem}[{\cite[Lemma 4.8]{benoist2013cubic}}]\label{thm:picktensor}
Let $\Sg$ be a closed oriented surface of genus $g\ge 1$. Let $h$ be a Riemannian metric on $\Sg$ and $J$ be the induced (almost) complex structure. Suppose that a $(1,2)$ tensor $A$ and a $(0,3)$ tensor $C$ are related by $A=h^{-1}C$. Assume further that the tensor $C$ is totally symmetric. Then, $A(X)$ is trace-free for all $X\in\Gamma(T\Sg)$ if and only if $C$ is the real part of a complex cubic differential, which can be expressed as $q=C(\cdot,\cdot,\cdot)-iC(J\cdot,\cdot,\cdot)$. If this holds, then the following are equivalent: \newline $\bullet \ \mathrm{d}^{\nabla^h}A=0$;\newline
$\bullet \ C$ is the real part of a holomorphic cubic differential $q=C(\cdot,\cdot,\cdot)-iC(J\cdot,\cdot,\cdot)$;\newline
$\bullet \ (\nabla^h_{JX}A)(\cdot)=(\nabla^h_XA)(J\cdot), \forall X\in\Gamma(T\Sg)$.
\end{theorem}
%At this point, it should be clear that it is possible to describe the space of (equivariant) immersions of hyperbolic affine spheres $f:\widetilde\Sg\to \R^3$ in terms of the Pick tensor and the Blaschke metric. In particular, we can obtain another parametrization of the deformation space of convex $\RP$-structures on $\Sg$ (see Corollary \ref{cor:affineconvexproj}) and thus of the $\PSL(3,\R)$-Hitchin component (see Theorem \ref{teo:choigoldman}). \\ \\
The embedding data of hyperbolic affine spheres in $\R^3$ can be described in terms of the Blaschke metric $h$ and the Pick form $A$ satisfying the following equations: \begin{equation}\label{Gausscodazzi}\tag{HS}
    \begin{cases}
    K_h-\vl\vl q\vl\vl_h^2=-1 \\ \mathrm d^{\nabla^h}A=0 \ ,
    \end{cases}
\end{equation}where $q=C(\cdot,\cdot,\cdot)-iC(J\cdot,\cdot,\cdot)$ is the holomorphic cubic differential determined by the Pick tensor $C$, $K_h$ is the Gaussian curvature of the Blaschke metric $h$ and $A=h^{-1}C$ is the associated End$_0(T\Sg, h)$-valued 1-form. Moreover, for any tangent vector fields $X,Y,Z$ on $\Sg$ the exterior-derivative $\mathrm d^{\nabla^h}A$ is the $\mathrm{End}_0(T\Sg, h)$-valued 2-form   \begin{equation}
    \big(\mathrm d^{\nabla^h}A\big)(X,Y)Z=\big(\nabla^h_XA\big)(Y)Z-\big(\nabla^h_YA\big)(X)Z \ , 
\end{equation}where End$_0(T\Sg, h)$ denotes the vector bundle of $h$-symmetric and trace-less endomorphisms of the tangent bundle. \begin{remark}\label{rem:conformalinvarianceHS}Notice that the second equation in (\ref{Gausscodazzi}) is invariant under conformal change of metric. In fact, it is equivalent to require that the cubic differential $q=C(\cdot,\cdot,\cdot)-iC(J\cdot,\cdot,\cdot)$ is holomorphic with respect to the complex structure defined by the conformal class of $h$. For this reason, in the following discussion, we will use either the tensor $A$ or $C$ according to which is more convenient.
\end{remark}
Viceversa, every pair $(h,C)$ satisfying (\ref{Gausscodazzi}), with $h$ a Riemannian metric and $C$ a totally symmetric $(0,3)$-tensor equal to the real part of a $h$-cubic differential, i.e. a cubic differential that is holomorphic for the conformal class of $h$, represents the embedding data of a hyperbolic affine sphere in $\R^3$ (\cite{changping1990some},\cite{loftin2001affine},\cite{benoist2013cubic}). Considering that such a correspondence is natural by the action of $\Diff(\Sg)$, we introduce the space parametrizing the embedding data of $\pi_1(\Sg)$-equivariant hyperbolic affine spheres in $\R^3$ as: \begin{equation}
    \haff:=\bigslant{\left\{ (h,C) \ \Bigg | \ \parbox{19em}{$h$ is a Riemannian metric \\ $C$ is the real part of a $h$-cubic differential \\ equations (\ref{Gausscodazzi}) are satisfied} \right\}}{\Diff_0(\Sg)} 
\end{equation}
Thus, according to the above discussion, we obtain the following result:
\begin{proposition}\label{prop:defgandhaff}
    Let $\Sg$ be a closed surface of genus $g\ge 2$, then there exists a MCG$(\Sg)$-invariant homeomorphism between $\defg$ and $\haff$, given by the embedding data of the unique equivariant hyperbolic affine sphere.
\end{proposition}
Because of this identification, for the rest of the discussion we will equivalently use one of the two pieces of notation in Proposition \ref{prop:defgandhaff} to denote the deformation space of convex $\RP$-structures, hence the $\SL(3,\R)$-Hitchin component.

\subsection{Wang's equation}\label{sec:2.1}
Here we discuss the relation between the hyperbolic affine sphere immersion $f:\widetilde\Sg\to\R^3$ and the conformal geometry of the surface. In particular, it is possible to rewrite the structure equations (\ref{structurequations}) in terms of a local holomorphic coordinate on the surface. We also see, how these equations can be solved whenever a certain integrability condition is satisfied (see \cite{loftin2013cubic} and \cite{changping1990some}).\\ \\ Since we are interested in equivariant hyperbolic affine spheres, we can pick a parametrization $f\!:\Delta \to \R^3$, where $\Delta$ is a simply-connected domain in $\C$ biholomorphic to the open unit disk. Let $z=x+iy$ be a local conformal coordinate with respect to the Blaschke metric $h$, so that $h=e^{\psi}\vl\mathrm{d}z\vl^2$, where $\vl\mathrm{d}z\vl^2$ is defined as the symmetric product between $\mathrm{d}z$ and $\mathrm{d}\bar{z}$. After some algebraic manipulations (see \cite[\S 2.1.5]{loftin2013cubic} and \cite[\S 2.2]{rungi2021pseudo}), the structure equations (\ref{structurequations}) with $S=-\mathrm{Id}_{T\Sg}$ can be rewritten as a system of ODEs\begin{equation}\label{systemODEs}
    \begin{cases} \frac\partial{\partial z}\mathbf{F}=A\cdot\mathbf{F} \\ \frac\partial{\partial\bar z}\mathbf{F}=B\cdot\mathbf{F} \ ,
\end{cases}\end{equation} where $A$ an $B$ are $3\times 3$ matrices completely determined by the Blaschke metric and the Pick tensor (see Definition \ref{def:picktensor}), and $\mathbf{F}$ is the frame field of the hyperbolic affine sphere. \begin{proposition}[\cite{changping1990some}]
Given an initial condition for $\boldsymbol{F}$ at $z_0\in\Delta$, there exists a unique solution to the system (\ref{systemODEs}) as long as the following two equations are satisfied: 
\begin{equation}\label{integrabilityequations}\begin{cases}\frac{\partial^2}{\partial z\partial \bar z}\psi+\frac{1}{2}\vl Q\vl^2e^{-2\psi}-\frac{1}{2}e^\psi=0 \\
        \frac{\partial}{\partial\bar z}Q=0 \ ,
    \end{cases}\end{equation}where $Q$ is as smooth function on $\Delta$ completely determined by the Pick tensor of the immersion. In particular, if equations (\ref{integrabilityequations}) hold, the expression $q=Q\mathrm{d}z^3$ gives rise to a holomorphic cubic differential on $\Delta$. 
\end{proposition}
\begin{remark}\label{rem:osservazionedifferenzialecubicoriscalatowangequation}
Attention must be given to the factor $\frac{1}{2}$ that has been placed in front of the term $\vl Q\vl^2e^{-2\psi}$ in (\ref{integrabilityequations}) and that does not appear in \cite[Theorem 2.10]{rungi2021pseudo}. In fact, as we shall see shortly, through this rescaling of the cubic differential we will be able to rewrite the first equation of the above system in terms of a vortex-type equation defined on the surface.
    \end{remark}
%In order to understand when these integrability conditions are satisfied, we follow the argument presented in \cite[\S 3]{changping1990some}. Roughly speaking, the idea is to rewrite the first equation in (\ref{integrabilityequations}) in a global form on a Riemann surface and show that it admits a unique solution in each conformal class of metrics.\newline 
Now let $(\Sg, J)$ be a closed Riemann surface with genus $g\ge 2$. By the well-known Poincar\'e-Koebe Uniformization Theorem we can pick a Riemannian metric $g_0$ of constant curvature $k_0$ on $\Sg$ which is compatible with the initial complex structure $J$. Let $H^0(\Sg, K^3)$ be the holomorphic sections of the tri-canonical bundle over $(\Sg, J)$, namely the $\C$-vector space of holomorphic cubic differentials. It is easy to see, using the Riemann-Roch Theorem, that this space has complex dimension equal to $5g-5$. If $z=x+iy$ is a local holomorphic coordinate on $(\Sg, J)$, then we can define a norm on $H^0(\Sg, K^3)$, given by: $$\vl\vl q\vl\vl_{g_0}^2:=\vl Q\vl^2e^{-3\phi} \ ,$$where $q=Q\mathrm{d}z^3$ and $g_0=e^\phi\vl\mathrm{d}z\vl^2$ in this local coordinate. \begin{theorem}[\cite{changping1990some}]\label{semilinearellipticproposition}
Pick the metric $g_0$ so that its Gaussian curvature is equal to $-1$. Let $h=e^ug_0$ be a Riemannian metric in the same conformal class as $g_0$ and $q\in H^0(\Sg,K^3)$, then $\psi=u+\phi$ satisfies the first equation of (\ref{integrabilityequations}) if and only if the metric $h$ satisfies:\begin{equation}\label{vortex-type}
        K_h-\vl\vl q\vl\vl^2_h=-1 \ ,
    \end{equation}where $K_h$ is the Gaussian curvature of $h$ and $\vl\vl q\vl\vl^2_h=\vl\vl q\vl\vl_{g_0}^2e^{-3u}$.
    \end{theorem}
    \begin{lemma}
    In the setting of the previous theorem, the metric $h$ satisfies equation (\ref{vortex-type}) if and only if the function $u:\Sg\to\R$ satisfies the following semi-linear elliptic equation \begin{equation}\label{semilinearelliptic}
        \Delta_{g_0}u+2\vl\vl q\vl\vl^2_{g_0}e^{-2u}-2e^u+2=0
    \end{equation}
    \end{lemma}
    \begin{proof}
    This is an easy application of the formula for the curvature $K_h=e^{-u}(k_0-\frac{1}{2}\Delta_{g_0}u)$ under conformal change of metric $h=e^ug_0$. In fact, since $g_0$ can be chosen so that $k_0=-1$, we get $$K_h-\vl\vl q\vl\vl_h^2+1=-e^{-u}-\frac{1}{2}e^{-u}\Delta_{g_0}u-e^{-3u}\vl\vl q\vl\vl_{g_0}^2+1 \ .$$ Multiplying the right-hand side of the equation above by the factor $-2e^u$, we have the following equivalence $$-e^{-u}-\frac{1}{2}e^{-u}\Delta_{g_0}u-e^{-3u}\vl\vl q\vl\vl_{g_0}^2+1=0 \iff \Delta_{g_0}u+2\vl\vl q\vl\vl^2_{g_0}e^{-2u}-2e^u+2=0 \ .$$
    \end{proof}
The original approach used by Wang to study existence and uniqueness of the solution to (\ref{semilinearelliptic}) (and thus to (\ref{vortex-type})) was the theory of elliptic operators between Sobolev spaces (\cite[\S 4]{changping1990some}). About ten years later, Loftin simplified a lot the original argument, deducing the following stronger result: \begin{lemma}[\cite{loftin2001affine}]\label{loftinlemma}
Let $(M, \tilde{g})$ be a smooth compact Riemannian manifold and let $\tilde{\varphi}$ be a smooth non-negative function on $M$. Then, the equation \begin{equation}\label{loftinequation}
    \Delta_{\tilde g}u+\tilde\varphi(p)e^{-2u}-2e^u+2=0
\end{equation}has a unique smooth solution.
\end{lemma}
In the end, combining all the results together we get the following important consequence:
\begin{corollary}[\cite{changping1990some, loftin2001affine}]\label{cor:affineandcubic}
A hyperbolic affine sphere in $\R^3$ with center $0$ which admits a properly discontinuous action of a discrete group $\Gamma<\SL(3,\R)$, so that the quotient is a closed oriented surface $\Sg$ of genus $g\ge 2$, is completely determined by a conformal structure on $\Sg$ and a section $q\in H^0(\Sg,K^3)$. Moreover, all such hyperbolic affine spheres are determined in this way.
\end{corollary}
\subsection{Labourie and Loftin's parametrization}\label{sec:2.2}
The core of this section is to use the results presented in Sections \ref{sec:1.2} and Section \ref{sec:2.1} to describe the parameterization, found by Labourie and Loftin, of the space of convex $\RP$-structures on $\Sg$. It should be emphasized that the same result found in \cite{loftin2001affine} and \cite{Labourie_cubic} was proven using slightly different techniques, and in our case, we will follow Loftin's approach more. \\ \\ Let $\pi:\cubicg\to\Teichc$ be the holomorphic vector bundle of cubic differentials over Teichm\"uller space. The fibre over an equivalence class $[J]\in\Teichc$ is the $\C$-vector space of holomorphic sections of the tri-canonical bundle. Given any pair $([J],q)\in\cubicg$, we have an embedding $\widetilde\Sg\to\R^3$ as a (equivariant) hyperbolic affine sphere whose Pick tensor and Blaschke metric are completely determined by $[J]$ and $q$ (Corollary \ref{cor:affineandcubic}). In particular, by the argument in Section \ref{sec:1.2}, we get a family of convex $\RP$-structures on $\Sg$ in the same $\Diff_0(\Sg)$-orbit. Conversely, if we start with an equivalence class of convex $\RP$-structures, by Theorem \ref{thmchengyau} we get an (equivariant) embedding $\widetilde{\Sg}\to\R^3$ as a hyperbolic affine sphere, which is equivalent to a pair $([J],q)$ as above. In the end, the main result is the following:\begin{theorem}[\cite{loftin2001affine},\cite{Labourie_cubic}]\label{teoloftinlabourie}
Let $\Phi:\defg\to\cubicg$ be the map which associates to each equivalence class of convex $\RP$-structures the pair $([J],q)$ described above. Then, $\Phi$ is an homeomorphism.
\end{theorem} 
There is a pull-back action of MCG$(\Sg)$ on $\cubicg$ given by: $$[\psi]\cdot([J],q):=([\psi^*J], \psi^*q) \ .$$ It is well defined as it does not depend on the chosen representative in $[\psi]\in\textrm{MCG}(\Sg)$. Moreover, the pair $([\psi^*J],\psi^*q)$ still defines a point in $\cubicg$ as $\psi^*q$ is holomorphic with respect to $\psi^*J$ if and only if $q$ is $J$-holomorphic. Thanks to Theorem \ref{teo:choigoldman}, taking into account the action of MCG$(\Sg)$ on $\defg$ and $\hitc$, one gets the following remarkable consequence: \begin{corollary}[\cite{loftin2001affine},\cite{Labourie_cubic}]\label{complexloftinlabourie}
The space $\hitc$ carries a mapping class group invariant complex structure, denoted with $\i$.
\end{corollary}

\subsection{A conformal change of metric}\label{sec:2.3}
The next step now is to introduce a conformal change of metrics on the surface that allows us to find an equivalent description of $\haff$ which will be crucial for the symplectic reduction. In order to do this, we need to fix an area form $\rho$ on the surface. Then, using the so-called Moser's trick in symplectic geometry we will obtain a different model of $\defg$ as the quotient of an infinite dimensional space by $\Symp_0(\Sg,\rho)$. \newline First, we introduce the function that will allow us to make the conformal change of metric.
\begin{lemma}\label{lem:functionFef}
    There exists a unique smooth function $F:[0,+\infty)\to\R$ such that \begin{equation}\label{functionalequationF}ce^{-F(t)}-2te^{-3F(t)}+1=0,\qquad F(0)=\log\abs{c},\end{equation}where $c:=\frac{2\pi\chi(\Sg)}{\mathrm{Vol}(\Sg,\rho)}$ is a strictly negative constant depending only on the topology and the area of the surface. Moreover, if the function $f:[0,+\infty)\to(-\infty,0]$ is defined as \begin{equation}\label{definitionf}
        f(t):=-\Bigg(\int_0^tF'(s)s^{-\frac{1}{3}}\mathrm ds\Bigg)t^{\frac{1}{3}}
    \end{equation} then it is smooth and it satisfies the following properties:\begin{itemize}
        \item[(1)]$f(0)=0$;
        \item[(2)]$f'(t)<0$ for all $t>0$;
        \item[(3)]$\displaystyle\lim_{t\to+\infty}f(t)=-\infty \ .$
    \end{itemize}
\end{lemma}\begin{proof}
    The existence and uniqueness of the smooth function $F$ follows from a standard application of the implicit function theorem to $G(t,y):=ce^{-y}-2te^{-3y}+1$. In particular, $G\big(0,F(0)\big)=0$ implies that $F(0)=\log|c|$. Using the formulas for the derivative of the function $y=F(t)$ in terms of the derivatives of $G(t,y)$, we obtain that: $$F'(t)>0 \ \forall t\ge 0,\qquad F''(t)<0 \ \forall t\ge 0,\qquad \lim_{t\to+\infty}F(t)=+\infty \ .$$ From the definition it is clear that $f$ is smooth, $f(0)=0$ and it only attains non-positive values. The fundamental theorem of calculus implies that $$f'(t)=-\Big(F'(t)+\frac{t^{-\frac{2}{3}}}{3}\int_0^tF'(s)s^{-\frac{1}{3}}\mathrm ds\Big) \ ,$$hence $f'(t)<0$ for all $t>0$. The behavior of $f$ at infinity can be obtained by using the explicit expression $$F(t)=\ln\left(\frac{(4t)^{\frac{1}{3}}}{\sqrt[3]{1+\sqrt{1+\frac{\zeta}{t}}}+\sqrt[3]{1-\sqrt{1+\frac{\zeta}{t}}}}\right) \ ,t\neq 0$$that can be derived from the functional equation in the statement, which is a cubic equation in $e^{-F(t)}$, and where $\zeta=\frac{2}{27}|c|$.
\end{proof}
\begin{lemma}\label{lem:combinationoffunctionfandfprime}
    Let $f:[0,+\infty)\to(-\infty,0]$ and $F:[0,+\infty)\to\R$ be the functions defined above. Then, \begin{itemize}
        \item[(1)] $\displaystyle f'(t)=-F'(t)+\frac{f(t)}{3t},$ for all $t>0$;
        \item[(2)] $1-f(t)+3tf'(t)>0,$ for all $t\geq 0$;
        \item[(3)] $f'(t)$ is monotonically increasing for any $t>0$.
    \end{itemize}
\end{lemma}
\begin{proof}
    The first identity can be obtained by computing the derivative of the function $f(t)$. In fact, 
    \begin{equation} \label{eq:derivative_f}
    f'(t)=-\Big(F'(t)+\frac{t^{-\frac{2}{3}}}{3}\int_0^tF'(s)s^{-\frac{1}{3}}\mathrm ds\Big)=-F'(t)+\frac{f(t)}{3t}, \quad\forall t>0
    \end{equation}
    Regarding the second identity, we need to use the explicit expression of $F(t)$ found in the proof of Lemma \ref{lem:functionFef}. Hence, for any $t>0$, we have $$F(t)=\ln\bigg(\frac{(4t)^{\frac{1}{3}}}{g(t)}\bigg), \qquad g(t):=\sqrt[3]{1+\sqrt{1+\frac{\zeta}{t}}}+\sqrt[3]{1-\sqrt{1+\frac{\zeta}{t}}},\quad \zeta=\frac{2}{27}|c| \ .$$ This implies, $$F'(t)=\frac{1}{3t}-\frac{g'(t)}{g(t)} \ .$$ In the end, combining $(1)$ with the explicit expression for $F'(t)$, we get \begin{align*}1-f(t)+3tf'(t)&=1-f(t)+3t\Big(-F'(t)+\frac{f(t)}{3t}\Big) \\ &=1-f(t)+3t\Big(-\frac{1}{3t}+\frac{g'(t)}{g(t)}+\frac{f(t)}{3t}\Big) \\ &=3t\frac{g'(t)}{g(t)} .\end{align*}By using the classical theory of the study of a real function of one variable, we deduce that $g(t)$ is strictly positive and monotonically increasing for every $t>0$. Moreover, at $t=0$, we have $1-f(0)+3\cdot 0\cdot f'(0)=1>0$.\newline Using equation $(1)$ in the statement, we have \begin{equation}\label{secondderivativef}f''(t)=-F''(t)+\frac{1}{3t^2}\big(f'(t)t-f(t)\big) \ .\end{equation} Using the definition of $f(t)$, the new function $G(t):=f'(t)t^{\frac{2}{3}}-f(t)t^{-\frac{1}{3}}$ is equal to zero when $t=0$ and its derivative is given by \begin{align*}
        G'(t)&=\big(f'(t)t^{\frac{2}{3}}-f(t)t^{-\frac{1}{3}}\big)'\\
        &=\Big(-t^{\frac{2}{3}}F'(t)-\frac{2}{3}f(t)t^{-\frac{1}{3}}\Big)' \tag{Equation \eqref{eq:derivative_f}} \\ &=-F''(t)t^{\frac{2}{3}}-\frac{2}{3}F'(t)t^{-\frac{1}{3}}-\frac{2}{3}f'(t)t^{-\frac{1}{3}}+\frac{2}{9}f(t)t^{-\frac{4}{3}} \tag{Equation \eqref{eq:derivative_f}} \\ &=-F''(t)t^{\frac{2}{3}}>0, \quad\forall t>0 \ .
    \end{align*}This implies that $G(t)\ge 0$ for any $t\ge 0$, hence, by using (\ref{secondderivativef}), $f'(t)$ is monotonically increasing for any $t>0$.
\end{proof}
Recall that, in general, given a (0,3)-tensor $C$ and a Riemannian metric $g$ on $\Sg$, one can define $A:=g^{-1}C$ to be the associated (1,2)-tensor, namely a $1$-form with values in End$(T\Sg)$. Suppose also that $C$ is totally symmetric, then according to Theorem \ref{thm:picktensor} the tensor $C$ is the real part of a cubic differential if and only if the endomorphism part of $A$ is trace-less. Let us introduce the following space:  \begin{equation*}
    \haffzero:=\bigslant{\left\{ (g,C) \ \Bigg | \ \parbox{19em}{$g$ is a Riemannian metric on $\Sg$ \\ $C$ is the real part of a $g$-cubic differential \\ $\big(h:=e^{F\big(\frac{\vl\vl q\vl\vl_g^2}{2}\big)}g, A:=g^{-1}C\big)$ satisfy (\ref{Gausscodazzi})} \right\}}{\Diff_0(\Sg)}
\end{equation*}
Notice that the map sending the pair $(g,C)$ to $(h,A)$, where $h=e^{F\big(\frac{\vl\vl q\vl\vl_g^2}{2}\big)}g$ and $A=g^{-1}C$, induces a MCG$(\Sg)$-equivariant map from $\haffzero$ to $\haff$. There exists an inverse map constructed by sending the pair $(h,A)$ satisfying (\ref{Gausscodazzi}) to the pair $(g,C)$ where $g=e^{-F\big(\frac{\vl\vl q\vl\vl_g^2}{2}\big)}h$ and $C:=gA$. Since all the process is invariant by the action of $\Diff(\Sg)$, we get the following:
\begin{lemma}
    The correspondence described above induces a MCG$(\Sg)$-equivariant isomorphism between $\haffzero$ and $\haff$.
\end{lemma}
 %It suffices to construct the inverse map as all the process is invariant under the action of $\Diff(\Sg)$. Suppose $g$ is a Riemannian metric on $\Sg$ and define $h:=e^{F\big(\frac{\vl\vl q\vl\vl_g^2}{2}\big)}g$. Then, the pair $(h,A=h^{-1}C)$ satisfy (\ref{Gausscodazzi}) if and only if $K_h-\vl\vl q\vl\vl_h^2=-1$ and $\mathrm d^{\nabla^h}A=0$, where $C=\Ree(q)$. By using the Laplacian formula under conformal change of metrics $\Delta_h=e^{-F\big(\frac{\vl\vl q\vl\vl_g^2}{2}\big)}\Delta_g$ and the standard equality $K_h=-\frac{1}{2}\Delta_h\log h$, we obtain that the first equation in (\ref{Gausscodazzi}) is equivalent to $$e^{-F\big(\frac{\vl\vl q\vl\vl_g^2}{2}\big)}\bigg(K_g-\frac{1}{2}\Delta_gF\bigg(\frac{\vl\vl q\vl\vl_g^2}{2}\bigg)\bigg)-e^{-3F\big(\frac{\vl\vl q\vl\vl_g^2}{2}\big)}\vl\vl q\vl\vl_g^2=-1 \ .$$ Moreover, using the defining functional equation of $F(t)$ in Lemma \ref{lem:functionFef} we get $$e^{-F\big(\frac{\vl\vl q\vl\vl_g^2}{2}\big)}\bigg(K_g-\frac{1}{2}\Delta_gF\bigg(\frac{\vl\vl q\vl\vl_g^2}{2}\bigg)\bigg)=e^{-F\big(\frac{\vl\vl q\vl\vl_g^2}{2}\big)}c \ ,$$hence 
Let $\rho$ be the area form fixed at the beginning of the discussion, then for any (almost) complex structure $J$ the pairing $g_J:=\rho(\cdot,J\cdot)$ defines a Riemannian metric on the surface. Let us introduce the space \begin{equation*}\haffrhozero:=\bigslant{\left\{ (J,C) \ \Bigg | \ \parbox{19em}{$J$ is an (almost) complex structure on $\Sg$ \\ $C$ is the real part of a $J$-cubic differential \\ $\big(h:=e^{F\big(\frac{\vl\vl q\vl\vl^2_{g_J}}{2}\big)}g_J, A:=g_J^{-1}C\big)$ satisfy (\ref{Gausscodazzi})} \right\}}{\Symp_0(\Sg)}
\end{equation*} %as the set of pairs $(J,C)$ such that $h=e^{F\big(\frac{\vl\vl q\vl\vl_{g_J}^2}{2}\big)}g_J$ and $A=g_J^{-1}C$ satisfy (\ref{Gausscodazzi}), quotient out by $\Symp_0(\Sg,\rho)$ (i.e. the identity component of the group formed by those diffeomorphisms of the surface which preserve the area form $\rho$).
\begin{proposition}\label{prop:mosertrick}
The map sending the pair $(J,C)$ to $(h=e^{F\big(\frac{\vl\vl q\vl\vl_{g_J}^2}{2}\big)}g_J,A=g^{-1}_JC)$ induces a MCG$(\Sg)$-equivariant homeomorphism between $\haffrhozero$ and $\haff$.
\end{proposition}\begin{proof}
The proof is based on the so-called Moser's trick in symplectic geometry. Since this argument is standard in contexts similar to ours, we will only give an idea of how it is applied (for more details see \cite[\S 3.2.3]{hodge2005hyperkahler}). Moser's stability theorem (\cite[Theorem 3.2.4]{mcduff2017introduction}) claims that given a family of cohomologous symplectic forms $\omega_t$ on a closed manifold, there exists a family of diffeomorphisms $\phi_t$ such that $\phi_0=\text{Id}$ and $\phi_t^*\omega_t=\omega_0$. For a closed surface $\Sg$ of genus $g\ge 2$, given two area forms $\rho,\rho'$ of the same total area, one can apply Moser's stability theorem to the family $\rho_t:=(1-t)\rho+t\rho'$ and deduce that there exists $\phi\in\Diff_0(\Sg)$ such that $\phi^*\rho'=\rho$. In particular, for any $\Diff_0(\Sg)$-equivalence class $[g,C]$ in $\haff$, there exists a representative of the form $(g_J,C)$. Finally, if one has a family of diffeomorphisms $\psi_t$ with $\psi_0=\text{Id}$ and $\psi_1^*\rho=\rho$, by applying Moser's stability again to $\rho_t:=\psi_t^*\rho$ one can deform $\psi_t$ to a family of symplectomorphisms $\phi_t$ such that $\phi_0=\text{Id}$ and $\phi_1=\psi_1$. Combining it all together, it has been shown that $$\Symp_0(\Sg,\rho)=\Diff_0(\Sg)\cap\Symp(\Sg,\rho) \ .$$
\end{proof}

\section{The Weil-Petersson K\"ahler metric on Teichm\"uller space}
In this section we briefly recall the definition of the group of (Hamiltonian) symplectomorphisms of a closed oriented surface of genus $g\ge 2$ and their corresponding Lie algebras. Next, we briefly describe the construction of the Weil-Petersson K\"ahler metric on Teichm\"uller space using the theory of symplectic reduction, which inspires our construction for $\defg$.
\subsection{The Lie algebra of the group of (Hamiltonian) symplectomorphisms}\label{sec:3.1}
%The main goal of this section is to introduce the Lie algebras of $\Symp_0(\Sg,\rho)$ and $\Ham(\Sg, \rho)$. We briefly recall the standard definitions and the (unique) decomposition of a vector field $X\in\Gamma(T\Sg)$ as the sum of a symplectic and a harmonic part. Finally, we introduce two pairings that allow us to get a new description of the aforementioned Lie algebras. \\ \\
Let $\rho$ be a fixed area form on a closed surface $\Sg$ of genus $g\ge 2$. The group $\Symp_0(\Sg,\rho)$ is given by those diffeomorphisms $\phi$ isotopic to the identity and such that $\phi^*\rho=\rho$. Thanks to Cartan's magic formula: $$\liederivative_X\rho=\iota_X\mathrm d\rho+\mathrm d(\iota_X\rho)$$ and the fact that $\mathrm d\rho=0$, we obtain the following identification for the Lie algebra of $\Symp_0(\Sg,\rho)$: $$\Lsymp(\Sg,\rho)=\{X\in\Gamma(T\Sg) \ | \ \mathrm d(\iota_X\rho)=0\}\cong_\rho Z^1(\Sg) \ ,$$ where the last isomorphism is given by the identification of $\Gamma(T\Sg)$ with the space of $1$-forms $\Omega^1(\Sg)$, and $Z^1(\Sg)$ denotes the space of closed $1$-forms. A symplectomorphism $\phi$ is called \emph{Hamiltonian} if there is an isotopy $\phi_\bullet:[0,1]\to\Symp_0(\Sg,\rho)$, with $\phi_0=\text{Id}$ and $\phi_1=\phi$, and a smooth family of functions $H_t:\Sg\to\R$ such that $\iota_{X_t}\rho=\mathrm{d}H_t$, where $X_t$ is the infinitesimal generator of the symplectomorphism $\phi_t$. Let us denote by $\Ham(\Sg,\rho)$ the group of Hamiltonian symplectomorphisms, which is a normal subgroup of $\Symp(\Sg,\rho)$ (\cite[\S 3.1]{mcduff2017introduction}). The Lie algebra of $\Ham(\Sg,\rho)$ can be characterized as: $$\Lham(\Sg,\rho)=\{X\in\Gamma(T\Sg) \ | \ \iota_X\rho \ \text{is exact}\}\cong_\rho B^1(\Sg) \ ,$$where $B^1(\Sg)$ is the space of exact $1$-forms on $\Sg$. \begin{lemma}\label{lem:vectorfieldssurface}
    Let $\rho$ be a fixed area form and $J$ be a complex structure on $\Sg$, then any $X\in\Gamma(T\Sg)$ has a unique decomposition \begin{equation}
        X=V+W+JW' \ ,
    \end{equation}where $W,W'\in\Lham(\Sg,\rho)$ and $\mathrm d(\iota_V\rho)=\mathrm d(\iota_{JV}\rho)=0$
\end{lemma}\begin{proof}
    Let $\rho$ be a fixed area form on $\Sg$ and consider the induced isomorphism \begin{align*}
        \Gamma(T&\Sg)\overset{\cong}{\longrightarrow}\Omega^1(\Sg) \\ & X\mapsto \iota_X\rho \ .
    \end{align*}For any (almost) complex structure $J$ on $\Sg$ we get a Riemannian metric $g_J:=\rho(\cdot,J\cdot)$. Hodge theory for compact Riemannian surfaces implies the existence of a decomposition $$\Omega^1(\Sg)=\mathrm d\big(C^\infty(\Sg)\big)\oplus\mathrm d^*\big(\Omega^2(\Sg)\big)\oplus\mathcal H^1(\Sg) \ ,$$where $\mathrm d^*$ is the $L^2$-adjoint of the exterior differential and $\mathcal{H}^1(\Sg)=\{\alpha\in\Omega^1(\Sg) \ | \ \mathrm d\alpha=\mathrm d^*\alpha=0\}$ is the space of harmonic $1$-forms. In particular, for any $X\in\Gamma(T\Sg)$ we have a unique decomposition $\iota_X\rho=\mathrm d f+\mathrm d^*\omega+\alpha$, with $f\in C^\infty(\Sg), \omega\in\Omega^2(\Sg)$ and $\alpha\in\mathcal H^1(\Sg)$. Since each element of the decomposition is a $1$-form, there must exist three vector fields $V,W,\widetilde W$ such that $\mathrm df=\iota_W\rho, \mathrm d^*\omega=\iota_{\widetilde W}\rho$ and $\alpha=\iota_V\rho$, which implies that $X=V+W+\widetilde W$. Now notice that $\iota_W\rho$ is exact, hence $W\in\Lham(\Sg,\rho)$. Since $\alpha$ is harmonic we have $\mathrm d(\iota_V\rho)=\mathrm d^*(\iota_V\rho)=0$, but the term in between can be written as $\mathrm d(\iota_V\rho\circ J)$, which implies that $\mathrm d(\iota_{JV}\rho)=0$. Finally, in order to end the proof, we only need to show that $\widetilde W=JW'$ for some $W'\in\Lham(\Sg,\rho)$. This follows from the fact that $\iota_{\widetilde W}\rho=\mathrm d^*\omega=\ast_{g_J}\circ\mathrm d\circ\ast_{g_J}\omega=(\mathrm d\circ\ast_{g_J}\omega)\circ J$, where $\ast_{g_J}$ denotes the Hodge-star operator with respect to $g_J$. Since $\mathrm d\circ\ast_{g_J}\omega$ is an exact $1$-form, there exists a vector field $W'\in\Lham(\Sg,\rho)$ such that $\iota_{W'}\rho\circ J=\iota_{\widetilde W}\rho$, hence $\widetilde W=JW'$. \end{proof}
Because of the close connection with harmonic $1$-forms, the vector fields $V$ on $(\Sg,J)$ for which $\mathrm d(\iota_V\rho)=\mathrm d(\iota_{JV}\rho)=0$ will be called \emph{harmonic}. The space of harmonic vector fields on $\Sg$ will be denoted with $\mathfrak h_J$ and it is a Lie subalgebra of $\Lsymp(\Sg,\rho)$. Moreover, there is a splitting \begin{equation}\label{splittingsymplecticvectorfield}\Lsymp(\Sg,\rho)=\Lham(\Sg,\rho)\oplus\mathfrak h_J \end{equation} as infinite-dimensional vector spaces. Let $\psi\in\Symp_0(\Sg,\rho)$, then there exists a family of symplectomorphisms $\{\psi_t\}$ with $\psi_1=\psi$ and $\psi_0=\text{Id}$. Denote with $X_t$ the vector field which generates the isotopy, namely $\partial_t\psi_t=X_t\circ\psi_t$. Then one has a well-defined map called the \emph{Flux homomorphism}  \begin{equation}\begin{aligned}
        \text{Flux}:\Symp_0(&\Sg,\rho)\longrightarrow H^1_{\text{dR}}(\Sg,\R) \\ & \{\psi_t\}\mapsto \int_0^1[\iota_{X_t}\rho]\mathrm dt
    \end{aligned}\end{equation}\begin{lemma}[\cite{mcduff2017introduction}]\label{lem:fluxisomorphism}
        The Flux homomorphism is surjective and it induces an isomorphism \begin{equation}
            \bigslant{\Symp_0(\Sg,\rho)}{\Ham(\Sg,\rho)}\cong H^1_{\text{dR}}(\Sg,\R) \ .
        \end{equation}
    \end{lemma}
    We end the discussion in this section by introducing two non-degenerate pairings: \begin{equation}\begin{aligned}\label{symplecticpairing}
    \langle\cdot | \cdot\rangle_{\Lsymp}:\bigslant{\Omega^1(\Sg)}{B^1(\Sg)}&\times Z^1(\Sg)\longrightarrow \R \\ & ([\alpha],\beta) \ \ \longmapsto\int_{\Sg}\alpha\wedge\beta
\end{aligned}\end{equation}
\begin{align*}
    \langle\cdot | \cdot\rangle_{\Lham}:\bigslant{\Omega^1(\Sg)}{Z^1(\Sg)}&\times B^1(\Sg)\longrightarrow \R \\ & ([\alpha],\beta) \ \ \longmapsto\int_{\Sg}\alpha\wedge\beta \ .
\end{align*}Thanks to the identifications $Z^1(\Sg)\cong_\rho\Lsymp(\Sg,\rho), B^1(\Sg)\cong_\rho\Lham(\Sg,\rho)$ and the isomorphism 
 between $B^2(\Sg)$ and $\Omega^1(\Sg)/Z^1(\Sg)$ induce by the differential $\mathrm d$, we get $$\bigslant{\Omega^1(\Sg)}{B^1(\Sg)}\subset\Lsymp(\Sg,\rho)^*, \quad B^2(\Sg)\cong_{\mathrm d}\bigslant{\Omega^1(\Sg)}{Z^1(\Sg)}\subset\Lham(\Sg,\rho)^* \ .$$ \begin{remark}
     Observe that, since the above pairings are defined on infinite dimensional vector spaces $\mathcal V\times\mathcal W$, the notion of non-degeneracy we are referring to is the one that sometimes in the literature is called \emph{weakly non-degenerate}, namely the induced map $\mathcal V\to\mathcal W^*$ is injective.
 \end{remark}
 Using the standard property of the contraction operator $\iota$ with respect to the wedge product, for any vector field $V$ and any $1$-form $\alpha$ on $\Sg$, one has \begin{equation}\label{vectorfieldand1form}
     \iota_V\alpha\rho=\alpha\wedge\iota_V\rho \ .
 \end{equation}Moreover, if $V$ is Hamiltonian, namely $\iota_V\rho=\mathrm dH$ for some smooth function $H$, we get \begin{equation}
 \langle\alpha,\mathrm dH\rangle_{\Lham}=\int_{\Sg}\alpha\wedge\mathrm dH=\int_{\Sg}\alpha(V)\rho \ ,
 \end{equation}where $[\alpha]\in\Omega^1(\Sg)/Z^1(\Sg)$. 
\subsection{Teichm\"uller space as a symplectic reduction}\label{sec:3.2}%breve richiamo 
Before starting the study of $\defg$, let us briefly recall the construction that has been done for Teichm\"uller space. For a more detailed exposition see \cite{donaldson2003moment}, \cite{trautwein2018infinite} and \cite[\S 4.2]{mazzoli2021parahyperkahler}. \\ \\
Let $\rho_0:=\mathrm{d}x\wedge\mathrm{d}y$ be the standard area form on $\R^2$ and consider the space $$\almost:=\{J\in\End(\R^2) \ | \ J^2=-\mathds{1}, \ \rho_0(v,Jv)>0 \ \text{for some} \ v\in\R^2\setminus\{0\}\} \ .$$ Such a space is a $2$-dimensional manifold and it is easy to see that $\forall J\in\almost$, the tensor $g_{J}^0(\cdot,\cdot):=\rho_0(\cdot,J\cdot)$ is a scalar product on $\R^2$, with respect to which $J$ is an orthogonal endomorphism.
By differentiating the identity $J^2=-\mathds{1}$, it follows that $$T_J\almost=\{\dot{J}\in\End(\R^2) \ | \ J\dot J+\dot J J=0\} \ .$$ Equivalently, the space $T_J\almost$ can be identified with the trace-less and $g_J^0$-symmetric endomorphisms of $\R^2$. It carries a natural (almost) complex structure given by \begin{align*}\widehat{\mathcal{I}}:T_J&\almost\to T_J\almost \\ &\dot J\mapsto-J\dot J\end{align*}
Moreover, there is a natural scalar product defined on each tangent space $$\langle\dot J, \dot J'\rangle_J:=\frac 1{2}\tr(\dot J\dot J') \ ,$$for every $\dot J,\dot J'\in\almost$. The group $\SL(2,\R)$ acts by conjugation on $\almost$: for $P\in\SL(2,\R)$ and $J\in\almost$ one defines $P\cdot J:=PJP^{-1}$. The same formula can be used to define the $\SL(2,\R)$-action on $T_J\almost$ as well. \begin{lemma}
    The pairing given by $$\widehat\Omega_J(\dot J,\dot J):=-\frac{1}{2}\tr(\dot JJ\dot J')$$ defines a symplectic form on $\almost$, compatible with $\widehat{\mathcal{I}}$ and $\langle\cdot,\cdot\rangle_J$. In particular, the triple $(\langle\cdot,\cdot\rangle_J, \widehat{\mathcal{I}}, \widehat\Omega_J)$ is an $\SL(2,\R)$-invariant K\"ahler structure on $\almost$.
\end{lemma}
Now let $P$ be the $\SL(2,\R)$ frame bundle over $(\Sg,\rho)$, namely the fibre over a point $x\in\Sg$ is given by those linear maps $F:\R^2\to T_x\Sg$ such that $F^*\rho_x=\rho_0$. The frame bundle $P$ inherits the structure of an $\SL(2,\R)$-principal bundle with the following action: $B\cdot(x,F):=(x,F\circ B^{-1})$, for $B\in\SL(2,\R)$. Notice that any symplectomorphism $\psi$ of $(\Sg,\rho)$ naturally lifts to a diffeomorphism $\hat\psi$ of the total space $P$, by setting $$\hat\psi(x,F):=(\psi(x), \mathrm d_x\psi\circ F)\in P \ ,$$for every $(x,F)\in P$. Let us define the bundle $$P\big(\almost\big):=\bigslant{P\times\almost}{\SL(2,\R)},$$ where $\SL(2,\R)$ acts diagonally on the two factors. Notice that a section of $P\big(\almost\big)$ induces an almost complex structure $J$ on $\Sg$ which is compatible with $\rho$, i.e. $g_J(\cdot,\cdot):=\rho(\cdot,J\cdot)$ defines a Riemannian metric on $\Sg$. The induced almost complex structure on $\Sg$ is fibre-wise defined on $T_x\Sg$ as: $F_x\circ J_x\circ F_x^{-1}$. It is easy to see that the section $J$ is well-defined as if two pairs $\big((x,F), J_x\big)$ and $\big((x,F', J_x')\big)$ differ by the diagonal action of $\SL(2,\R)$, then they induce the same almost complex structure on $T_x\Sg$. According to the above construction, let us introduce the space of almost complex structures on $\Sg$: $$\almostg:=\Gamma\big(\Sg, P\big(\almost\big)\big) \ .$$ Given any $J\in\almostg$, a tangent vector $\dot J\in T_J\almostg$ identifies with a section of the pull-back vector bundle $J^*\big(T^{\text{vert}}P\big(\almost\big)\big)\to \Sg$, where $T^{\text{vert}}P\big(\almost\big)$ stands for the vertical sub-bundle of $TP\big(\almost\big)$ with respect to the projection $\pi:P\big(\almost\big)\to\Sg$. Equivalently, $\dot J$ is a section of End$(T\Sg)$ that satisfies $\dot J J+J\dot J=0$. One can formally define a symplectic form on the infinite-dimensional manifold $\almostg$ by integrating fibre-wise that on $\almost$. In other words, \begin{equation}
    \Omega_J(\dot J,\dot J'):=-\frac{1}{2}\int_{\Sg}\tr(\dot JJ\dot J')\rho \ .
\end{equation}
Furthermore, one obtains a complex structure $\mathcal I$ on $\almostg$, by applying point-wise $\widehat{\mathcal I}$ which is defined on $\almost$. At this point, the main goal is to explain that such a symplectic form and complex structure can be induced from the ambient $\almostg$ to Teichm\"uller space, using the theory of symplectic reduction. In the end, one succeeds in doing more, namely, $\Omega_J$ and $\mathcal I$ will be part of a K\"ahler metric on $\Teichc$ which turns out to be a multiple of the Weil-Petersson metric. Let us briefly recall the definition of Hamiltonian action and moment map.
\begin{definition}\label{def:momentmap}
Let $G$ be a Lie group, with Lie algebra $\mathfrak g$, acting on a symplectic manifold $(M,\omega)$ by symplectomorphisms. We say the action is \emph{Hamiltonian} if there exists a smooth function $\mu:M\to\mathfrak g^*$ satisfying the following properties: \begin{itemize}
    \item[(i)] The function $\mu$ is equivariant with respect to the $G$-action on $M$ and the co-adjoint action on $\mathfrak g^*$, namely \begin{equation}
        \mu_{g\cdot p}=\Ad^*(g)(\mu_p):=\mu_p\circ \Ad(g^{-1})\in\mathfrak g^* \ .
    \end{equation}
    \item[(ii)]Given $\xi\in\mathfrak g$, let $X_\xi$ be the vector field on $M$ generating the action of the $1$-parameter subgroup generated by $\xi$, i.e. $X_\xi=\frac{\mathrm d}{\mathrm dt}\text{exp}(t\xi)\cdot p |_{t=0}$. Then, for every $\xi\in\mathfrak g$ we have\begin{equation}
        \mathrm d\mu^{\xi}=\iota_{X_\xi}\omega=\omega(X_\xi,\cdot)
    \end{equation}where $\mu^\xi:M\to\R$ is the function $\mu^\xi(p):=\mu_p(\xi)$.
\end{itemize}A map $\mu$ satisfying the two properties above is called a \emph{moment map} for the Hamiltonian action.
\end{definition}
\begin{theorem}[\cite{donaldson2003moment},\cite{trautwein2018infinite}]\label{thm:momentmapteich}
   Let $c:=\frac{2\pi\chi(\Sg)}{\text{Vol}(\Sg,\rho)}$, then the function \begin{equation}\begin{aligned}
       \mu : \ & \almostg  \longrightarrow \Lham(\Sg,\rho)^* \\ & J 
 \ \longmapsto \ -2(K_J-c)\rho
   \end{aligned}\end{equation} is a moment map for the action of $\Ham(\Sg,\rho)$ on $(\almostg,\Omega)$, where $K_J\in\mathcal{C}^\infty(\Sg)$ is the Gaussian curvature of $g_J$.
\end{theorem} Observe that, by the Gauss-Bonnet Theorem, the $2$-form $-2(K_J-c)\rho$ is exact, according to the inclusion $B^2(\Sg)\subset\Lham(\Sg,\rho)^*$ introduced in Section \ref{sec:3.1}. Because of property (i) in Definition (\ref{def:momentmap}), the subset $\mu^{-1}(0)\subset\almostg$ is preserved by the action of $\Ham(\Sg,\rho)$. In particular, any variation $\dot J=\liederivative_X J$, with $X$ an Hamiltonian vector field and $J\in\almostg$, lies inside $\text{Ker}\mathrm d_J\mu$, which is identified with $T_J\mu^{-1}(0)$. In other words, the tangent space to the $\Ham(\Sg,\rho)$-orbit is entirely contained in the Kernel of $\mathrm d_J\mu$, for any $J\in\almostg$. Furthermore, by property (ii) in Definition \ref{def:momentmap}, for any $J\in\mu^{-1}(0)$ the space $\text{Ker}(\mathrm d_J\mu)$ is identified with the $\Omega_J$-orthogonal to $T_J\big(\Ham(\Sg,\rho)\cdot J\big)$, namely the tangent space to the $\Ham(\Sg,\rho)$-orbit. By using a geometric characterization of the elements in the $\Omega_J$-orthogonal to the orbit (\cite{donaldson2003moment}) one can induce a symplectic form on the quotient $\widetilde{\mathcal{T}}(\Sg):=\mu^{-1}(0)/\Ham(\Sg,\rho)$.  %\begin{lemma}[\cite{mazzoli2021parahyperkahler}]Let $J\in\mu^{-1}(0)\subset\almostg$, then an element $\dot J\in T_J\almostg$ lies in the $\Omega_J$-orthogonal to the $\Ham(\Sg, \rho)$-orbit if and only if $\dive_{g_J}\dot J$ is \NR{Spiegare qui che cos'e' la divergenza di un endomorfismo?} a closed $1$-form.\end{lemma}
However, the space $\widetilde{\mathcal T}(\Sg)$ is not isomorphic to Teichm\"uller space of the surface as it is a manifold of real dimension $6g-6+2g$. The further quotient of the space $\widetilde{\mathcal T}(\Sg)$ by the group \begin{equation*}H:=\bigslant{\Symp_0(\Sg,\rho)}{\Ham(\Sg,\rho)}\cong H^1_{\text{dR}}(\Sg,\R) \ , \tag{see Lemma \ref{lem:fluxisomorphism}}\end{equation*} can be identified with $\mathcal{T}^\mathfrak c(\Sg)$ (see \cite[\S 2.2]{donaldson2003moment}). The $H$-orbits in $\widetilde{\mathcal T}(\Sg)$ are complex and symplectic submanifolds (see \cite[\S 2.2]{donaldson2003moment} and \cite[Lemma 4.4.8]{trautwein2018infinite}), hence one gets an induced symplectic form on $\Teichc$ given by: $$\boldsymbol{\Omega}_{[J]}([\dot J], [\dot J'])=\Omega_J(\dot J_h,\dot J_h') \ ,$$ where the vectors $\dot J_h,\dot J_h'\in \text{Ker}(\mathrm d_J\mu)$ are lifts of $\dot J,\dot J'$ that are $\Omega_J$-orthogonal to the $\Symp_0(\Sg,\rho)$-orbit. %As in the case of the orbit of $\Ham(\Sg,\rho)$, we have a geometric characterization of such lifts. %\begin{lemma} [\cite{mazzoli2021parahyperkahler}]Let $J\in\mu^{-1}(0)\subset\almostg$, then an element $\dot J\in T_J\almostg$ lies in the $\Omega_J$-orthogonal to the $\Symp_0(\Sg, \rho)$-orbit if and only if $\dive_{g_J}\dot J$ is an exact $1$-form.\end{lemma} 
If one further re-normalizes the lift $\dot J_h$ so that it is $L^2$-orthogonal to the tangent space to the orbit, one recovers the classical description of the tangent space to Teichm\"uller space as the space of traceless Codazzi tensors (\cite{tromba2012teichmuller}). In that case, the formula of Weil-Petersson metric is also recovered by choosing an area form $\rho$ with $\text{Vol}(\Sg, \rho)=-2\pi\chi(\Sg)$, which means $c=-1$ in Theorem \ref{thm:momentmapteich}.
\begin{proposition}[{\cite[\S 2.1]{bonsante2012cyclic}}]\label{prop:weylpetersson}
    Let $\dot J,\dot J'$ be elements in $T_{[J]}\Teichc$, then the Weil-Petersson symplectic form and metric are respectively given by: \begin{equation}\big(\Omega_{\text{WP}}\big)_{[J]}(\dot J,\dot J')=-\frac{1}{8}\int_{\Sg}\tr(\dot J J\dot J')\mathrm dV,\qquad \big(G_{\text{WP}}\big)_{[J]}(\dot J,\dot J')=\frac{1}{8}\int_\Sg\tr(\dot J\dot J')\mathrm dV \ ,\end{equation}where $\mathrm dV$ is the area form of the unique hyperbolic metric with conformal structure $J$.
\end{proposition}
  \begin{remark}
      One of the key facts of this construction is that any choice of a supplement $V$ of $T_J\big(\Symp_0(\Sg,\rho)\cdot J\big)$ inside the Kernel of $\mathrm d\mu$ and $\Omega_J$-orthogonal to $T_J\big(\Symp_0(\Sg,\rho)\cdot J\big)$, provides a well-defined model for the tangent space to $\Teichc$, such that $(V, \Omega_J|_V)$ is symplectomorphic to $(T_{[J]}\Teichc, 4\Omega_{\text{WP}})$. %This is due to the fact that $\Omega_J(\mathcal I\cdot,\cdot)$ defines a Riemannian metric on the ambient space $\almostg$ (\cite[Theorem 1.3.2]{tromba2012teichmuller}).
  \end{remark}
\subsection{A formula for the differential of the curvature}\label{sec:3.3}
Here we briefly explain how to derive a formula for the first variation of the curvature $K_J$, using the theory introduced in the previous section. That expression will be useful later, when we show how equations (\ref{differentialequations}) come from a symplectic reduction argument (see Section \ref{sec:5.3}). We will follow closely the approach in \cite[\S 4.2]{mazzoli2021parahyperkahler}. \\ \\ Given any $B\in\mathrm{End}(T\Sg)$ and given a Riemannian metric $g$ on $\Sg$, we define the \emph{divergence} of the endomorphism $B$ as the $1$-form: \begin{equation}\label{divergenzadefinizione}
    (\dive_gB)(X):=\sum_ig\big((\nabla^g_{e_i}B)X,e_i\big) \ ,
\end{equation}where $(e_i)_i$ is a $g$-orthonormal frame of $T\Sg$, $\nabla^g$ is the Levi-Civita connection with respect to $g$ and $X$ is a smooth vector field on the surface. We will denote, likewise, the divergence of a vector field $V$ by $\dive_gV$. Moreover, whenever there is a fixed almost complex structure $J$ on the surface, the divergence will be taken with respect to $g_J=\rho(\cdot,J\cdot)\equiv g$. Given that $J$ is $\nabla^g$-parallel, namely $(\nabla^g_XJ)Y=0$ for all $X,Y\in\Gamma(T\Sg)$, one can deduce the following useful formula: \begin{equation}\label{divergenzaendomorfismoeJ}
    \dive_g(JB)=-(\dive_gB)\circ J
\end{equation} for any trace-less and $g$-symmetric endomorphism $B$. Another relation we will be using later is the following: \begin{equation}\label{divergenzaXformularho}
    \dive_g(X)=\mathrm d(\iota_X\rho)(v, Jv)
\end{equation}for any unit vector $v$. \begin{lemma}[\cite{mazzoli2021parahyperkahler}]\label{traceliederivative}
    Let $X$ be a vector field on $\Sg$, then $$\frac{1}{2}\tr\bigg(\dot J J\liederivative_XJ\bigg)=(\dive_g\dot J)(X)-\dive_g(JX) \ ,$$where $\big(\liederivative_XJ\big)(Y):=[X,JY]-J([X,Y])$ for any $Y\in\Gamma(T\Sg)$.
\end{lemma}By using the explicit expression of the moment map $\mu$ in Theorem \ref{thm:momentmapteich} and the above Lemma, one can prove the following result:
\begin{proposition}[\cite{mazzoli2021parahyperkahler}]\label{differentialofcurvature}
    Let $J$ be any almost complex structure on $\Sg$ and $\rho$ a fixed area form, then $$\mathrm d K_J(\dot J)\rho=\frac{1}{2}\mathrm d\big(\dive_g\dot J\big) \ ,$$ where $K_J$ is the Gaussian curvature of $g_J\equiv g$.
\end{proposition}
%\begin{proof}By the explicit expression of the moment map in Theorem \ref{thm:momentmapteich}, we have that $$\langle\mathrm d\mu(\dot J), X\rangle_{\Lham}=-2\int_\Sg \mathcal H\mathrm d K_J(\dot J)\rho$$ for any vector field $X$ with Hamiltonian function $\mathcal H$\end{proof}
\section{The pseudo-K\"ahler metric on the Hitchin component} 
This is the core section of the paper, in which the main result, namely Theorem \ref{thmA}, will be proved (Section \ref{sec:4.2}). In particular, after defining an infinite-dimensional space $\pickg$ starting from a general  construction of Donaldson (Section \ref{sec:4.1}), we look for a specific distribution inside the tangent to a set $\haffrhozerotilde$ sitting inside $\pickg$. Each vector space $W_{(J,A)}$ of this distribution is defined as the space of solutions to a system of PDEs. After studying in detail the above system of equations (Section \ref{sec:4.3}), we show that such a finite dimensional subspace $W_{(J,A)}$, can be identified (up to a further finite dimensional decomposition) with the $\SL(3,\R)$-Hitchin component (Section \ref{sec:4.4}). Finally, in Section \ref{sec:4.5} we generalize a previous result obtained in the case of the torus (\cite[\S 3.2]{rungi2021pseudo}).
\subsection{Construction of \texorpdfstring{$\pickg$}{D^3(S)}}\label{sec:4.1}
Here we use the notations introduced in Section \ref{sec:3.2}. Recall that $\almost$ is the space of (almost) complex structures on $\R^2$ compatible with the standard orientation. Now, consider the following space \begin{equation}
    \pick:=\{(J, C)\in \almost\times S_3(\R^2) \ | \ C(J\cdot,J\cdot,J\cdot)=-C(J\cdot,\cdot,\cdot)\} \ ,
\end{equation} where $S_3(\R^2)$ is the space of totally symmetric $(0,3)$-tensors. Any pair $(J,C)\in\pick$ defines a unique pair $(J,q)$, where $q$ is a $J$-holomorphic cubic differential on $\R^2$. In particular, there is a $\SL(2,\R)$-equivariant isomorphism between $\pick$ and the holomorphic vector bundle of cubic differentials $\cubic$ over Teichm\"uller space of the torus (\cite[Corollary 3.3]{rungi2021pseudo}).\begin{remark} Thanks to the relations between the tensor $C$ and the $(1,2)$-tensor $A:=g_J^{-1}C$ (see Theorem \ref{thm:picktensor}), it is possible to describe $\pick$ as the space formed by pairs $(J,A)$, where $J\in\almost$ and $A$ is a $1$-form with values in the bundle of trace-less and $g_J$-symmetric endomorphisms of $\R^2$ such that $A(J\cdot)=A(\cdot)J$ and $A(X)Y=A(Y)X$ for any $X,Y\in\Gamma(T\Sg)$. We will make repetitive use of this correspondence, using the tensor $C$ or the Pick form $A$ whichever is more convenient. \end{remark}
Because of the identification $\almost\cong\mathcal{T}(T^2)$, the space $\pick$ has the structure of a vector bundle over $\almost$, whose fiber at a point $J\in\almost$ is a two dimensional real vector space, denoted with $\pick_J$. Let $\{e_1,e_2\}$ be a $g_J$-orthonormal basis of $\R^2$ and $\{e_1^*,e_2^*\}$ be its dual, then any element $A$ in $\pick_J$ can be written as $A=A_1e_1^*+A_2e_2^*$, where $A_k:=A(e_k)$ for $k=1,2$. Hence, we can introduce a scalar product on $\pick_J$ by \begin{equation}\label{def:scalarproductpick}
    \langle A,B\rangle_J:=\tr(A\wedge\ast_JB)(e_{1},e_{2}) \ ,
\end{equation} 
or, more explicitly after expanding the wedge product and evaluating the $2$-form,
\begin{equation}\label{def:scalarproductpick2}
    \langle A,B\rangle_J=\tr(A_{1}B_{1}+A_{2}B_{2})e_{1}^{*}\wedge e_{2}^{*} (e_{1},e_{2})= \tr(A_{1}B_{1}+A_{2}B_{2}) \ .
\end{equation}
\begin{remark}
    Although the definition of the scalar product in (\ref{def:scalarproductpick}) differs from that found in \cite{rungi2021pseudo}, defined by integrating the $2$-form $\tr(A\wedge\ast_JB)$ over the torus, the two are equivalent. In fact, one can always assume that the flat metric $g_J$ is of unit area. Moreover, it is easy to check that relation (\ref{def:scalarproductpick2}) does not depend on the choice of basis.
\end{remark}
%We can give a more precise description of this scalar product: let $J_0\in\almost$ be the standard linear complex structure, namely \begin{equation*}J_0=\begin{pmatrix}0 & -1 \\ 1 & 0\end{pmatrix}\end{equation*} then, $g_{J_0}(\cdot, \cdot)$ is the Euclidean metric on $\R^2$ and the standard basis $\{\frac\partial{\partial x_0}, \frac\partial{\partial y_0}\}$ is such that $J_0(\frac\partial{\partial x_0})=\frac\partial{\partial y_0}, J_0(\frac\partial{\partial y_0})=-\frac\partial{\partial x_0}$. Similarly, for any other $J \in \almost$, we can find a $g_J$-orthonormal basis $\{e_1, e_2\}$ such that $Je_1=e_2, \ Je_2=-e_1$. Then, if $\{e_1^*, e_2^*\}$ is the dual basis and $A,B\in\pick_J$, we have $A=A_1e_1^*+A_2e_2^*, \ B=B_1e_1^*+B_2e_2^*$ where $A_i:=A(e_i)$ and $B_i:=B(e_i)$. Hence $\ast_J B=B_1e_2^*-B_2e_1^*$, since the endomorphisms $B_1,B_2$ are $g_J$-symmetric. In the end, the scalar product is given by \begin{equation}\label{scalarproductlocalcoordinates}\langle A, B\rangle_J=\int_{T^2}\tr(A_1B_1+A_2B_2)e_1^*\wedge e_2^* \ . \end{equation}
By exploiting the definition in (\ref{def:scalarproductpick2}), the following relation can be deduced: \begin{equation}\label{rel:cpxstructurescalarprod1}
    \langle AJ,BJ\rangle_J=\langle A,B\rangle_J \ , 
\end{equation}which is equivalent to \begin{equation}\label{rel:cpxstructurescalarprod2}
    \langle AJ,B\rangle_J=-\langle A,BJ\rangle_J \ .
\end{equation}
\begin{lemma}[{\cite[Lemma 3.6]{rungi2021pseudo}}]\label{lem:tangentpickform}
Let $(J,A)\in\pick$ and let $\dot A:=g_J^{-1}\dot C$ denotes the unique $(1,2)$-tensor such that $g_{J}(\dot A(X)Y,Z)=\dot C(X,Y,Z)$ for all $X,Y,Z \in \Gamma(T\R^{2})$. Then, an element $(\dot J,\dot A)$ belongs to $T_{(J,A)}\pick$ if and only if \begin{equation}\label{tracedotA}
    \dot J\in T_J\almost,\quad \tr\dot A(X)=\tr(JA(X)\dot J) \ \forall X\in\Gamma(T\R^2),\quad \dot A_0=\dot{\tilde A}_0+T(J,A,\dot J)
\end{equation}where $\dot A_0$ is the full trace-free part of $\dot A$, while the tensor $\dot{\tilde A}_0$ is the trace-free part of $\dot A$ independent of $\dot J$ and $T(J,A,\dot J)=A_1J\dot JEe_1^*+2A_2J\dot JEe_2^*$ in a local basis dual to a $g_J$-orthonormal frame $\{e_1,e_2\}$, with $E=\diag(1,-1)$.
\end{lemma} Notice that, the endomorphism part of the tangent vector $\dot A$ at the point $(J,A)$ is still $g_J$-symmetric, but its trace-part is completely determined by $\dot J$ due to (\ref{tracedotA}). In particular, there is a decomposition $$\dot A=\dot A_0+\frac{1}{2}\tr(JA\dot J)\mathds 1 \ ,$$ where $\mathds 1$ is the $2\times 2$ identity matrix and the second term in the RHS of the above equation is the trace-part of $\dot A$, which will be denoted by $\dot A_{\text{tr}}$. There is an $\SL(2,\R)$-action on $\pick$ given by \begin{equation}\label{SLaction}
    P\cdot (J,A):=(PJP^{-1}, PA(P^{-1}\cdot)P^{-1}) \ ,
\end{equation}where $P\in\SL(2,\R)$ and $A(P^{-1}\cdot)$ has to be interpreted as the action of $P^{-1}$ via pull-back on the $1$-form part of $A$. It is easy to show that the scalar product introduced in (\ref{def:scalarproductpick}) is invariant under the above action. Hence, by differentiating the $\SL(2,\R)$-action we get an induced linear isomorphism between $T_{(J,A)}\pick$ and $T_{P\cdot (J,A)}\pick$, which is given explicitly by
$$P\cdot (\dot J,\dot A)=(P\dot JP^{-1}, P\dot A(P^{-1}\cdot)P^{-1})\\$$ where $(\dot J,\dot A)\in T_{(J,A)}\pick$ and $P\in\SL(2,\R)$. We can define a similar scalar product on pairs $\dot A,\dot B$ as: $$\langle \dot A,\dot B\rangle_J:=\tr(\dot A_1\dot B_1+\dot A_2\dot B_2) \ .$$In the following we will denote with $\vl\vl\cdot\vl\vl_J=\vl\vl\cdot\vl\vl$ the norm induced by the scalar product $\langle\cdot,\cdot\rangle_J=\langle\cdot,\cdot\rangle$. In order to simplify the notation we define $\normpick:=\frac{1}{8}\vl\vl A\vl\vl_J^2$. \\ \\ Let $f:[0,+\infty)\to(-\infty,0]$ be a smooth function such that $f(0)=0, \ f'(t)<0$ for each $t>0$ and $\displaystyle\lim_{t\to+\infty}f(t)=-\infty$. Then, we define the following symmetric bi-linear form on $T_{(J,A)}\pick$ \begin{equation}\begin{aligned}\label{pseudoriemannianmetric}
   (\widehat\g_f)_{(J,A)}\big((\dot J,\dot A); (\dot J',\dot A')\big):=\big(1-f(\normpick)\big)\langle\dot J,\dot J'\rangle&+\frac{f'(\normpick)}{6}\langle \dot A_0, \dot A_0'\rangle \\ &-\frac{f'(\normpick)}{12}\langle \dot A_{\text{tr}}, \dot A_{\text{tr}}'\rangle 
\end{aligned}\end{equation}and the endomorphism $\widehat\i$ of $T_{(J,A)}\pick$ \begin{equation}\label{complexstructure}
    \widehat\i_{(J,A)}(\dot J, \dot A):=(-J\dot J, -\dot AJ-A\dot J)
\end{equation}where the products $\dot A J$ and $A\dot J$ have to be interpreted as a matrix multiplication. Matching these two objects together we get the following $2$-form: $$\widehat\ome_f(\cdot,\cdot)=\widehat\g_f(\cdot,\widehat\i\cdot)$$which is given by:\begin{equation}\begin{aligned}\label{symplecticform}(\widehat\ome_f)_{(J,A)}\big((\dot J, \dot A); (\dot J', \dot A')\big)=\big(f(\normpick)-1\big)\langle\dot J, J\dot J'\rangle&-\frac{f'(\normpick)}{12}\langle\dot A_{\tr}, \ast_J\dot A_{\tr}'\rangle \\ &-\frac{f'(\normpick)}{6}\langle\dot A_0, \dot A_0'J\rangle \ . 
\end{aligned}\end{equation}
\begin{theorem}[\cite{rungi2021pseudo}]
        The triple $(\widehat\g_f,\widehat\i,\widehat\ome_f)$ defines an $\SL(2,\R)$-invariant pseudo-K\"ahler metric on $\pick$.
\begin{remark} \label{rm:rescaling}
We want to point out the difference of the coefficients in front of the two terms with $f'$ in both the symplectic form and the pseudo-Riemannian metric, from our previous work (\cite{rungi2021pseudo}). This change is due to the fact that, in the case of genus $g\ge 2$, the cubic differential, hence the tensor $A$, is rescaled by a factor of $\sqrt 2$ (see also Remark \ref{rem:osservazionedifferenzialecubicoriscalatowangequation}), so that the moment map we introduce in Section \ref{sec:5.2} will be closely related to the Wang's equation for hyperbolic affine spheres in $\R^3$.
\end{remark}
 \end{theorem} Now let $\Sg$ be a closed smooth surface of genus $g\ge 2$. The next step is to perform a construction similar to that done for $\almostg$ in Section \ref{sec:3.2}, so as to obtain an infinite-dimensional space, associated with $\Sg$, and endowed with a (formal) pseudo-K\"ahler structure. Let $P$ be the $\SL(2,\R)$-frame bundle over $(\Sg,\rho)$ introduced in Section \ref{sec:3.2} and consider the bundle $$P\big(\pick\big):=\bigslant{P\times\pick}{\SL(2,\R)} \ ,$$ where $\SL(2,\R)$ acts diagonally on the two factors. The fibre of $P\big(\pick\big)$ over a point $x\in\Sg$ identifies with $D^3\big(\mathcal J(T_x\Sg)\big)$, namely the space of pairs $(J_x,A_x)$ where $J_x$ is an almost complex structure on $T_x\Sg$ compatible with $\rho_x$, and $A_x$ is an End$_0(T_x\Sg, (g_{J_x})_x)$-valued $1$-form such that $A_x(J_x\cdot)=A_x(\cdot)J_x$ and $A_x(X)Y=A_x(Y)X, \ \forall X,Y\in\Gamma(T_x\Sg)$. Since the pseudo-K\"ahler metric on $\pick$ is $\SL(2,\R)$-invariant, each fibre $D^3\big(\mathcal J(T_x\Sg)\big)$ is naturally endowed with a pseudo-K\"ahler structure, still denoted with $((\widehat\g_f)_x,\widehat\i_x,(\widehat\ome_f)_x)$, obtained by identifying $T_x\Sg$ with $\R^2$ using an area-preserving isomorphism $F_x:T_x\Sg\to\R^2$. The space of smooth sections $$\pickg:=\Gamma\big(\Sg,P(\pick)\big)$$ is identified with the set of pairs $(J,A)$, where $J$ is a complex structure on $\Sg$, and $A$ is an End$_0(T\Sg, g_J)$-valued $1$-form such that $A(J\cdot)=A(\cdot)J$ and $A(X)Y=A(Y)X, \ \forall X,Y\in\Gamma(T\Sg)$. Moreover, there is an identification between the tangent space to $\pickg$ at $(J,A)$ and the space of sections of the vector bundle $(J,A)^*\big(T^{\text{vert}}P\big(\pick\big)\big)\to\Sg$, where $T^{\text{vert}}P\big(\pick\big)$ stands for the vertical sub-bundle of $TP\big(\pick\big)$ with respect to the projection map $P\big(\pick\big)\to\Sg$. We can consider tangent vectors $(\dot J,\dot A)$ at $(J,A)$ as the data of (see Lemma \ref{lem:tangentpickform}):\begin{itemize}
     \item a section $\dot J$ of End$(T\Sg)$ such that $\dot JJ+J\dot J=0$, namely $\dot J$ is a $g_J$-symmetric and trace-less endomorphism of $T\Sg$; \item an End$(T\Sg, g_J)$-valued $1$-form $\dot A$ such that \begin{equation}\label{decompositiondotA}\dot A=\dot A_0+\frac{1}{2}\tr(JA\dot J)\mathds 1 \ ,\end{equation}where $\mathds 1$ is the $2\times 2$ identity matrix and $\dot A_0$ is the trace-less part of $\dot A$. In particular, the trace-part of $\dot A$ is uniquely determined by $\dot J$.
 \end{itemize} The infinite-dimensional space $\pickg$ inherits a (formal) family of pseudo-K\"ahler structures, where the symplectic form is defined as \begin{equation}
     (\ome_f)_{(J,A)}\big((\dot J,\dot A),(\dot J',\dot A')\big):=\int_{\Sg}\hat\ome_f\big((\dot J,\dot A),(\dot J',\dot A')\big)\rho
 \end{equation}and the pseudo-Riemannian metric is given by \begin{equation}
     (\g_f)_{(J,A)}\big((\dot J,\dot A),(\dot J',\dot A')\big):=\int_{\Sg}\hat\g_f\big((\dot J,\dot A),(\dot J',\dot A')\big)\rho \ ,
 \end{equation}where $\widehat\ome_f$ and $\widehat\g_f$ denote, respectively, the symplectic form and pseudo-Riemannian metric induced on the pull-back of the vertical sub-bundle inside $TP\big(\pick\big)$ as described above. Likewise we get a linear endomorphism $$\i_{(J,A)}:T_{(J,A)}\pickg\to T_{(J,A)}\pickg$$ obtained by applying pointwisely the endomorphism $\widehat\i$ to a smooth section $(\dot J,\dot A)$ (see \cite[\S 2]{koiso1990complex}). \begin{remark}\label{rem:identitypointwise}
It is important to point out that the definition of each element of the pseudo-K\"ahler structure on $\pickg$ is identical to that given in (\ref{pseudoriemannianmetric}), (\ref{complexstructure}) and (\ref{symplecticform}), the only change is that now the elements $J, A, g_J, \dot J, \dot A$ are all tensors and $f\big(\normpick\big)$ is a smooth function on $\Sg$. Because of this similarity, in the remainder of the discussion we will make use of some relations proved in the previous work \cite{rungi2021pseudo} and which we will recall when necessary. The general idea to keep in mind is that the identities for elements in $\pick$ can be interpreted as point-wise identities at the level of smooth sections inside $\pickg$. In both our setting and anti-de Sitter one the definition of the infinite-dimensional space follows the lines of a much more general construction given by Donaldson (\cite[\S 2.1]{donaldson2003moment}), and for this reason the same phenomenon described above happens in either situation (\cite[Remark 4.9]{mazzoli2021parahyperkahler}).
 \end{remark}
 
\subsection{Proof of Theorem A}\label{sec:4.2} The aim of this section is to summarize the strategy of the proof of Theorem \ref{thmA}. We will present preliminary results, already stated in Section \ref{sec:1.5} and proved later in the paper, which will allow us to give a quite immediate proof of the main theorem mentioned above. The same approach was used in \cite[\S 4.4]{mazzoli2021parahyperkahler} with the appropriate differences.\\ \\ Recall from Section \ref{sec:2.3} that $\haffrhozero$ is the quotient of the infinite-dimensional space \begin{equation*}\haffrhozerotilde:=\left\{ (J,C) \ \Bigg | \ \parbox{19em}{$J$ is an (almost) complex structure on $\Sg$ \\ $C$ is the real part of a $J$-cubic differential \\ $\big(h:=e^{F\big(\frac{\vl\vl q\vl\vl^2_{g_J}}{2}\big)}g_J, A:=g_J^{-1}C\big)$ satisfy (\ref{Gausscodazzi})} \right\}
\end{equation*} 
by the action of $\Symp_0(\Sg,\rho)$, where $F$ is the smooth function defined in Lemma \ref{lem:functionFef}. The main idea is to define an $\Ham(\Sg,\rho)$-invariant distribution $\{W_{(J,A)}\}_{(J,A)}$ inside $T\haffrhozerotilde$, whose integral manifold $\defgtilde$ is the finite-dimensional quotient $\haffrhozerotilde/\Ham(\Sg,\rho)$. Because of the very specific choice of the subspaces $W_{(J,A)}$, the further finite-dimensional quotient $\defgtilde/H$, with $H:=\Symp(\Sg,\rho)/\Ham(\Sg,\rho)$, is isomorphic to $\hitc$.  \begin{definition}\label{def:ellipticequations}
    Given $(J,A)\in\haffrhozerotilde$, define $W_{(J,A)}$ to be the subspace of $T_{(J,A)}\pickg$ formed by those elements $(\dot J,\dot A)$ satisfying the following system of equations: 
    \begin{equation*}
\begin{cases} 
\mathrm d\big(\divr\big((f-1)\dot J\big)+\mathrm d\dot f\circ J-\frac{f'}{6}\beta\big)=0  \\ \mathrm d\big(\divr\big((f-1)\dot J\big)\circ J+\mathrm d\dot f_0\circ J-\frac{f'}{6}\beta\circ J\big)=0 \\ \mathrm d^\nabla\dot A_0(\bullet,\bullet)-J(\divr\dot J\wedge A)(\bullet,\bullet)=0
\end{cases}
\end{equation*}where $\beta(\bullet):=\langle(\nabla_\bullet A)J, \dot A_0\rangle$ is a $1$-form, $\dot f_0=-\frac{f'}{4}\langle A, \dot A_0J \rangle$ is a smooth function on $\Sg$ and $f$ is the function given by (\ref{definitionf}). Moreover, all the expressions for $f,f'$ and $\dot f$ are evaluated at $\normpick=\frac{\vl\vl A\vl\vl^2_J}{8}$.
\end{definition}
\begin{remark}\label{rem:codazzilinearizzatainvariante}The third equation in the above system can be re-written as $\mathrm d^\nabla\dot A_0(\bullet,\bullet)J=(\dive_g\dot J\wedge A)(\bullet,\bullet)$, which is equivalent to the following:\begin{equation}\label{codazzilinearizzatainvariante}\mathrm d^\nabla\dot A_0(\bullet,J\bullet)=\dive_g\dot J(\bullet)A(J\bullet)-\dive_g\dot J(J\bullet)A(\bullet) \ .\end{equation} In fact, by $C^\infty(\Sg)$-linearity, it is sufficient to perform the computation on a pair $\{X,JX\}$, for $X\in\Gamma(T\Sg)$. Therefore, we have  \begin{align*}\mathrm d^\nabla\dot A_0(X,JX)J&=(\dive_g\dot J\wedge A)(X,JX) \\ &= (\dive_g\dot J)(X)\cdot A(JX)-(\dive_g\dot J)(JX)\cdot A(X)\end{align*} which is exactly the right-hand side of (\ref{codazzilinearizzatainvariante}) computed on $X$ (as an End$_0(T\Sg, g)$-valued $1$-form). This new form of the equation will be crucial to some key steps in our argument.\end{remark}
\begin{manualtheorem}E
    Let $(J,A)$ be a point in $\haffrhozerotilde$, then $$\mathrm{dim} W_{(J,A)}\ge 16g-16+2g \ .$$
\end{manualtheorem} The latter result will be a consequence of a detailed study of the system of equations defining $W_{(J,A)}$. The difficult part lies in computing the principal symbols of the matrix operator associated with the three equations in (\ref{differentialequations}). It is then possible to conclude, using standard results from the theory of elliptic operators on compact manifolds. \begin{manualtheorem}F
    For every element $(J,A)\in\haffrhozerotilde$, the vector space $W_{(J,A)}$ is contained inside $T_{(J,A)}\haffrhozerotilde$ and it is invariant by the complex structure $\i$. Moreover, the collection $\{W_{(J,A)}\}_{(J,A)}$ defines a $\Ham(\Sg,\rho)$-invariant distribution on $\haffrhozerotilde$ and the natural projection $\pi:\haffrhozerotilde\to\defgtilde$ induces a linear isomorphism $$\mathrm d_{(J,A)}\pi:W_{(J,A)}\longrightarrow T_{[J,A]}\defgtilde$$
\end{manualtheorem}
Combining together Theorem \ref{thmD} and Theorem \ref{thmE}, we observe that the integral manifold $\defgtilde$ has dimension equal to $16g-16+2g$ and, for this reason, cannot be isomorphic to the $\SL(3,\R)$-Hitchin component. In fact, it is necessary to perform an additional (finite-dimensional) quotient of $\defgtilde$ by the group $H:=\Symp_0(\Sg,\rho)/\Ham(\Sg,\rho)$ isomorphic to $H^1_{\text{dR}}(\Sg,\R)$ (see Lemma \ref{lem:fluxisomorphism}). \begin{manualtheorem}G 
The $H$-action on $\defgtilde$ is free and proper, with complex and symplectic $H$-orbits. Moreover, the pseudo-K\"ahler structure $(\g_f,\i,\ome_f)$ descend to the quotient which is identified with $\hitc$. Finally, the complex structure $\i$ induced on the $\SL(3,\R)$-Hitchin component coincides with the one found by Labourie and Loftin.\end{manualtheorem} \begin{remark}\label{rem:equationsHitchincomponent}
The tangent space to the integral manifold $\defgtilde$, i.e. the subspace $W_{(J,A)}$, decomposes as a direct sum $V_{(J,A)}\oplus S_{(J,A)}$, where $V_{(J,A)}$ is the tangent space to the Hitchin component and $S_{(J,A)}:=\{\big(\liederivative_XJ,g_J^{-1}\liederivative_XC\big) \ | \ X\in\Gamma(T\Sg), \ \mathrm d(\iota_X\rho)=\mathrm d(\iota_{JX}\rho)=0\}$, namely the tangent space to the harmonic orbit (see Section \ref{sec:3.1}). Using the definition of $W_{(J,A)}$ in terms of the system of equations (\ref{differentialequations}), we get a similar description of the tangent space to the Hitchin component. In particular, $V_{(J,A)}$ can be characterized as the subspace of $W_{(J,A)}$ (see Section \ref{sec:6.1}) defined by the following system: 
\begin{equation}\label{equationHitchincomponent}
\begin{cases}
\divr\big((f-1)\dot J\big)+\mathrm d\dot f\circ J-\frac{f'}{6}\beta=\mathrm d \gamma_1  \\ \divr\big((f-1)\dot J\big)\circ J+\mathrm d\dot f_0\circ J-\frac{f'}{6}\beta\circ J=\mathrm d\gamma_2 \\ \mathrm d^\nabla\dot A_0(\bullet,\bullet)-J(\divr\dot J\wedge A)(\bullet,\bullet)=0
\end{cases}
\end{equation}for some $\gamma_1,\gamma_2\in C^\infty(\Sg)$. In a more concise form: \begin{equation}
    V_{(J,A)}=\left\{(\dot J,\dot A)\in T_{(J,A)}\haffrhozerotilde \ \bigg| \ \parbox{15em}{ $\alpha_1+i\alpha_2$ is exact \\ $\mathrm d^\nabla\dot A_0(\bullet,\bullet)-J(\divr\dot J\wedge A)(\bullet,\bullet)=0$}\right\}
\end{equation}where $\alpha_1$ and $\alpha_2$ are the $1$-forms in (\ref{equationHitchincomponent}) defined by the LHS of the first two equations.
\end{remark}
At this point, we have all the ingredients to present a concise proof of the main result of the paper.  \begin{manualtheorem}A 
Let $\Sg$ be a closed oriented surface of genus $g\ge 2$. Then, there exists a neighborhood $\mathcal N_{\mathcal{F}(\Sg)}$ of the Fuchsian locus in $\hitc$, which admits a mapping class group invariant pseudo-K\"ahler metric $(\g_f,\i,\ome_f)$. Moreover, the Fuchsian locus embeds as a totally geodesic submanifold and the triple $(\g_f,\i,\ome_f)$ restricts to a (multiple of) the Weil-Petersson K\"ahler metric of Teichm\"uller space.
\end{manualtheorem}
\begin{proof}
    The tangent space $T_{[J,A]}\defgtilde$ can be identified with $W_{(J,A)}$ (Theorem \ref{thmE}), hence we can define a complex structure $\i$, and a pseudo-Riemannian metric $\g_f$ by restriction from the infinite-dimensional space $\pickg$. This definition does not depend on the representative in the $\Ham(\Sg,\rho)$-orbit by the invariance statement in Theorem \ref{thmE} and the $\Symp(\Sg,\rho)$-invariance of $\i$ and $\g_f$. It is immediate that $\i$ is still compatible with $\g_f$ and that the pairing $\g_f(\cdot,\i\cdot)$ coincides with the $2$-form $\ome_f$ restricted to $W_{(J,A)}$. % The space $\mathfrak h_J$ is the Lie algebra associated with the Lie group $H$, hence the further (finite-dimensional) quotient $\defgtilde/H$, endowed with the restricted complex structure, is biholomorphic to $(\hitc, \i)$ (Theorem \ref{thmF}). Given that the mapping class group of the surface is isomorphic to $\Symp(\Sg,\rho)/\Symp_0(\Sg,\rho)$ we get a well-defined triple $(\g_f, \i, \ome_f)$ on $\hitc$ invariant under the MCG$(\Sg)$-action.
    Moreover, Theorem \ref{thmF} allows us to induce the triple $(\g_f, \i, \ome_f)$ on the quotient $\defgtilde/H\cong\hitc$, in such a way that the induced complex structure is equivalent to the one found by Labourie and Loftin. Thanks to the $\Symp(\Sg,\rho)$-invariance of $\g_f$ and $\i$, it follows that the induced structure on $\hitc$ is invariant under the action of the mapping class group since it is isomorphic to $\Symp(\Sg,\rho)/\Symp_0(\Sg,\rho)$.\newline Notice that the Fuchsian locus $\mathcal{F}(\Sg)$ (see Section \ref{sec:1.1} and \ref{sec:1.2}) inside $\hitc\cong\haff$ corresponds to pairs $(J,A)$ with $A=0$. According to Remark \ref{rem:equationsHitchincomponent}, the tangent space to $\haff$ along the Fuchsian locus is isomorphic to $V_{(J,0)}$, and thus consists of pairs $(\dot J, \dot A)$ such that $\dive_g\dot J=0$ and $\mathrm d^\nabla\dot A_0(\cdot,\cdot)=0$. The pseudo-metric restricted to $V_{(J,0)}$ is equal to $$(\g_f)_{(J,0)}\big((\dot J,\dot A); (\dot J', \dot A')\big)=\int_\Sg\langle\dot J,\dot J'\rangle\rho+\int_\Sg\frac{f'}{6}\langle\dot A_0,\dot A_0'\rangle\rho$$ since the trace-part of $\dot A$ is equal to zero according to relation (\ref{decompositiondotA}). Notice that $\g_f$ on $V_{(J,0)}$ coincides with $4G_{\text{WP}}$ along horizontal directions ($\dot A=0$) and it is negative-definite along vertical directions ($\dot J=0$). Because of the explicit description of $\g_f$, this is sufficient to conclude that the pseudo-metric is non-degenerate on arbitrary directions inside $V_{(J,0)}$ as well. In particular, there must exist an open neighborhood $\neighborhood$ of $\mathcal F(\Sg)$ inside $\hitc$ in which $\g_f$ is non-degenerate.  \newline
    Finally, the Fuchsian locus is the set of fixed points of the circle action, that consists of isometries for $\g_f$ by Theorem \ref{thmC} (which is proved in Section \ref{sec:4.5}). Using a standard argument in (pseudo)-Riemannian geometry, this implies that the Fuchsian locus is totally geodesic.
\end{proof}
\begin{remark}
It is important to emphasize again that the triple $(\g_f, \i, \ome_f)$ is defined over the entire Hitchin component $\hitc$, but may be degenerate away from the Fuchsian locus. The main problem lies in the fact that the restriction of an indefinite metric on a subspace is not necessarily non-degenerate (as in the positive-definite case). Partial results have been obtained concerning the non-existence of degenerate vectors outside $\neighborhood$, which will be explained in detail in Section \ref{sec:6.2}. 

\end{remark}
\subsection{The system of equations}\label{sec:4.3}
This section is devoted to the study of the system of equations defined by (\ref{differentialequations}) and to the proof of Theorem \ref{thmD}. More precisely, in Lemma \ref{lem:noparallelsection} and Lemma \ref{lem:exteriorcovariantderivative} we study the induced connection on the endomorphism bundle and the associated exterior covariant derivative. Then, Lemma \ref{lem:ordertwoderivatives} allows us to compute the terms involving derivatives of order two of $(\dot J,\dot A)$ in the first and second equation appearing in the system defining $W_{(J,A)}$. Further on, we explain how the space $W_{(J,A)}$ can be seen as the kernel of a matrix of mixed-order smooth differential operators, which is proven to be elliptic (Lemma \ref{lem:simboloprincipaleellittico}). Finally, using the homotopy invariance of the Fredholm index for elliptic operators, we deduce the lower bound on the dimension of $W_{(J,A)}$.
\begin{lemma}\label{lem:noparallelsection}
    Let $\nabla$ be the Levi-Civita connection with respect to $g_J$, then the induced connection $$\widebar\nabla:\Omega^0(\Sg,\mathrm{End}_0(T\Sg, g_J))\longrightarrow\Omega^1(\Sg,\mathrm{End}_0(T\Sg, g_J))$$ does not admit any non-zero parallel section, where $\mathrm{End}_0(T\Sg, g_J)$ denotes the real vector bundle of $g_J$-symmetric and trace-less endomorphisms of $T\Sg$.
\end{lemma}\begin{proof}
    Let $B\in\Omega^0(\Sg,\mathrm{End}_0(T\Sg, g_J))$ such that $\widebar\nabla B=0$. Let $x_0\in\Sg$ be a fixed point and $x\in\Sg$ be arbitrary. Consider a path $\gamma:[0,1]\to\Sg$ with $\gamma(0)=x_0$ and $\gamma(1)=x$. Let $\{e_1,e_2\}$ be a basis of $T_{x_0}\Sg$ and denote with $\{e_1(t), e_2(t)\}$ the basis of $T_{\gamma(t)}\Sg$ obtained by parallel transport $\{e_1,e_2\}$ along the path $\gamma$. Then, if $b_{ij}(t)$ denotes the $(i,j)$-th entry of $B_{\gamma(t)}$ for $i,j=1,2$, we have $b_{ij}(t)=g_J\big(B_{\gamma(t)}(e_j(t)), e_i(t)\big)$. By differentiating the last identity with respect to the parameter $t$, we get: \begin{align*}
        \frac{\mathrm d}{\mathrm dt}b_{ij}(t)=g_J\big(\underbrace{\big(\widebar\nabla_{\dot\gamma}B\big)}_{\substack{=0}}\big(e_j(t)\big)+B_{\gamma(t)}\big(\nabla_{\dot\gamma}e_j\big), e_i(t)\big)+g_J\big(B_{\gamma(t)}\big(e_j(t)\big), \nabla_{\dot\gamma}e_i\big) \ .
    \end{align*} Since the basis $\{e_1(t), e_2(t)\}$ has been obtained by parallel transport, we have $\nabla_{\dot\gamma}e_j=0$ for any $j=1,2$. In particular, we deduce that each entry of $B$ is constant along $\gamma$, hence $B_x=B_{\gamma(1)}=B_{\gamma(0)}=B_{x_0}$. Since $x\in\Sg$ was arbitrary, it follows that the endomorphism $B$ is actually constant on the whole surface. At this point, it is enough to show that every section of $E:=\mathrm{End}_0(T\Sg, g_J)$ admits at least one zero to conclude the proof. Since the real rank of $E$ is equal to the real dimension of the surface, any section $B$ is nowhere zero if and only if the Euler class $e\big(E\big)$ is trivial in $H^2(\Sg, \R)$. In our case, it can be shown (see for example \cite[\S 2.4]{tromba2012teichmuller}) that $E$ is the realization of the holomorphic line bundle $K\otimes K$ defined on $(\Sg, J)$. In particular, $e\big(E\big)=c_1\big(K\otimes K\big)$, where $c_1$ denotes the first Chern class of a complex vector bundle. Therefore, $$\int_{\Sg}e\big(E\big)=\int_{\Sg}c_1\big(K\otimes K\big)=\text{deg}(K\otimes K)=2(2g-2)\neq 0 \ .$$ The last chain of equalities implies that $e\big(E\big)$ is not trivial in cohomology by Poincar\'e duality, and thus any such section $B$ admits at least one zero.
\end{proof}
    \begin{lemma}\label{lem:exteriorcovariantderivative}Let $\nabla$ be the Levi-Civita connection with respect to $g_J$, then the exterior covariant derivative $$\mathrm d^\nabla:\Omega^1(\Sg,\mathrm{End}_0(T\Sg, g_J))\longrightarrow\Omega^2(\Sg,\mathrm{End}_0(T\Sg, g_J))$$ is surjective and its kernel has real dimension equal to $10g-10$.
\end{lemma}
\begin{proof}
    Recall that for $A\in\Omega^1\big(\Sg,\mathrm{End}_0(T\Sg,g_J)\big)$ and for any $X,Y, Z\in\Gamma(T\Sg)$ we have $$\big(\mathrm d^{\nabla}A\big)(X,Y)Z=\big(\nabla_XA\big)(Y)Z-\big(\nabla_YA\big)(X)Z \ .$$ In particular, if we define the $(0,3)$-tensor $C(X,Y,Z):=g_J(A(X)Y, Z)$, then $A\in\text{Ker}(\mathrm d^\nabla)$ if and only if $C$ is the real part of a $g_J$-holomorphic cubic differential (Theorem \ref{thm:picktensor}). The space of holomorphic cubic differentials on $(\Sg, g_J)$ coincides with the space $H^0(\Sg,K^{\otimes^3})$ of holomorphic sections of the tri-canonical bundle, which is isomorphic (as a real vector space) to $\R^{10g-10}$ by an easy application of Riemann-Roch Theorem for curves. \newline Concerning the surjectivity of $\mathrm d^\nabla$, we will prove that its Co-kernel is trivial. Let us denote with $\ast_J$ the Hodge-star operator defined on differential forms with respect to $g_J$, which can be extended to an isomorphism $\ast_J:\Omega^k(\Sg,\mathrm{End}_0(T\Sg, g_J))\overset{\cong}{\longrightarrow}\Omega^{2-k}(\Sg,\mathrm{End}_0(T\Sg, g_J))$. Let $\big(\mathrm d^\nabla\big)^*$ be the formal adjoint of the exterior covariant derivative with respect to the $L^2$-inner product on $\Omega^2(\Sg,\mathrm{End}_0(T\Sg, g_J))$ induced by $\ast_J$ and integration over $\Sg$. A standard computation shows that $$\big(\mathrm d^\nabla\big)^*=-\ast_J\circ \ \mathrm d^\nabla\circ\ast_J:\Omega^2(\Sg,\mathrm{End}_0(T\Sg, g_J))\longrightarrow\Omega^1(\Sg,\mathrm{End}_0(T\Sg, g_J)) \ .$$ Since Range$(\mathrm{d}^\nabla)$ is a closed subspace of $\Omega^2(\Sg,\mathrm{End}_0(T\Sg, g_J))$, we get that $\text{CoKer}(\mathrm d^\nabla)=\text{Ker}\Big(\big(\mathrm d^\nabla\big)^*\Big)$. In particular, if $\alpha\in\Omega^2(\Sg,\mathrm{End}_0(T\Sg, g_J))$ then \begin{align*}&\big(\mathrm d^\nabla\big)^*\alpha=0 \iff \tag{$\ast_J$ is an isomorphism} \\ &\mathrm d^\nabla(\ast_J\alpha)=0 \iff \tag{$\mathrm d^\nabla\equiv\widebar\nabla$ on $\Omega^0(\Sg,\mathrm{End}_0(T\Sg, g_J))$}\\ &\widebar\nabla(\ast_J\alpha)=0 \ ,\end{align*} where $\widebar\nabla$ in the last equation is the induced connection on $\Omega^0(\Sg,\mathrm{End}_0(T\Sg, g_J))$. According to Lemma \ref{lem:noparallelsection}, the induced connection $\widebar\nabla$ does not admit any non-zero parallel section, hence $\ast_J\alpha=0$, which implies $\alpha=0$.
\end{proof}
\begin{lemma}\label{lem:ordertwoderivatives}Let $(J,A)\in\pickg$ and consider the following $2$-forms on the surface \begin{align*}&\eta_1:=\mathrm d\bigg(\dive_g\big((f-1)\dot J\big)+\mathrm d\dot f\circ J-\frac{f'}{6}\langle(\nabla_{\bullet}A)J,\dot A_0\rangle\bigg) \ , \\ &\eta_2:=\mathrm d\bigg(\dive_g\big((f-1)\dot J\big)\circ J+\mathrm d\dot f_0\circ J+\frac{f'}{6}\langle\nabla_{\bullet}A,\dot A_0\rangle\bigg)\end{align*}where $(\dot J,\dot A_0)\in T_{(J,A)}\pickg$, the function $f$ is the one defined by (\ref{definitionf}), $\dot f_0=-\frac{f'}{4}\langle A,\dot A_0J\rangle$, and $f, f', \dot f, \dot f_0$ are computed in $\normpick=\frac{\vl\vl A\vl\vl_J^2}{8}$. Then, the part involving second order derivatives of $(\dot J,\dot A_0)$ in $\eta_1$ and $\eta_2$ is, respectively\begin{equation*} (f-1)\mathrm d\big(\dive_g\dot J\big)+\frac{f'}{4}\mathrm d\langle A,\nabla_{J\bullet}\dot A_0\rangle,\qquad (f-1)\mathrm d\big(\dive_g\dot J\circ J\big)-\frac{f'}{4}\mathrm d\langle A,\big(\nabla_{J\bullet}\dot A_0\big)J\rangle \ . \end{equation*}\end{lemma}\begin{proof}  By using (\ref{divergenzadefinizione}), we get the following equation involving the divergence of a smooth section of End$_0(T\Sg, g_J)$ multiplied by a smooth function $\varphi$ \begin{equation}\dive_g\big(\varphi\dot J\big)(X)=\mathrm d\varphi(\dot JX)+\varphi\big(\dive_g\dot J\big)(X),\quad\forall X\in\Gamma(T\Sg) \ .\end{equation}Therefore, \begin{align*} \mathrm d\bigg(\dive_g\big((f-1)\dot J\big)\bigg)=\mathrm d\big(\mathrm df\circ\dot J\big)+\mathrm df\wedge\dive_g\dot J+(f-1)\mathrm d\big(\dive_g\dot J\big) \ ,\end{align*} and it is clear that $(f-1)\mathrm d\big(\dive_g\dot J\big)$ is the only part involving second order derivatives of $\dot J$ in the expression above. Regarding the other two terms in $\eta_1$, let us first define $\tau_1:=\mathrm d\dot f\circ J$ and $\tau_2:=-\frac{f'}{6}\beta$, where $\beta=\langle(\nabla_\bullet A)J,\dot A_0\rangle$. \vspace{0.3cm}\newline \underline{\emph{The differential of $\tau_1$}}\vspace{0.3cm}\newline  Notice that the first order variation of $f\big(\normpick\big)$ is \begin{align*}\dot f&=\frac{f'}{8}(\langle\dot A, A\rangle +\langle A,\dot A\rangle) \tag{see Lemma 3.22 in \cite{rungi2021pseudo}} \\ &=\frac{f'}{4}\langle A,\dot A_0\rangle \ .\end{align*} Therefore, \begin{align*} \mathrm d\tau_1&=\mathrm d\Big(\mathrm d\big(\frac{f'}{4}\langle A,\dot A_0\rangle\big)\circ J\Big) \\ &=\frac{1}{4}\mathrm d\bigg(\langle A,\dot A_0\rangle\mathrm d f'\circ J+f'\mathrm d\langle A,\dot A_0\rangle\circ J\bigg) \\ &=\frac{1}{4}\mathrm d\bigg(\frac{f''}{4}\langle A,\nabla_{J\bullet}A\rangle\langle A,\dot A_0\rangle+f'\big(\langle\nabla_{J\bullet}A,\dot A_0\rangle+\langle A,\nabla_{J\bullet}\dot A_0\rangle\big)\bigg) \ , \end{align*}where in the last step we used that $\mathrm df'=\frac{f''}{4}\langle A,\nabla_{\bullet}A\rangle$. The only interesting part, for our purpose, is the term containing $f'\langle A,\nabla_{J\bullet}\dot A_0\rangle$. In particular, $$\frac{1}{4}\mathrm d \bigg(f'\langle A,\nabla_{J\bullet}\dot A_0\rangle\bigg)=\frac{f''}{16}\langle A,\nabla_{\bullet}A\rangle\wedge\langle A,\nabla_{J\bullet}\dot A_0\rangle+\frac{f'}{4}\mathrm d\langle A,\nabla_{J\bullet}\dot A_0\rangle \ .$$ Again, the only part involving second order derivatives is the term $\frac{f'}{4}\mathrm d\langle A,\nabla_{J\bullet}\dot A_0\rangle$. %By using Lemma \ref{lem:dnablaformula} we have $$\langle A,\nabla_{J\bullet}\dot A_0\rangle=\underbrace{\langle A,\nabla_\bullet\dot A_0 J\rangle}_{\substack{(a)}}-\underbrace{\langle A,\big(\mathrm d^\nabla\dot A_0\big)(\bullet,J\bullet)\rangle}_{\substack{(b)}}-\underbrace{\langle A,\mathrm d^\nabla\dot A_0(\bullet,\bullet)J\rangle}_{\substack{(c)}} \ .$$ Regarding the differential of term $(a)$ we get $$\mathrm d\langle A,\nabla_\bullet\dot A_0 J\rangle=\langle\nabla_\bullet A, \nabla_\bullet\dot A_0 J\rangle+\langle A,\nabla_\bullet\nabla_\bullet\dot A_0 J\rangle \ ,$$ and applying Lemma ?? to the second term of the last equality we see that there is no second derivative of $\dot A_0$ involved. Now we study the differential of term $(b)$, and with a similar approach we also derive the differential for term $(c)$. Before proceeding with the computations, we observe that if the pair $(\dot J,\dot A)$ is inside $T_{(J,A)}\haffrhozerotilde$, then the equation $\mathrm d^\nabla\dot A_0(\bullet,J\bullet)=\dive_g\dot J(\bullet)A(J\bullet)-\dive_g\dot J(J\bullet)A(\bullet)$ is satisfied (see Remark \ref{rem:codazzilinearizzatainvariante}). In particular, for any fixed $\dot J$, we have $\dot A_0=\dot A^{\dot J}_0+g^{-1}\Ree(q)$ (see Remark \ref{rem:solutionthirdequation}), where $\dot A^{\dot J}_0$ is a solution of (\ref{codazzilinearizzatainvariante}) and $\mathrm d^\nabla\Big(g^{-1}\Ree(q)\Big)=0$ (see Theorem \ref{thm:picktensor}). Therefore, we get the following chain of equality regarding the term $(b)$ \begin{align*}-\langle A,\big(\mathrm d^\nabla\dot A_0\big)(\bullet,J\bullet)\rangle&=-\langle A,\big(\mathrm d^\nabla\dot A^{\dot J}_0\big)(\bullet,J\bullet)\rangle \\ &=-\langle A,\dive_g\dot J(\bullet)A(J\bullet)\rangle+\langle A,\dive_g\dot J(J\bullet)A(\bullet)\rangle \\ &=-\dive_g\dot J(\bullet)\langle A,A(J\bullet)\rangle+\dive_g\dot J(J\bullet)\langle A,A(\bullet)\rangle \\ &=\dive_g\dot J(J\bullet)\vl\vl A\vl\vl^2_J \ ,\end{align*} where in the last step we used relation (\ref{rel:cpxstructurescalarprod2}). Thus, the part involving second order derivatives of $(\dot J,\dot A)$ for term $(b)$ is exactly $\mathrm d\Big(\dive_g\dot J(J\cdot)\Big)\vl\vl A\vl\vl_J^2$. Performing a similar computation for the term $(c)$, we obtain \begin{align*} -\langle A,\big(\mathrm d^\nabla\dot A_0\big)(\bullet,\bullet)J\rangle&=\langle A,\big(\mathrm d^\nabla\dot A^{\dot J}_0\big)(\bullet,J^2\bullet)J\rangle \\ &=-\langle A,\dive_g\dot J(J\bullet)A(J\bullet)J\rangle+\langle A,\dive_g\dot J(\bullet)A(J^2\bullet)J\rangle \\ &=-\dive_g\dot J(J\bullet)\langle A,A(J^2\bullet)\rangle-\dive_g\dot J(\bullet)\langle A,A(J\bullet)\rangle \\ &=\dive_g\dot J(J\bullet)\vl\vl A\vl\vl_J^2 \ , \end{align*}where again we used relation (\ref{rel:cpxstructurescalarprod2}) and $A(J\cdot)=A(\cdot)J$. Also in this case, the term we are interested in is given by $\mathrm d\Big(\dive_g\dot J(J\cdot)\Big)\vl\vl A\vl\vl_J^2$. Finally, summing up the part containing second order derivatives of $(\dot J,\dot A)$ in the differential of term $(a), (b)$ and $(c)$ multiplied by the factor $\frac{f'}{4}$,  we get $$\frac{f'}{2}\vl\vl A\vl\vl_J^2\big(\dive_g\dot J\big)(J\cdot) \ .$$ 
\vspace{0.3cm}\newline \underline{\emph{The differential of $\tau_2$}}\vspace{0.3cm}\newline To conclude the proof it must be proven that $\mathrm d\tau_2$ does not involve second derivatives of $\dot J$ or $\dot A_0$. In fact, by carrying out calculations similar to those made above \begin{align*}\mathrm d\gamma_2&=\mathrm d\bigg(-\frac{f'}{6}\langle(\nabla_\bullet A)J,\dot A_0\rangle\bigg) \\ &=-\frac{1}{6}\bigg(\frac{f''}{4}\langle A,\nabla_\bullet A\rangle\langle(\nabla_\bullet A)J, \dot A_0\rangle+f'\big(\langle(\nabla_\bullet\nabla_\bullet A)J, \dot A_0\rangle+\langle(\nabla_\bullet A)J, \nabla_\bullet\dot A_0\rangle\big)\bigg) \ . \end{align*} Because of the very similar expression of the 2-forms $\eta_1, \eta_2$ it is easy to see, by going over the calculations already done, that the part involving second derivatives of $(\dot J,\dot A_0)$ in $\eta_2$ is exactly $$(f-1)\mathrm d\big(\dive_g\dot J\circ J\big)-\frac{f'}{4}\mathrm d\langle A,\big(\nabla_{J\bullet}\dot A_0\big)J\rangle \ .$$\end{proof}

The next step is to write down in coordinates the expressions found in Lemma \ref{lem:ordertwoderivatives}, so that, later, we will be able to explicitly deduce the principal symbol of the matrix of operators associated with the PDEs defining the subspace $W_{(J,A)}$. In order to do this, we need to recall the construction in coordinates for $\pick$ (see \cite[\S 3.1]{rungi2021pseudo}), and then use the particular definition of $\pickg$ to infer that the same can be done, point-wise, in the genus $g\ge 2$ case (see Remark \ref{rem:identitypointwise}). Let $\rho_0:=\mathrm dx_0\wedge\mathrm dy_0$ be the standard area form on $\R^2$ and $g_J^0:=\rho_0(\cdot,J\cdot)$ be the associated scalar product, for some $J\in\almost$. Recall from Section \ref{sec:4.1} that $\pick$ is $\SL(2,\R)$-equivariantly isomorphic to the holomorphic vector bundle of cubic differentials over Teichm\"uller space of the torus, denoted with $\cubic$. The latter can be identified with $\Hyp\times\C$, where $\Hyp$ is a copy of $\mathcal T(T^2)$ and $\C$ is isomorphic to the fibre over an oriented (almost) complex structure $J:\R^2\to\R^2$. Let $z=x+iy$ and $w=u+iv$ be the complex coordinates on $\Hyp$ and $\C$, respectively. Then, we have the following correspondence $$\Hyp\times\C\ni(z,w)\longmapsto\big(j(z),C_{(z,w)}\big)\in\pick$$where $C_{(z,w)}=\Ree(q_{(z,w)})$ with $q_{(z,w)}=\widebar w(\dx_0-\widebar z\dy_0)^3$ (see Corollary \cite[Lemma 3.18]{rungi2021pseudo} and \cite[Lemma 5.2.1]{trautwein2018infinite}). Because of the $\SL(2,\R)$-invariance, one can compute the pair $(j(z), C_{(z,w)})$ at points $(i,w)\equiv (0,1,u,v)$, for some $w\in\C$. Using the relation $A=(g_J^0)^{-1}C$, one can deduce:
\begin{equation*}%\label{pickformcoordinate}
   \dot J=\mathrm{d}_ij(\dot x,\dot y)=\begin{pmatrix}
\dot x & -\dot y \\ -\dot y & -\dot x
\end{pmatrix}, \qquad A_{(i,w)}=\begin{pmatrix}
    u & v \\ v & -u
    \end{pmatrix}\dx_{0}+\begin{pmatrix}
    v & -u \\ -u & -v
    \end{pmatrix}\dy_{0} \ .
\end{equation*}
\begin{equation*}%\label{dotpickformcoordinates}
     (\dot A_0)_{(i,w)}=\begin{pmatrix}
    \dot u+u\dot y+v\dot x & -u\dot x+\dot v+v\dot y \\ -u\dot x+\dot v+v\dot y & -\dot u-u\dot y-v\dot x
    \end{pmatrix}\dx_{0}+\begin{pmatrix}
    \dot v+2(v\dot y-u\dot x) & -\dot u-2(u\dot y+v\dot x) \\ -\dot u-2(u\dot y+v\dot x) & -\dot v+2(u\dot x-v\dot y)
    \end{pmatrix}\dy_{0}
\end{equation*}\begin{equation*}
    (\dot A_{\tr})_{(i,w)}=\begin{pmatrix}
    -u\dot y-v\dot x & 0 \\ 0 & -u\dot y-v\dot x
    \end{pmatrix}\dx_{0}+\begin{pmatrix}
    u\dot x-v\dot y & 0 \\ 0 & u\dot x-v\dot y
    \end{pmatrix}\dy_{0} \ . 
\end{equation*}
Now, let us define the following matrix of smooth differential operators \begin{equation}\begin{aligned}\label{matrixdifferentialoperators}
\Lambda:T_{(J,A)}&\pickg\longrightarrow\Omega^2_\C(\Sg)\oplus\Omega^2\big(\Sg,\mathrm{End}_0(T\Sg,g_J)\big) \\ &(\dot J,\dot A_0)\longmapsto\big((L_1+iL_2)(\dot J,\dot A_0), S(\dot J,\dot A_0)\big)\end{aligned}\end{equation} where \begin{align*}&L_1(\dot J,\dot A_0):=\mathrm d\bigg(\dive_g\big((f-1)\dot J\big)+\mathrm d\dot f\circ J-\frac{f'}{6}\langle(\nabla_{\bullet}A)J,\dot A_0\rangle\bigg)\in\Omega^2(\Sg) \ , \\ &L_2(\dot J,\dot A_0):=\mathrm d\bigg(\dive_g\big((f-1)\dot J\big)\circ J+\mathrm d\dot f_0\circ J+\frac{f'}{6}\langle\nabla_{\bullet}A,\dot A_0\rangle\bigg)\in\Omega^2(\Sg) \ , \\ &S(\dot J,\dot A_0)=\mathrm d^\nabla\dot A_0(\cdot,\cdot)-J(\divr\dot J\wedge A)(\cdot,\cdot)\in\Omega^2\big(\Sg,\mathrm{End}_0(T\Sg,g_J)\big) \ .\end{align*}
%Then, for any point $p$ on the surface, since $\Lambda$ does not act on the trace-part of $\dot A$, it can be seen as a $4\times 4$ matrix of smooth differential operators written in coordinates $(\dot x,\dot y,\dot u,\dot v)$.
It is possible to define the principal symbol of a matrix of mixed-order differential operators as the matrix obtained by taking the principal symbols of each differential operator. The corresponding system of PDEs is called \emph{elliptic}, if the symbol matrix has non-zero determinant (see \cite{agmon1964estimates} and \cite{grubb1977boundary} for more details). %in our case, for any $p\in\Sg$ and for any $0\neq\xi\in T^*\Sg$ we can consider the following $4\times 4$\footnote{Anything in $\Lambda$ involving the tensors $(\dot J,\dot A_0)$ can be written using the coordinates $(\dot x,\dot y,\dot u,\dot v)$}block-matrix \begin{equation}\sigma\big(\Lambda\big)_p(\xi)=\begin{pmatrix}\Theta & \Xi \\ \Gamma & \Delta\end{pmatrix} \ ,\end{equation}and deduce that $\Lambda$ is elliptic, if and only if $\mathrm{det}\Big(\sigma\big(\Lambda\big)_p(\xi)\Big)\neq 0$. 
\begin{lemma}\label{lem:simboloprincipaleellittico}
 Let $(J,A)$ be an arbitrary point in $\pickg$. Then, for any $p\in\Sg$ and for any $0\neq\xi\in T^*\Sg$, the symbol matrix $\sigma\big(\Lambda\big)_p(\xi)$ has non-zero determinant.
\end{lemma}
\begin{proof}
Let $\{e_1,e_2\}$ be a $g_J$-orthonormal basis and let $\{e_1^*,e_2^*\}$ be the dual basis, so that $\xi=\xi_1e_1^*+\xi_2e_2^*$. We first note that $\sigma(\Lambda)_p(\xi)$ is a $4\times4$ matrix as any term in $\Lambda$, involving the tensors $(\dot J,\dot A_0)$, can be written in the coordinates $(\dot x,\dot y,\dot u,\dot v)$, for what explained above. Moreover, we have the following decomposition: \begin{equation}\label{principalsymbol}
\sigma\big(\Lambda\big)_p(\xi)=\begin{pmatrix}\Theta & \Xi \\ \Gamma & \Delta\end{pmatrix} \ ,\end{equation} where each block is a $2\times 2$ matrix, and each entry in the first and second block-row is a homogeneous polynomial in $\xi_1,\xi_2$ of degree two and one, respectively. After a fairly long computation in coordinates the final expression for $\sigma(\Lambda)_p(\xi)$ is \begin{equation*}
    %\sigma(\Lambda)_p(\xi)=
    \begin{pmatrix}
        -2(f-1)\xi_1\xi_2 & (f-1)(\xi_1^2-\xi_2^2)+\frac{3}{2}\vl w\vl^2f'\vl\xi\vl^2 & -f'u(\xi_1^2\xi_2^2) & -f'v\vl\xi\vl^2 \\ (f-1)(\xi_2^2-\xi_1^2)-\frac{3}{2}\vl w\vl^2f'\vl\xi\vl^2 & -2(f-1)\xi_1\xi_2 & -f'v\vl\xi\vl^2 &
f'u\vl\xi\vl^2  \\  -3u\xi_1 &  -3v\xi_1 & -\xi_2 & \xi_1 \\ -3v\xi_1 & 3u\xi_1 & -\xi_1 & -\xi_2 \end{pmatrix} \ ,
\end{equation*}where $\vl\xi\vl^2:=\xi_1^2+\xi_2^2$ and the second column corresponds to the coefficient of $-\dot y$. We only show how to get block $\Theta$, as with a similar calculation one can obtain the remaining ones. To write down explicitly each entry of $\Theta$, we need to compute the principal symbol of $L_1$ and $L_2$, along directions $(\dot x,-\dot y,\dot u,\dot v)$ with $\dot u=\dot v=0$. According to Lemma \ref{lem:ordertwoderivatives}, $\sigma_2(L_1)$ and $\sigma_2(L_2)$ depend, respectively, on $(f-1)\mathrm d(\dive_g\dot J)+\frac{f'}{4}\mathrm d\langle A,\nabla_{J\bullet}\dot A_0\rangle$ and on $(f-1)\mathrm d(\dive_g\dot J\circ J)-\frac{f'}{4}\mathrm d\langle A,(\nabla_{J\bullet}\dot A_0)J\rangle$. In particular, \begin{align*}
    &\mathrm d(\dive_g\dot J)(\xi,\xi)=\dot x(-2\xi_1\xi_2)-\dot y(\xi_1^2-\xi_2^2) \ , \qquad \mathrm d\big(\langle A, \nabla_{J\bullet}\dot A_0\rangle\big)(\xi,\xi)=-6\vl w\vl^2(\xi_1^2+\xi_2^2)\dot y \ , \\ &\mathrm d(\dive_g\dot J\circ J)(\xi,\xi)=\dot x(\xi_2^2-\xi_1^2)-\dot y(-2\xi_1\xi_2) \ , \qquad \mathrm d\big(\langle A, (\nabla_{J\bullet}\dot A_0)J\rangle\big)(\xi,\xi)=6\vl w\vl^2(\xi_1^2+\xi_2^2)\dot x \ ,
\end{align*} 
where all the above equality are to be intended up to lower order terms in $\xi$. In the end, the upper left block in $\sigma(\Lambda)_p(\xi)$ is given by: \begin{equation*}
    \Theta=\begin{pmatrix}
        -2(f-1)\xi_1\xi_2 & (f-1)(\xi_1^2-\xi_2^2)+\frac{3}{2}\vl w\vl^2f'\vl\xi\vl^2 \\ (f-1)(\xi_2^2-\xi_1^2)-\frac{3}{2}\vl w\vl^2f'\vl\xi\vl^2 & -2(f-1)\xi_1\xi_2 
    \end{pmatrix} \ .
\end{equation*} If $\xi\neq 0$, then the matrix $\Delta$ is invertible as its determinant is equal to $\vl\xi\vl^2$. This allows us to use the determinant formula of block matrices to obtain \begin{align*}
\mathrm{det}\big(\sigma(\Lambda)_p(\xi)\big)&=\vl\xi\vl^2\mathrm{det}\big(\Theta-\Xi\Delta^{-1}\Gamma\big) \\ &=\vl\xi\vl^2\Big(4\xi_1^2\xi_2^2\big(1-f+\frac{3}{2}f'\vl w\vl^2\big)^2+(\xi_1^2-\xi_2^2)^2\big(1-f+\frac{3}{2}f'\vl w\vl^2\big)^2\Big) \ .\end{align*} Since $1-f+\frac{3}{2}f'\vl w\vl^2$ is strictly positive (Lemma \ref{lem:combinationoffunctionfandfprime}), requiring that last expression vanishes is equivalent to the conditions $\xi_1\xi_2=0$ and $\xi_1=\xi_2$, which clearly is possible if and only if $\xi_1=\xi_2=0$. 
%Regarding the case $w=0$, we obtain \begin{align*}\mathrm{det}\big(\sigma(\Lambda)_p(\xi)\big)=4\xi_1^2\xi_2^2+(\xi_2^2-\xi_1^2)^2 \ , \tag{$f(0)=0$} \end{align*} and the same argument as in the previous case applies.
\end{proof}

\begin{manualtheorem}E
    Let $(J,A)$ be a point in $\haffrhozerotilde$, then $$\mathrm{dim} W_{(J,A)}\ge 16g-16+2g \ .$$
\end{manualtheorem}
\begin{proof}
    Notice that, the space $W_{(J,A)}$ can be seen as the kernel of $\Lambda$, namely the matrix of smooth differential operators defined in (\ref{matrixdifferentialoperators}). Let us consider the deformation $tA$, for some $t\in[0,1]$, and look at the corresponding smooth $1$-parameter family of matrices of differential operators: \begin{align*}\big\{\Lambda_t\big\}_{t\in [0,1]}: T_{(J,tA)}&\pickg\to\Omega^2_\mathbb C(\Sg)\oplus\Omega^2(\Sg,\mathrm{End}_0(T\Sg,g_J)) \\ &(\dot J,\dot A_0)\longmapsto (D_t(\dot J,\dot A_0), S_t(\dot J,\dot A_0)) \ ,\end{align*}
    %Let us define the following smooth $1$-parameter family of differential operators $\big\{D_t\big\}_{t\in [0,1]}: T_{(J,A)}\pickg\to\Omega^2_\mathbb C(\Sg)$ as 
    \begin{align*}
        D_t(\dot J,\dot A_0):=&\mathrm d\bigg(\dive_g\big((f_t-1)\dot J\big)+\mathrm d\dot f_t\circ J-\frac{f'_t}{6}\langle(\nabla_{\bullet}tA)J,\dot A_0\rangle\bigg) \\ &+i\mathrm d\bigg(\dive_g\big((f_t-1)\dot J\circ J\big)+\mathrm d(\dot f_0)_t\circ J+\frac{f'_t}{6}\langle\nabla_{\bullet}tA,\dot A_0\rangle\bigg)  \ ,
    \end{align*}$$S_t(\dot J,\dot A_0):=\mathrm d^\nabla\dot A_0(\cdot,\cdot)-tJ(\dive_g\dot J\wedge A)(\cdot,\cdot)  \ ,$$ where $f_t:=f\big(t^2\normpick\big), \dot f_t=t\frac{f_t'}{4}\langle\dot A_0, A\rangle$, and $(\dot f_0)_t=-t\frac{f_t'}{4}\langle\dot A_0J, A\rangle$. Observe that the matrix $\Lambda_t$ is elliptic for any $t\in[0,1]$ (in the sense explained above) as Lemma \ref{lem:simboloprincipaleellittico} holds for any $(J,A)\in\pickg$. In particular, since $\Sg$ is closed the operator matrix $\Lambda_t$ has a well-defined index for any $t\in[0,1]$. By definition, $\Lambda_0$ associates (up to a sign), to each $(\dot J,\dot A_0)$, the element $$\big(\mathrm d\big(\dive_g\dot J\big)+i\mathrm d \big(\dive_g\dot J\circ J\big), \mathrm d^\nabla\dot A_0\big) \ .$$ The homotopy invariance of the Fredholm index (see \cite{nicolaescu2020lectures} for example) implies the following chain of equalities:$$\mathrm{ind}\big(\Lambda\big)=\mathrm{ind}\big(\Lambda_1\big)=\mathrm{ind}\big(\Lambda_t\big)=\mathrm{ind}\big(\Lambda_0\big) \ .$$
    Since the differential equations obtained from the kernel of the matrix $\Lambda_0$ are decoupled in $\dot J$ and $\dot A_0$, we have the following index decomposition: \begin{align*}
\mathrm{ind}\big(\Lambda_0\big)&=\mathrm{ind}\big(\mathrm d(\dive_g \cdot)+i\mathrm d(\dive_g\circ J)\big)+\mathrm{ind}\big(\mathrm d^\nabla\big) \\ &=\mathrm{ind}\big(\mathrm d(\dive_g \cdot)+i\mathrm d(\dive_g\circ J)\big)+10g-10 \ .
    \end{align*}where in the last step we used Lemma \ref{lem:exteriorcovariantderivative}. It is well-known (see \cite{tromba2012teichmuller} for example) that the divergence operator $\dive_g:T_J\almostg\to\Omega^1(\Sg)$ is surjective and its kernel has real dimension equal to $6g-6$. In particular, for any $\alpha\in\Omega^1(\Sg)$ there exists $\dot J\in T_J\almostg$ such that $\dive_g\dot J=\alpha$. Any such real $1$-form has a decomposition $\alpha=\alpha^{1,0}+\alpha^{0,1}$, with $\widebar{\alpha^{0,1}}=\alpha^{1,0}$. Thus, $$\alpha+i\alpha\circ J=\alpha^{1,0}+\alpha^{0,1}+i\big(i\alpha^{1,0}-i\alpha^{0,1}\big)=2\alpha^{0,1} \ .$$ According to this last identity and the surjectivity of the divergence operator, it follows that the cokernel of $\mathrm d\big((\dive_g \cdot)+i(\dive_g\circ J)\big)$ is isomorphic to $$\mathrm{Coker}\big(\partial:\Omega^{0,1}(\Sg)\longrightarrow\Omega^{1,1}(\Sg)\big)\cong\bigslant{\Omega^{1,1}(\Sg)}{\mathrm{Im}(\partial)}=H^{1,1}_{\partial}(\Sg)\cong\R^2 \ ,$$as there are no $(0,2)$-forms on $(\Sg,J)$. In addition, the kernel of $\mathrm d\big((\dive_g \cdot)+i(\dive_g\circ J)\big)$ is given by
    \[
            \{ \dot{J} \in T_{J}\almostg \ | \ \partial\big((\dive_g \cdot)+i(\dive_g\circ J)\big)=0 \} \cong H^{0,1}_{\partial}(\Sigma) \times \mathrm{Ker}(\dive_{g} \cdot) \cong \R^{6g-6} \times \R^{2g}
    \]
    using again the surjectivity of the divergence operator. Therefore, we have
    \begin{align*}
        \mathrm{ind}(\Lambda_{0}) & = \mathrm{ind}\big(\mathrm d(\dive_g \cdot)+i\mathrm d(\dive_g\circ J)\big)+10g-10 \\
        & = 6g-6+2g-2+10g-10 = 16g-16+2g-2 .
    \end{align*}
    To conclude, we notice that all operators $D_{t}$ take value into the subspace of complex exact $2$-forms, hence the dimension of the cokernel of $\Lambda_{t}$ is at least equal to the dimension of $H^{2}_{\C}(\Sigma) \cong \R^{2}$. Thus
    \begin{align*}
        \mathrm{dim} W_{(J,A)} &= \mathrm{dim}(\Ker(\Lambda_{1})) \\
                                &= \mathrm{ind}(\Lambda_{1}) + \mathrm{dim}(\mathrm{Coker}(\Lambda_{1})) \\
                                & \geq \mathrm{ind}(\Lambda_{0}) + 2 \\
                                & = 16g-16+2g-2+2 = 16g-16+2g \ .
    \end{align*}
\end{proof}

\subsection{The preferred subspace inside the tangent to \texorpdfstring{$\haffrhozerotilde$}{HS_0(S,rho)}}\label{sec:4.4}
In this section we prove Theorem \ref{thmE} by using the theory developed so far. In particular, in Lemma \ref{lem:symplectomorphisminvariancedistribution} and Lemma \ref{lem:invariance by the complex structure} we prove the $\Symp(\Sg,\rho)$ and $\i$ invariance of $W_{(J,A)}$, respectively. Then, we find a formula for the action of the almost-complex structure $\i$ on tangent vectors to the $\Symp(\Sg,\rho)$-orbit (Lemma \ref{lem:complexstructureliederivative}) and we study the operator associated with the first equation in (\ref{Gausscodazzi}) (Lemma \ref{lem:differentialGaussequation}). Finally, if $\pi:\haffrhozerotilde\to\defgtilde$ denotes the quotient projection, where $\defgtilde$ is the quotient of $\haffrhozerotilde$ by $\Ham(\Sg,\rho)$, the injectivity of the map $\mathrm d_{(J,A)}\pi:W_{(J,A)}\to T_{[J,A]}\defgtilde$ is proven in Lemma \ref{lem:Wisomorphictofinitedimensionalquotient} by using all the previous results. The only part of Theorem \ref{thmE} that is left is the inclusion $W_{(J,A)}\subset T_{(J,A)}\haffrhozerotilde$, as it is necessary to explain first the connection between the system of differential equations (\ref{differentialequations}) and the theory of symplectic reduction. For this reason its discussion is postponed to Section \ref{sec:5.3}. The results presented in this section follow closely the ones given for the anti-de Sitter case (\cite[\S 4.5]{mazzoli2021parahyperkahler}), even though one of the two tensors we work with is of a different type.
\begin{lemma}\label{lem:symplectomorphisminvariancedistribution}
    For every symplectomorphism $\phi$ of $(\Sg,\rho)$ and for every $(\dot J,\dot A)\in W_{(J,A)}$, we have $(\phi^*\dot J,\phi^*\dot A)\in W_{(\phi^*J, \phi^*A)}$. In other words, the distribution $\{W_{(J,A)}\}_{(J,A)\in\haffrhozerotilde}$ is invariant under the action of $\Symp(\Sg,\rho)$.
\end{lemma}
\begin{proof}
    The assumption that $\phi$ is a symplectomorphism ($\phi^*\rho=\rho$) is crucial in order to prove that $g_{\phi^*J}$, the metric associated with the area form $\rho$ and complex structure $\phi^*J$, is equal to the pull-back metric $\phi^*g_J=\phi^*\big(\rho(\cdot,J\cdot)\big)$. In other words, we are saying that $\phi:(\Sg, g_{\phi^*J})\to(\Sg, g_J)$, is an isometry. In particular, for any endomorphism of the tangent bundle $B$ we get $$\phi^*\Big(\dive_g B\Big)=\dive_{\phi^*g}(\phi^*B) \ .$$ Moreover, the parts involving the scalar product between tangent vectors $\dot J,\dot J'\in T_J\almostg$ and $\dot A,\dot A'$ are preserved by $\phi$ (see Section \ref{sec:3.2} and Section \ref{sec:4.1}). As for the rest of the terms in the equations defining $W_{(J,A)}$, we see that they are preserved by $\phi$ using the naturality of the action and the functoriality of the involved operators, such as the induced connection $\nabla$ and the exterior covariant derivative $\mathrm d^{\nabla}$ on End$_0(\Sg, g)$-valued $1$-form.
\end{proof}
\begin{remark}
    Notice that the above lemma holds for any symplectomorphism $\phi$ not necessarily Hamiltonian. This is a stronger result than what we need in Theorem \ref{thmE}. In particular, with the same argument it is possible to prove the $\Symp(\Sg,\rho)$-invariance of $\i$ and $\g_f$.
\end{remark}
\begin{lemma}\label{lem:invariance by the complex structure}
    For any $(J,A)\in\haffrhozerotilde$, the subspace $W_{(J,A)}$ is preserved by $\i$.
\end{lemma}
\begin{proof}
Recall that by definition $\i(\dot J,\dot A)=(-J\dot J, -\dot AJ-A\dot J)=:(\dot J',\dot A')$. We only need to show that the pair $(\dot J',\dot A')$ still satisfies the equations defining $W_{(J,A)}$. In fact, \begin{align*}
     (f-1)\dive_g(\dot J')&=(f-1)\dive_g\big(-J\dot J\big) \tag{rel. (\ref{divergenzaendomorfismoeJ})} \\ &=(f-1)\dive_g\dot J\circ J \ .\end{align*}  Moreover, \begin{align*}
         &\mathrm d\Big(\frac{f'}{4}\langle\dot A_0',A\rangle\Big)\circ J=\mathrm d\Big(\frac{f'}{4}\langle-\dot A_0J,A\rangle\Big)\circ J \ , \\ &-\frac{f'}{6}\langle(\nabla_\bullet A)J,\dot A_0'\rangle=-\frac{f'}{6}\langle(\nabla_\bullet A)J,-\dot A_0J\rangle=\frac{f'}{6}\langle\nabla_\bullet A,\dot A_0\rangle \ , 
     \end{align*}where in the last step we used relation (\ref{rel:cpxstructurescalarprod1}). By using $\i^2=-\mathds 1$, it follows that the element $\i(\dot J,\dot A)$ satisfies the second equation in (\ref{differentialequations}) as well. Regarding the last equation, notice that, according to Remark \ref{rem:codazzilinearizzatainvariante}, it is equivalent to $\mathrm d^\nabla\dot A_0(\bullet,J\bullet)=\dive_g\dot J(\bullet)A(J\bullet)-\dive_g\dot J(J\bullet)A(\bullet)$. Therefore, for any $X\in\Gamma(T\Sg)$, we get \begin{align*}
         \mathrm d^\nabla\big(\dot A_0'\big)(X,JX)&=-\mathrm d^\nabla\big(\dot A_0J\big)(X,JX) \\ &=-\big(\mathrm d^\nabla\dot A_0\big)(X,JX)J \tag{$\nabla_\bullet J=0$} \\ &=-(\dive_g\dot J)(X)A(JX)J+(\dive_g\dot J)(JX)A(X)J \tag{$A(J\cdot)=A(\cdot)J$} \\ &=(\dive_g\dot J)(X)A(X)+\dive_g\dot J(JX)A(X)J \ .
     \end{align*}On the other hand, \begin{align*}
         (\dive_g\dot J')(X)A(JX)-(\dive_g\dot J')(JX)A(X)&=-(\dive_gJ\dot J)(X)A(JX)+(\dive_gJ\dot J)(JX)A(X) \\ &=(\dive_g\dot J)(JX)A(JX)+(\dive_g\dot J)(X)A(X)  \\ &=(\dive_g\dot J)(JX)A(X)J+(\dive_g\dot J)(X)A(X) \ ,
     \end{align*} where we used relation (\ref{divergenzaendomorfismoeJ}) on the first step and $A(J\cdot)=A(\cdot)J$ on the second one. 
\end{proof}
\begin{lemma}\label{lem:complexstructureliederivative}
    For every symplectic vector field $X$ on $(\Sg,\rho)$ and for every $(J,A)\in\pickg$, with $C(\cdot,\cdot,\cdot)=g_J\big(A(\cdot)\cdot,\cdot\big)$ equal to the real part of a holomorphic cubic differential on $(\Sg, J)$, we have $\i\big(\liederivative_XJ,g_J^{-1}\liederivative_XC\big)=\big(-\liederivative_{JX}J,-g_J^{-1}\liederivative_{JX}C\big)$.
\end{lemma}
\begin{proof}
    For any vector field $V$ on the surface, let use define the operator $M_V:\Gamma(T\Sg)\to\Gamma(T\Sg)$ as $M_V(Y):=\nabla^g_YV$, where $\nabla^g$ is the Levi-Civita connection with respect to $g\equiv g_J=\rho(\cdot,J\cdot)$. Then, for any $Y\in\Gamma(T\Sg)$, we have\begin{align*}
        (\liederivative_VJ)Y&=[V,JY]-J([V,Y]) \\ &=\nabla^g_V(JY)-\nabla^g_{JY}V-J(\nabla^g_VY)+J(\nabla^g_YV) \tag{$\nabla^g$ is torsion-free}  \\ &=J(\nabla^g_VY)-M_V(JY)-J(\nabla^g_VY)+JM_V(Y) \tag{$\nabla^g_\bullet J=0$} \\ &=(JM_V-M_VJ)(Y) \ .
    \end{align*}The above computation implies that \begin{equation}\label{liederivativeJ}
        \liederivative_VJ=JM_V-M_VJ \ .
    \end{equation}Now since $C(\cdot,\cdot,\cdot)$ is a $(0,3)$-tensor, for any $Y,Z,U\in\Gamma(T\Sg)$, its Lie derivative can be computed as follows \begin{align*}
(\liederivative_VC)(Y,Z,U)=V\cdot C(Y,Z,U)-C([V,Y],Z,U)-C(Y,[V,Z],U)-C(Y,Z,[V,U]) \ .
    \end{align*}Moreover, using the relation $$V\cdot C(Y,Z,U)=(\nabla^g_VC)(Y,Z,U)+C(\nabla^g_VY,Z,U)+C(Y,\nabla^g_VZ,U)+C(Y,Z\nabla_V^gU) \ ,$$ we obtain that \begin{equation}\label{liederivativeC}
        (\liederivative_VC)(\cdot,\cdot,\cdot)=(\nabla^g_VC)(\cdot,\cdot,\cdot)+C(M_V\cdot,\cdot,\cdot)+C(\cdot,M_V\cdot,\cdot)+C(\cdot,\cdot,M_V\cdot) \ .
    \end{equation}In particular, by re-writing the last relation using the associated $(1,2)$-tensor defined as $A=g^{-1}C$ and using the compatibility between $\nabla^g$ and the metric $g$, we get \begin{equation}\label{liederivative-g-minusone-C}
        (g^{-1}\liederivative_VC)(\cdot)=(\nabla^g_VA)(\cdot)+A(M_V\cdot)+A(\cdot)M_V+M_V^*A(\cdot) \ ,
    \end{equation}where $M_V^*$ denotes the $g$-adjoint operator of $M_V$. Now let us apply the almost-complex structure $\i$ to the pair $(\liederivative_XJ,g^{-1}\liederivative_XC)$ with $X$ a $\rho$-symplectic vector field. Therefore, $$\i\big(\liederivative_XJ,g^{-1}\liederivative_XC\big)=\big(-J\liederivative_XJ,-(g^{-1}\liederivative_XC)(\cdot)J-A(\cdot)\liederivative_XJ\big) \ .$$ Since $J$ is $\nabla^g$-parallel then $M_{JX}=JM_X$, so that the first component of $\i\big(\liederivative_XJ,g^{-1}\liederivative_XC\big)$ is given by \begin{align*}
    -J\liederivative_XJ&=-J(JM_X-M_XJ) \tag{rel. (\ref{liederivativeJ}) for $V=X$} \\ &=-(JM_{JX}-M_{JX}J) \\ &=-\liederivative_{JX}J \tag{rel. (\ref{liederivativeJ}) for $V=JX$} \ .
    \end{align*}Regarding the second component of $\i\big(\liederivative_XJ,g^{-1}\liederivative_XC\big)$, using relation (\ref{liederivativeJ}) and (\ref{liederivative-g-minusone-C}) for $V=X$, we have \begin{align*}
        -(g^{-1}\liederivative_XC)(\cdot)J-A(\cdot)\liederivative_XJ&=-(\nabla_X^gA)(\cdot)J-A(M_X\cdot)J-M^*_XA(\cdot)J-A(\cdot)M_XJ \\ & \ \ \ \ -A(\cdot)JM_X+A(\cdot)M_XJ \\ &=-(\nabla_X^gA)(\cdot)J-A(M_X\cdot)J-A(\cdot)JM_X+M^*_XJA(\cdot) \ ,
    \end{align*}where in the last equality we used $A(\cdot)J=-JA(\cdot)$. On the other hand, using relation (\ref{liederivative-g-minusone-C}) with $V=JX$, we get \begin{align*}
        -(g^{-1}\liederivative_{JX}C)(\cdot)&=-(\nabla_{JX}^gA)(\cdot)-A(M_{JX}\cdot)-A(\cdot)M_{JX}-M^*_{JX}A(\cdot) \\ &=-(\nabla_{JX}^gA)(\cdot)-A(JM_{X}\cdot)-A(\cdot)JM_{X}-(JM_{X})^*A(\cdot) \tag{$M_{JX}=JM_X$} \\ &=-(\nabla_{X}^gA)(J\cdot)-A(JM_{X}\cdot)-A(\cdot)JM_{X}-(JM_{X})^*A(\cdot) \tag{Theorem \ref{thm:picktensor}} \\ &=-(\nabla_{X}^gA)(J\cdot)-A(M_{X}\cdot)J-A(\cdot)JM_{X}+M_X^*JA(\cdot) \ , \end{align*}where in the last step we used $A(J\cdot)=A(\cdot)J$ and $J^*=-J$. Finally, we conclude by observing that \begin{align*}
            (\nabla^g_XA)(JY)Z&=\nabla^g_X\big(A(JY)Z\big)-A\big(\nabla^g_X(JY)\big)Z-A(JY)\nabla^g_XZ \\ &=\nabla^g_X\big(A(Y)JZ\big)-A\big(J\nabla^g_XY\big)Z-A(Y)J\nabla^g_XZ \\ &=\nabla^g_X\big(A(Y)JZ\big)-A\big(\nabla^g_XY\big)JZ-A(Y)\nabla^g_X(JZ) \\ &=(\nabla^g_XA)(Y)JZ, \qquad \forall X,Y,Z\in\Gamma(T\Sg) \ .
        \end{align*}
\end{proof}
\begin{lemma}\label{lem:differentialGaussequation}
Let $G:\pickg\to C^\infty(\Sg)$ be the operator defined as $G(J,A):=K_h+1-\vl\vl q\vl\vl_h^2$, where $h$ is the metric in the conformal class of $g_J$ with conformal factor $e^F$ (see (\ref{functionalequationF})), and $q$ is a cubic differential whose real part is equal to $C=g_JA$. Suppose that $(J,A)$ satisfies equations $\mathrm{(\ref{Gausscodazzi})}$ and let $U$ be a vector field on $\Sg$. Then, $$\mathrm d_{(J,A)}G\big(\liederivative_UJ,g_J^{-1}\liederivative_UC\big)=-\frac{1}{2}\Delta_h\lambda+(1+2\vl\vl q\vl\vl_h^2)\lambda \ ,$$ where $\lambda:=\dive_{g_J}U\bigg(\frac{3}{2}\vl\vl q\vl\vl_{g_J}^2F'\Big(\frac{\vl\vl q\vl\vl^2_{g_J}}{2}\Big)-1\bigg)$. In particular, if the element $\big(\liederivative_UJ,g_J^{-1}\liederivative_UC\big)$ belongs to the kernel of $\mathrm d_{(J,A)}G$, then $U$ is symplectic.
\end{lemma}
\begin{proof}
    Let us denote with $\{\psi_t\}_{t\in[0,1]}$ the flow of $U$, and let $(J,C)$ be a point in $\pickg$. Consider the path $\{(J_t,C_t)\}_{t\in[0,1]}\subset\pickg$ given by $(J_t,C_t)=(\psi^*_tJ,\psi^*_tC)$ so that $(J_0,C_0)=(J,C)$. In particular, $$\liederivative_UJ=\frac{\mathrm d}{\mathrm dt}\psi^*_tJ\Big|_{t=0},\quad g^{-1}\liederivative_UC=g^{-1}\frac{\mathrm d}{\mathrm dt}\psi^*_tC\Big|_{t=0}, \quad g\equiv g_J \ .$$ The final goal will be to compute $\frac{\mathrm d}{\mathrm d t}G(J_t,C_t)|_{t=0}$. Let us first determine the Riemannian metric $g_t:=\rho(\cdot,J_t\cdot)$, where $J_t=\mathrm d\psi^{-1}_t\circ J\circ\mathrm d\psi_t$. \begin{align*}
    g_t&=\rho\big(\cdot,(\mathrm d\psi^{-1}_t\circ J\circ\mathrm d\psi_t)\cdot\big)=\rho\big((\mathrm d\psi^{-1}_t\circ\mathrm d\psi_t)\cdot,(\mathrm d\psi^{-1}_t\circ J\circ\mathrm d\psi_t)\cdot\big) \\ &=\big(\mathrm{det}(\mathrm d\psi^{-1}_t)\circ\psi_t\big)\rho\big(\mathrm d\psi_t\cdot,(J\circ\mathrm d\psi_t)\cdot\big)=\big(\mathrm{det}(\mathrm d\psi^{-1}_t)\circ\psi_t\big)g\big(\mathrm d\psi_t\cdot,\mathrm d\psi_t\cdot\big) \\ &=\big(\mathrm{det}(\mathrm d\psi^{-1}_t)\circ\psi_t\big)\psi_t^*g \ .
\end{align*}In particular, $g_t$ is conformal to $\psi^*_tg$ with conformal factor given by $u_t:=\big(\mathrm{det}(\mathrm d\psi^{-1}_t)\circ\psi_t\big)$. Now let $F:[0,+\infty)\to\R$ be the function defined in Lemma \ref{lem:functionFef} and consider the conformal change of metric $h=e^Fg$, where $F$ is evaluated at $\vl\vl q\vl\vl_g^2$ divided by $2$. The next step is to determine the Riemannian metric $$h_t:=e^{F\Big(\frac{\vl\vl q_t\vl\vl_{g_t}^2}{2}\Big)}g_t=e^{F\Big(\frac{\vl\vl q_t\vl\vl_{g_t}^2}{2}\Big)}u_t\cdot\psi_t^*g \ ,$$where $q_t$ is the $J_t$-holomorphic cubic differential whose real part is equal to $C_t$. Therefore, \begin{align*}
    h_t&=e^{F\Big(\frac{\vl\vl q_t\vl\vl_{g_t}^2}{2}\Big)}u_t\cdot\psi_t^*g=e^{F\Big(\frac{\vl\vl q_t\vl\vl^2_{g_t}}{2}\Big)}u_t\cdot e^{-F\Big(\frac{\vl\vl q\vl\vl_{g}^2\circ\psi_t}{2}\Big)}\psi^*_th \\ &=e^{F\Big(\frac{\vl\vl q_t\vl\vl^2_{g_t}}{2}\Big)-F\Big(\frac{\vl\vl q\vl\vl_{g}^2\circ\psi_t}{2}\Big)}u_t\cdot\psi_t^*h=v_t\cdot\psi_t^*h \ .
\end{align*}Again, the metric $h_t$ is conformal to $\psi_t^*h$ with conformal factor $v_t$ (notice that $v_0=u_0\equiv 1$). Using the formula of curvature by conformal change of metric, we get \begin{align*}
    K_{h_t}&=K_{v_t\psi_t^*h} \\ &=v_t^{-1}\Big(K_{\psi_t^*h}-\frac{1}{2}\Delta_{\psi_t^*h}\ln v_t\Big) \\ &=v_t^{-1}\Big(K_h\circ\psi_t-\frac{1}{2}\big(\Delta_h\ln(v_t\circ\psi_t^{-1})\big)\circ\psi_t\Big) \ ,
\end{align*}where in the last equality we used the functoriality of the Gaussian curvature and of the Laplacian, namely $$K_{\psi_t^*h}=\psi_t^*(K_h),\quad \Delta_{\psi_t^*h}\ln v_t=\psi_t^*\Big(\Delta_h\ln(v_t\circ\psi_t^{-1})\Big) \ .$$ The last term we need to determine in $G(J_t,C_t)$ is the one involving the norm of the cubic differential $q_t$. \begin{equation}\label{normcubicT}\vl\vl q_t\vl\vl_{h_t}^2=\vl\vl q_t\vl\vl_{v_t\psi_t^*h}^2=v_t^{-3}\vl\vl \psi_t^*q\vl\vl_{\psi_t^*h}^2=v_t^{-3}\vl\vl q\vl\vl_h^2\circ\psi_t \ .\end{equation} We can finally deduce an expression for the term $$K_{h_t}-\vl\vl q_t\vl\vl_{h_t}^2=v_t^{-1}\Big(K_h\circ\psi_t-\frac{1}{2}\big(\Delta_h\ln(v_t\circ\psi_t^{-1})\big)\circ\psi_t-v_t^{-2}\vl\vl q\vl\vl_h^2\circ\psi_t\Big) \ , $$ and compute the first order variation of operator $G$ along the path $t\mapsto(J_t,C_t)$, obtaining \begin{align*}
    \frac{\mathrm d}{\mathrm dt}\big(K_{h_t}-\vl\vl q\vl\vl_{h_t}^2+1\big)\Big|_{t=0}&=-\dot v\big(K_h-\vl\vl q\vl\vl_h^2\big)+U(K_h)-\frac{1}{2}\Delta_h\dot v+2\dot v\vl\vl q\vl\vl_h^2-U\big(\vl\vl q\vl\vl_h^2\big) \\ &=\dot v\big(1+2\vl\vl q\vl\vl_h^2\big)-\frac{1}{2}\Delta_h\dot v \ ,
\end{align*}where in the last line we used that $(J,C)$ satisfies $G(J,C)=0$. At this point, it only remains to compute $\dot v$, i.e. the first order variation of $v_t$ \begin{align*}
    \dot v&=\frac{\mathrm dv_t}{\mathrm dt}\Big|_{t=0}=\frac{\mathrm d(v_t\circ\psi_t^{-1})}{\mathrm dt}\Big|_{t=0} \tag{$v_0\equiv 1$} \\ &=\frac{\mathrm d}{\mathrm dt}e^{F\Big(\frac{\vl\vl q_t\vl\vl^2_{g_t}\circ\psi_t^{-1}}{2}\Big)-F\Big(\frac{\vl\vl q\vl\vl_{g}^2}{2}\Big)}u_t\circ\psi_t^{-1}\Big|_{t=0} \\ &=\dot u+\frac{1}{2}F'\Big(\frac{\vl\vl q\vl\vl_g^2}{2}\Big)\frac{\mathrm d\big(\vl\vl q_t\vl\vl_{g_t}^2\circ\psi_t^{-1}\big)}{\mathrm dt}\Big|_{t=0} \ .
\end{align*}By imitating the steps performed for relation (\ref{normcubicT}), we deduce that $$\vl\vl q_t\vl\vl_{g_t}^2=u_t^{-3}\vl\vl q\vl\vl^2_g\circ\psi_t \ .$$ Since $(\psi_t)$ represents the flow of $U$, the first order variation of the conformal factor $u_t$ is given by $$\dot u=\frac{\mathrm d}{\mathrm dt}\Big(\det(\mathrm d\psi_t^{-1})\circ\psi_t\Big)\Big|_{t=0}=-\dive_g U \ .$$ To conclude, we have \begin{align*}
    \dot v&=\dot u-\frac{3}{2}F'\Big(\frac{\vl\vl q\vl\vl_g^2}{2}\Big)\dot u\vl\vl q\vl\vl_g^2 \\ &=\dive_gU\Big(\frac{3}{2}F'\Big(\frac{\vl\vl q\vl\vl_g^2}{2}\Big)\vl\vl q\vl\vl_g^2-1\Big) \ ,
\end{align*}hence the first order variation of operator $G$ along the path $t\mapsto(J_t,C_t)$ is $$\mathrm d_{(J,A)}G\big(\liederivative_UJ,g^{-1}\liederivative_UC\big)=-\frac{1}{2}\Delta_h\lambda+(1+2\vl\vl q\vl\vl_h^2)\lambda,\quad \lambda:=\dive_gU\Big(\frac{3}{2}F'\Big(\frac{\vl\vl q\vl\vl_g^2}{2}\Big)\vl\vl q\vl\vl_g^2-1\Big) \ .$$Regarding the second part of the statement, observe that the following inequality holds $$T(\lambda):=-\frac{1}{2}\Delta_h\lambda+(1+2\vl\vl q\vl\vl_h^2)\lambda\ge -\frac{1}{2}\Delta_h\lambda +\lambda=:S(\lambda) \ .$$ Since the linear operator $S$ is known to be self-adjoint and positive, hence injective, over $L^2(\Sg, \mathrm d\mathfrak a_h)$, so is the linear operator $T$. Therefore, if $(\liederivative_UJ,g^{-1}\liederivative_UC)$ lies inside the kernel of $\mathrm d_{(J,A)}G$, then the function $\lambda=\dive_gU\Big(\frac{3}{2}F'\Big(\frac{\vl\vl q\vl\vl_g^2}{2}\Big)\vl\vl q\vl\vl_g^2-1\Big)$ is send to $0$ by the operator $T$. At this point, we would conclude by saying that the divergence of the vector field $U$ is zero (see (\ref{divergenzaXformularho})), which is obviously true if $A=0$, i.e. $q=0$. As for the first order variation of the operator $G$ at points where $q\neq 0$, using Lemma \ref{lem:combinationoffunctionfandfprime} with $ t=\frac{\vl\vl q\vl\vl_J^2}{2}$, we get that the function $ \frac{3}{2}F'\Big(\frac{\vl\vl q\vl\vl_g^2}{2}\Big)\vl\vl q\vl\vl_g^2-1$ is strictly negative. In particular, $\dive_gU\Big(\frac{3}{2}F'\Big(\frac{\vl\vl q\vl\vl_g^2}{2}\Big)\vl\vl q\vl\vl_g^2-1\Big)$ is zero if and only if $\dive_gU=0$. 
\end{proof}
\begin{remark}
    In order to conclude the proof of the Theorem \ref{thmE}, one of the remaining results to show is the inclusion of the subspace $W_{(J,A)}$ inside the tangent space to the infinite-dimensional space $\haffrhozerotilde$. In order to show this inclusion, it is necessary to explain how the differential equations defining $W_{(J,A)}$ are related to the process of infinite-dimensional symplectic reduction. In view not to overextending the discussion too much, during the proof of the last lemma that follows, we will use a result presented and proved in Section \ref{sec:5.3}.
\end{remark}
\begin{lemma}\label{lem:Wisomorphictofinitedimensionalquotient}
    For every $(J,A)\in\haffrhozerotilde$, we have $$W_{(J,A)}\cap T_{(J,A)}\Big(\Ham(\Sg, \rho)\cdot(J,A)\Big)=\{0\} \ .$$ In particular, the natural quotient projection $\pi:\haffrhozerotilde\to\defgtilde$ induces a linear isomorphism $$\mathrm d_{(J,A)}\pi:W_{(J,A)}\overset{\cong}{\longrightarrow}T_{[J,A]}\defgtilde \  .$$ 
\end{lemma}\begin{proof}
    Let $X$ be a Hamiltonian vector field on $\Sg$ with Hamiltonian function $H$, and suppose that $(\liederivative_XJ,g^{-1}\liederivative_XC)$ belongs to $W_{(J,A)}$. Thus, according to Lemma \ref{lem:complexstructureliederivative} and the $\i$-invariance of $W_{(J,A)}$, the same has to hold for $\i(\liederivative_XJ,g^{-1}\liederivative_XC)=(-\liederivative_{JX}J, -g^{-1}\liederivative_{JX}C)$. Since $W_{(J,A)}$ is contained in $T_{(J,A)}\haffrhozerotilde$ (see Proposition \ref{prop:Winsidehaffrhozero}), the differential of operator $G$ considered in Lemma \ref{lem:differentialGaussequation} has to send the pair $(-\liederivative_{JX}J, -g^{-1}\liederivative_{JX}C)$ to zero. By the second part of Lemma \ref{lem:differentialGaussequation}, we deduce that $JX$ is $\rho$-symplectic, namely $\mathrm d(\iota_{JX})=0$. This implies that the $1$-form $-\mathrm dH\circ J=-(\iota_X\rho)\circ J=\iota_{JX}\rho$ is closed, and therefore the function $H$ is $g$-harmonic (since $\mathrm d\big(\mathrm d H\circ J\big)=-\Delta_gH\rho$). The only harmonic functions on a closed manifold are the constants, hence we deduce that the vector field $X$ is equal to zero, which proves the first part of the statement. Regarding the second one, let $\defgtilde$ be the quotient of the infinite-dimensional space $\haffrhozerotilde$ by the group $\Ham(\Sg,\rho)$ and consider the quotient projection $\pi:\haffrhozerotilde\to\defgtilde$. By definition, the kernel of $\mathrm d_{(J,A)}\pi$ coincides with $T_{(J,A)}\big(\Ham(\Sg,\rho)\cdot (J,A)\big)$. Hence, by the first part of the statement, the map $\mathrm d_{(J,A)}\pi$ is injective. Moreover, since $\dim W_{(J,A)} \geq 16g-16+2g$ (Theorem \ref{thmD}) and $\dim\defgtilde=16g-16+2g$, this is actually an isomorphism.
\end{proof}
\begin{remark}
The above lemma shows a major difference with the $\PSL(2,\R)\times\PSL(2,\R)$ case (\cite[Lemma 4.21]{mazzoli2021parahyperkahler}), where the authors were able to obtain a similar result for the group of symplectomorphisms of the surface not necessarily Hamiltonian. This forces us to perform an additional (finite-dimensional) quotient to obtain the Hitchin component, and thus produces additional analytical difficulties carried out in Section \ref{sec:6.1}.
\end{remark}
\subsection{The circle action on \texorpdfstring{$\haff$}{HS(S)}}\label{sec:4.5}
Recall that the space $\pick$ consists of pairs $(J,A)$, where $J$ is an almost-complex structure on $\R^2$ and $A$ is a $1$-form with values in the trace-less and $g_J^0$-symmetric endomorphisms bundle of $\R^2$ such that $A(J\cdot)=A(\cdot)J$ and $A(X)Y=A(Y)X, \ \forall X,Y\in T\R^2$ (see Section \ref{sec:4.1} for more details). In particular, there is $\text{MCG}(T^2)\cong\SL(2,\Z)$-equivariant isomorphism between $\pick$and the holomorphic vector bundle $\cubic$ of cubic differentials over Teichm\"uller space of the torus (\cite[Corollary 3.3]{rungi2021pseudo}). In fact, if $(J,A)\in\pick$ then the $(0,3)$-tensor $C(\cdot,\cdot,\cdot)=g_J^0\big(A(\cdot)\cdot,\cdot\big)$ is the real part of a $J$-holomorphic cubic differential $q$ on $(T^2, J)$. The natural $S^1$-action on $\cubic$ given by $(J,q)\mapsto (J,e^{-i\theta}q)$, can be induced on $\pick$ and results in the following formula \begin{align*}\widehat\Psi_\theta \ :& \ \pick\longrightarrow\pick \\ &(J,A)\mapsto (J,\cos\theta A(\cdot)-\sin\theta A(\cdot)J) \ . \ \end{align*}
It is clear from the definition that $\widehat\Psi_\theta$ preserves the $0$-section in $\pick$ (seen as a vector bundle over $\almost\cong\mathcal T(T^2)$), hence it induces an $S^1$-action on $\deft$ which will still be denoted by $\widehat\Psi_\theta$ by abuse of notation. \begin{lemma}[\cite{rungi2021pseudo}]\label{lem:circleactiontoruscase}
    The $S^1$-action on $\deft$ is Hamiltonian with respect to $\widehat\ome_f$ and it satisfies \begin{equation*}
     \widehat\Psi_\theta^{*}\widehat\g_f=\widehat\g_f \ ,
\end{equation*} for any choice of a smooth function $f:[0,+\infty)\to(-\infty,0]$ such that: $f(0)=0, f'(t)<0$ for any $t>0$ and $\displaystyle\lim_{t\to+\infty}f(t)=-\infty$. Moreover, the Hamiltonian function is given by $$\widehat H\big(J,A\big)=\frac 2{3}f\Big(\frac{\vl\vl q\vl\vl^2_J}{2}\Big) \ , $$ where $q$ is the $J$-holomorphic cubic differential whose real part is equal to $C=g_JA$.
\end{lemma}Moving on to the case of genus $g\ge 2$, we still have an $S^1$-action on $\cubicg$ given by $([J], q)\mapsto ([J], e^{-i\theta}q)$, which can be induced on the $\SL(3,\R)$-Hitchin component using the parametrization $$\Phi:\hitc\overset{\cong}{\longrightarrow}\cubicg$$ found by Labourie and Loftin (see Section \ref{sec:2.2}). Thanks to Proposition \ref{prop:defgandhaff} and to the construction explained in Section \ref{sec:2.3}, we know that $\hitc$ is diffeomorphic to the following space \begin{equation*}\haffrhozero:=\bigslant{\left\{ (J,C) \ \Bigg | \ \parbox{19em}{$J$ is an (almost) complex structure on $\Sg$ \\ $C$ is the real part of a $J$-cubic differential \\ $\big(h:=e^{F\big(\frac{\vl\vl q\vl\vl^2_{g_J}}{2}\big)}g_J, A:=g_J^{-1}C\big)$ satisfy (\ref{Gausscodazzi})} \right\}}{\Symp_0(\Sg)} \ ,\end{equation*}where $F:[0,+\infty)\to\R$ is the smooth function defined in Lemma \ref{lem:functionFef}. In particular, we can then describe the induced $S^1$-action by the following formula:
\begin{align*}\Psi_\theta \ :& \ \hitc\longrightarrow\hitc \\ &(J,A)\mapsto (J,\cos\theta A(\cdot)-\sin\theta A(\cdot)J) \ . \ \end{align*} 
\begin{manualtheorem}C 
Let $\rho$ be a fixed area form on $\Sg$, then the circle action on $\hitc$ is Hamiltonian with respect to $\ome_f$ and it satisfies: \begin{equation*}\Psi_\theta^*\g_f=\g_f,\quad\forall\theta\in\R \ .\end{equation*}The Hamiltonian function is given by: $$H(J,q):=\frac{2}{3}\int_\Sg f\Big(\frac{\vl\vl q\vl\vl^2_{g_J}}{2}\Big)\rho,$$ where $f:[0,+\infty)\to(-\infty,0]$ is the smooth function defined by (\ref{definitionf}).
\end{manualtheorem}
The proof of Theorem \ref{thmC} is simply an adaptation of the proof made in the torus case (\cite{rungi2021pseudo} and Lemma \ref{lem:circleactiontoruscase}). In fact, as already explained in Remark \ref{rem:identitypointwise}, identities valid for elements in $\pick$ can be interpreted as point-wise identities for smooth sections in $\pickg$, and, according to the construction explained in Section \ref{sec:4.1}, the Hitchin component $\hitc$ can be seen a subset of $\pickg$.
\section{The infinite dimensional symplectic reduction}
In this section we present the process that led us to the definition of the pseudo-K\"ahler structure on the $\SL(3,\R)$-Hitchin component and the characterization of its tangent space as described in Remark \ref{rem:equationsHitchincomponent}. The main theorem we will use was proved by Donaldson (see \cite{donaldson2003moment}).

\subsection{Donaldson's construction}\label{sec:5.1}
In the following discussion, we present a (much more general) construction done in \cite[\S 2]{donaldson2003moment} and adapted to our case of interest. \\ \\ 
Since we will be using a lot of notation from Section \ref{sec:4.1}, let us briefly recall the construction of the infinite-dimensional space $\pickg$. It has been defined as the space of smooth sections of the bundle $$P\big(\pick\big):=\bigslant{P\times\pick}{\SL(2,\R)}\longrightarrow\Sg \ ,$$ where $\SL(2,\R)$ acts diagonally on two factors. In particular, each element in $\pickg$ can be described as a pair $(J,A)$, with $J$ an almost-complex structure on $\Sg$, and $A$ a $1$-form with values in the trace-less and $g_J$-symmetric endomorphisms of $T\Sg$ such that $A(J\cdot)=A(\cdot)J$ and $A(X)Y=A(Y)X, \ \forall X,Y\in\Gamma(T\Sg)$. Moreover, a tangent vector $(\dot J,\dot A)$, where $\dot A:=g_J^{-1}\dot C$, at $(J,A)$ can be considered as the data of:\begin{itemize}
     \item a section $\dot J$ of End$(T\Sg)$ such that $\dot JJ+J\dot J=0$, namely $\dot J$ is a $g_J$-symmetric and trace-less endomorphism of $T\Sg$; \item an End$(T\Sg, g_J)$-valued $1$-form $\dot A$ such that \begin{equation*}\dot A=\dot A_0+\frac{1}{2}\tr(JA\dot J)\mathds 1 \ ,\end{equation*}where $\mathds 1$ is the $2\times 2$ identity matrix and $\dot A_0$ is the trace-less part of $\dot A$. In particular, the trace-part $\dot A_{\text{tr}}$ of $\dot A$ is uniquely determined by $\dot J$.
 \end{itemize}
Let us denote with $s=(J,A)$ an element in $\pickg$, and with $\dot s$ the corresponding tangent vector. Recall there is a natural $\SL(2,\R)$-action on $\pick$ (see (\ref{SLaction})). Given an $\SL(2,\R)$-invariant symplectic form $\widehat\omega$ on $\pick$, there is an induced symplectic structure on each vertical subspace of $P\big(\pick\big)$, denoted with $\widehat\omega_{s(x)}$ for $x\in\Sg$. In particular, given two tangent vectors $\dot s,\dot s'\in T_s\pickg$, we can define \begin{equation}\label{symplecticformDonaldson}\ome_s(\dot s,\dot s'):=\int_{\Sg}\widehat\omega_s(\dot s,\dot s')\rho \ .\end{equation} This gives rise to a formal symplectic structure on $\pickg$ which is invariant by the action of $\Symp_0(\Sg,\rho)$. Now, suppose for a moment that the $\SL(2,\R)$-action on $\pick$ we introduced in (\ref{SLaction}) is Hamiltonian with respect to the symplectic form $\widehat\omega$, and with moment map $\widehat\mu:\pick\to\Lsl(2,\R)^*$. Given any section $s\in\pickg$, the moment map $\widehat\mu$ induces a section $\widehat\mu_s$ of the bundle End$_0(T\Sg)^*$. Then, the following result holds \begin{theorem}[{\cite[Theorem 9]{donaldson2003moment}}]\label{thm:donaldson}
    Let $\rho$ be an area form on $\Sg$ and let $\nabla$ be any torsion-free connection on $T\Sg$ satisfying $\nabla\rho=0$. Define the map $\boldsymbol{\mu}:\pickg\to\Omega^2(\Sg)$ as follows: $$\boldsymbol{\mu}(s):=\widehat\omega\big(\nabla_\bullet s,\nabla_\bullet s\big)+\langle\widehat\mu_s \ | \ R^\nabla\rangle -\mathrm d\big(c(\nabla_\bullet\widehat\mu_s)\big) \ .$$ Then, \begin{enumerate}
        \item[(1)] $\boldsymbol\mu(s)$ is a closed $2$-form for any $s\in\pickg$; \item[(2)] $\boldsymbol\mu$ is equivariant with respect to the action of $\Ham(\Sg,\rho)$; \item[(3)] Given a vector field $V\in\Lham(\Sg,\rho)$, and $\gamma_V$ a primitive of $\iota_V\rho$, the differential of the map \begin{align*}
            &\pickg\longrightarrow\R \\ & s\mapsto\int_\Sg\gamma_V\cdot\boldsymbol\mu(s)
        \end{align*}equals $$\boldsymbol\omega_s(\dot s, \liederivative_Vs)=\int_\Sg\widehat\omega_s(\dot s,\liederivative_Vs)\rho \ .$$
    \end{enumerate}
\end{theorem}Before moving on, the meaning of each term in the definition of the moment map must be explained. First notice that $\nabla_\bullet s$ is a section of $T^*\Sg\otimes s^*\big(T^{\text{vert}}P\big(\pick\big)\big)$, hence we set $\widehat\omega\big(\nabla_\bullet s,\nabla_\bullet s\big)$ to be the $2$-form on $\Sg$ given by: $$\widehat\omega\big(\nabla_\bullet s,\nabla_\bullet s\big)(U,V):=\widehat\omega\big(\nabla_Us,\nabla_Vs\big), \ \text{for} \ U,V\in\Gamma(T\Sg) \ ,$$ where in the RHS of last equality we apply the symplectic form $\widehat\omega$ on $T^{\text{vert}}P\big(\pick\big)$ and the wedge product on the $1$-form part. Moreover, the covariant derivative $\nabla_\bullet\widehat\mu_s$ is a section of $T^*\Sg\otimes\mathrm{End}_0(T\Sg)^*$ and we get a $1$-form by performing the following contraction: $$c\big(\nabla_\bullet\widehat\mu_s\big)(v):=\sum_{j=1}^2\langle \nabla_{e_j}\widehat\mu_s \ | \ (v\otimes e^*_j)_0 \rangle, \ \text{for} \ v\in\Gamma(T\Sg) \ , $$ where $\{e_1,e_2\}$ is a local orthonormal frame of $T\Sg$ and $\{e_1^*,e_2^*\}$ is the associated orthonormal dual frame. Finally, the curvature tensor $R^\nabla$ of the torsion-free connection $\nabla$ is defined as: $$R^\nabla\big(U,V\big)W:=\nabla_V\nabla_UW-\nabla_U\nabla_VW-\nabla_{[V,U]}W \ ,$$ for any $U,V,W\in\Gamma(T\Sg)$. Because of the anti-symmetry in the first two entries of $R^\nabla$, it can be considered as a section of $\Omega^2(\Sg)\otimes\mathrm{End}_0(T\Sg)$. For this reason we can contract the endomorphism part of $R^\nabla$ with $\widehat\mu_s$ and obtain the $2$-form on $\Sg$ denoted with $\langle\widehat\mu_s \ | \ R^\nabla\rangle$. Let us recall the following technical result that will be useful later in the construction of our moment map.\begin{lemma}[{\cite[Lemma 13]{donaldson2003moment}}]\label{lem:Donaldsonausiliario} There exists a closed $2$-form $\widehat\omega_{P(\pick)}$ on $P\big(\pick\big)$ such that, for any section $s\in\pickg$, the following holds: $$s^*\widehat\omega_{P(\pick)}=\widehat\omega\big(\nabla_\bullet s,\nabla_\bullet s\big)+\langle\widehat\mu_s \ | \ R^\nabla\rangle \ .$$ In particular, since $\pick$ is contractible, the de-Rham cohomology class of $\momentmap(s)$ in $H^2_{\text{dR}}(\Sg,\R)$ does not depend on the chosen section.

\end{lemma}

\subsection{The moment map on \texorpdfstring{$\pickg$}{pickg}}\label{sec:5.2}
Before proceeding with the genus $g\ge 2$ case, as occurred before, we need to recall some results proved in the torus case. In Section \ref{sec:4.1} we introduced an $\SL(2,\R)$-action on elements $(J,A)\in\pick$, given by $$P\cdot (J,A)=(PJP^{-1}, PA(P^{-1}\cdot)P^{-1}) \ , $$ where $A(P^{-1}\cdot)$ has to be interpreted as the action of $P\in\SL(2,\R)$ via pull-back on the $1$-form part of $A$. \begin{lemma}[\cite{rungi2021pseudo}]\label{lem:SLmomentmaptoruscase}
    The $\SL(2,\R)$-action on $\pick$ is Hamiltonian with respect to $\widehat\ome_f$ with moment map $\widehat\mu:\pick\to\Lsl(2,\R)^*$ given by
\begin{equation}
    \widehat\mu_{(J,A)}(X)=\bigg(1-f\Big(\frac{\vl\vl q\vl\vl^2_J}{2}\Big)\bigg)\tr(JX) \ ,
\end{equation}for all $X\in\Lsl(2,\R)$ and for any choice of a smooth function $f:[0,+\infty)\to(-\infty,0]$ such that: $f(0)=0, f'(t)<0$ for any $t>0$ and $\displaystyle\lim_{t\to+\infty}f(t)=-\infty$. Moreover, $q$ is the $J$-holomorphic cubic differential whose real part is equal to $C=g_JA$.
\end{lemma}
Moving on to the genus $g\ge 2$ case, in Section \ref{sec:4.1} we defined a (formal) family of pseudo-K\"ahler structures $(\g_f, \i,\ome_f)$ on $\pickg$, depending on a choice of a smooth function $f$ as in Lemma \ref{lem:SLmomentmaptoruscase}. If $\widehat\ome_f$ denotes the symplectic form induced on each fibre by an area-preserving isomorphism between $T_x\Sg$ and $\R^2$, then $$(\ome_f)_{(J,A)}\big((\dot J,\dot A),(\dot J', \dot A')\big)=\int_\Sg\widehat\ome_f\big((\dot J,\dot A),(\dot J', \dot A')\big)\rho$$ is obtained from relation (\ref{symplecticformDonaldson}) by integrating fibre-wise the family of symplectic forms introduced in the torus case. Moreover, according to Donaldson's construction of previous section, the group $\Symp_0(\Sg,\rho)$ acts on $\pickg$ preserving $\ome_f$ and the action of $\Ham(\Sg,\rho)$ is Hamiltonian.
\begin{theorem}\label{thm:ourmomentmap}
    The moment map found by Donaldson for the action of $\Ham(\Sg,\rho)$ on $\big(\pickg, \ome_f\big)$ can be expressed as: $$\momentmap(J,A)=-\frac{2}{3}f'\Big(\frac{\vl\vl\tau\vl\vl^2}{2}\Big)\big(\vl\vl\bar\partial\tau\vl\vl^2-\vl\vl\partial\tau\vl\vl^2\big)\rho+2K_J\bigg(f\Big(\frac{\vl\vl \tau\vl\vl^2_J}{2}\Big)-1\bigg)\rho+2i\bar\partial\partial f\Big(\frac{\vl\vl\tau\vl\vl^2}{2}\Big) \ ,$$where $\tau$ is the complex cubic differential whose real part is equal to $C=g_JA$ and where $\bar\partial=\bar\partial_J, \partial=\partial_J$.
\end{theorem}
\begin{proof}
 We will determine the expression for $\momentmap$ using Theorem \ref{thm:donaldson}, hence starting from the explicit description of $\widehat\mu$ given in Lemma \ref{lem:SLmomentmaptoruscase}. As a torsion-free connection $\nabla$ we can choose the Levi-Civita connection with respect to $g_J=\rho(\cdot,J\cdot)$, which clearly satisfies $\nabla\rho=0$. Similar computations can be found in \cite{trautwein2018infinite}, where the functions $f$ and $F$ are chosen to have different properties. \vspace{0.2cm}\newline \textbf{The term $\widehat\omega(\nabla_\bullet s,\nabla_\bullet s)$:}\vspace{0.1cm}\newline Since $\nabla$ is the Levi-Civita connection for $g_J$, we have $\nabla_V J=0$, and the element $\nabla_VA$ is still an End$_0(T\Sg,g_J)$-valued $1$-form for any $V\in\Gamma(T\Sg)$. Now let $\{e_1,e_2\}$ be a local $g_J$-orthonormal frame of $T\Sg$ and let $\{e_1^*,e_2^*\}$ be the dual frame. Then, we get \begin{align*}(\widehat\omega_f)_{(J,A)}\big((0, \nabla_{e_1}A ), (0, \nabla_{e_2}A)\big)&=-\frac{1}{6}f'\Big(\frac{\vl\vl\tau\vl\vl^2}{2}\Big)\langle\nabla_{e_1}A,(\nabla_{e_2}A)J\rangle \ .\end{align*} According to the above observation, the tensors $\nabla_{e_1}A$ and $\nabla_{e_2}A$ can be written as $$\nabla_{e_1}A=(A_1)^1e^*_1+(A_2)^1e^*_2,\qquad \nabla_{e_2}A=(A_1)^2e^*_1+(A_2)^2e^*_2 \ ,$$ where $A_j:=A(e_j)$ for $j=1,2$ and $$(A_1)^k:=\begin{pmatrix}
        a^{1k}_{11} & a^{1k}_{12} \\ a^{1k}_{12} & -a^{1k}_{11}
    \end{pmatrix},\qquad (A_2)^k:=\begin{pmatrix}
        a^{2k}_{11} & a^{2k}_{12} \\ a^{2k}_{12} & -a^{2k}_{11}
    \end{pmatrix}, \ \text{for} \ k=1,2 \ .$$ Using the relation $A_2=A_1J$ and $\nabla_\bullet J=0$, we get $$(A_2)^k:=\begin{pmatrix}
        a^{1k}_{12} & -a^{1k}_{11} \\ -a^{1k}_{11} & -a^{1k}_{12}
    \end{pmatrix} \ .$$ Moreover, by (\ref{def:scalarproductpick2}), we have \begin{align*}
\langle\nabla_{e_1}A,(\nabla_{e_2}A)J\rangle&=\tr\Big((A_1)^1(A_1)^2J+(A_2)^1(A_2)^2J\Big) \\ &=4(a_{11}^{11}a_{12}^{21}-a_{12}^{11}a_{11}^{21}) \ .
    \end{align*} Recalling that $A=g_J^{-1}C$, last formula can be written in terms of $C(\cdot,\cdot,\cdot)$ by using the following relation: $$a_{lm}^{kj}=g_J\big((\nabla_{e_k}A)(e_j)\cdot e_l,e_m\big)=(\nabla_{e_k}C)(e_j,e_l,e_m)=:(\nabla_{e_k}C)_{jlm} \ .$$ In the end, we obtain  \begin{align*}(\widehat\omega_f)_{(J,A)}\big((0, \nabla_{e_1}A ), (0, \nabla_{e_2}A)\big)&=-\frac{2}{3}f'\Big(\frac{\vl\vl\tau\vl\vl_J^2}{2}\Big)\Big((\nabla_{e_1}C)_{111}(\nabla_{e_2}C)_{112}-(\nabla_{e_1}C)_{112}(\nabla_{e_2}C)_{111}\Big) \\ &=-\frac{2}{3}f'\Big(\frac{\vl\vl\tau\vl\vl_J^2}{2}\Big)\Big((\nabla_{e_1}C)_{222}(\nabla_{e_2}C)_{111}-(\nabla_{e_1}C)_{111}(\nabla_{e_2}C)_{222}\Big) \ ,\end{align*} where in the last step we used $C(J\cdot,J\cdot,J\cdot)=-C(J\cdot,\cdot,\cdot)=-C(\cdot,J\cdot,\cdot)=-C(\cdot,\cdot,J\cdot)$. The action of the operators $\partial$ and $\bar\partial$ on $\tau$ are defined as follows: $$(\partial\tau)(v,\cdot,\cdot,\cdot)=\frac{1}{2}\Big(\nabla_v\tau-i\nabla_{Jv}\tau\Big),\qquad(\bar\partial\tau)(v,\cdot,\cdot,\cdot)=\frac{1}{2}\Big(\nabla_v\tau+i\nabla_{Jv}\tau\Big) \ .$$With a fairly long calculation in local coordinates we deduce $$\vl\vl(\bar\partial\tau)(e_1,\cdot,\cdot,\cdot)\vl\vl^2-\vl\vl(\partial\tau)(e_1,\cdot,\cdot,\cdot)\vl\vl^2=(\nabla_{e_1}C)_{222}(\nabla_{e_2}C)_{111}-(\nabla_{e_1}C)_{111}(\nabla_{e_2}C)_{222} \ .$$ Finally, we get $$\widehat\omega(\nabla_\bullet(J,A),\nabla_\bullet(J,A))=(\widehat\omega_f)_{(J,A)}\big((0, \nabla_{e_1}A ), (0, \nabla_{e_2}A)\big)=-\frac{2}{3}f'\Big(\frac{\vl\vl\tau\vl\vl_J^2}{2}\Big)\big(\vl\vl\bar\partial\tau\vl\vl^2-\vl\vl\partial\tau\vl\vl^2\big)\rho \ ,$$ where $\vl\vl\partial\tau\vl\vl^2=\vl\vl\partial\tau(v,\cdot,\cdot,\cdot)\vl\vl^2$ and $\vl\vl\bar\partial\tau\vl\vl^2=\vl\vl\bar\partial\tau(v,\cdot,\cdot,\cdot)\vl\vl^2$ for some unit vector $v$ (the norm is independent of such vector).
    \vspace{0.2cm}\newline\textbf{The term $\langle\widehat\mu_s \ | \ R^\nabla\rangle$:}\vspace{0.1cm}\newline Since $\nabla$ is the Levi-Civita connection for $g_J$, the tensor $R^\nabla$ coincides with the Riemann tensor of $g_J$. A classical computation using a local orthonormal frame shows that $R^\nabla=K_J\rho J$, where $K_J$ is the Gaussian curvature of $g_J$. From Lemma \ref{lem:SLmomentmaptoruscase} we have $\widehat\mu_{(J,A)}(\cdot)=\bigg(1-f\Big(\frac{\vl\vl \tau\vl\vl^2_J}{2}\Big)\bigg)\tr(J\cdot)$, therefore $$\langle\widehat\mu_s \ | \ R^\nabla\rangle=2K_J\bigg(f\Big(\frac{\vl\vl \tau\vl\vl^2_J}{2}\Big)-1\bigg)\rho \ .$$ \vspace{0.2cm}\textbf{The term $\mathrm d\big(c(\nabla_\bullet\widehat\mu_s)\big)$:}\newline Notice that, for any $B\in\Lsl(T\Sg,\rho)\equiv\mathrm{End}_0(T\Sg)$ and for any $v\in\Gamma(T\Sg)$, we have $$\nabla_v\widehat\mu_{(J,A)}\big(B\big)=-\mathrm df\Big(\frac{\vl\vl \tau\vl\vl^2_J}{2}\Big)(v)\tr(JB) \ ,$$ where we used again $\nabla_\bullet J=0$. If $\{e_1,e_2\}$ denotes a local $g_J$-orthonormal frame for $T\Sg$ and $\{e_1^*,e_2^*\}$ denotes its dual frame, we obtain \begin{align*}
    c(\nabla_\bullet\widehat\mu_{(J,A)})(v)&=\langle\nabla_{e_1}\widehat\mu_{(J,A)} \ | \ (v\otimes e_1)^*_0\rangle+\langle\nabla_{e_2}\widehat\mu_{(J,A)} \ | \ (v\otimes e_2)^*_0\rangle \\ &=-\mathrm df\Big(\frac{\vl\vl \tau\vl\vl^2}{2}\Big)(e_1)\tr\big(J(v\otimes e_1)^*_0\big)-\mathrm df\Big(\frac{\vl\vl \tau\vl\vl^2}{2}\Big)(e_2)\tr\big(J(v\otimes e_2)^*_0\big) \\ &=-\mathrm df\Big(\frac{\vl\vl \tau\vl\vl^2}{2}\Big)(e_1)e^*_1\big(Jv\big)-\mathrm df\Big(\frac{\vl\vl \tau\vl\vl^2}{2}\Big)(e_2)e^*_2\big(Jv\big) \\ &=- \bigg(\mathrm d\Big(f\Big(\frac{\vl\vl \tau\vl\vl^2}{2}\Big)\Big)\circ J\bigg)(v) \ .
    \end{align*}In other words $c(\nabla_\bullet\widehat\mu_{(J,A)})=-\mathrm d\Big(f\Big(\frac{\vl\vl \tau\vl\vl^2}{2}\Big)\Big)\circ J$. It is not difficult to show that, for any $\psi\in C^{\infty}(\Sg)$, the following relation holds: $$\mathrm d\big(\mathrm d\psi\circ J\big)=-\Delta_{g_J}\psi=-2i\partial_J\bar\partial_J\psi=2i\bar\partial_J\partial_J\psi \ .$$ In the end, we get $$\mathrm d\big(c(\nabla_\bullet\widehat\mu_{(J,A)})\big)=-\mathrm d\bigg(\mathrm d\Big(f\Big(\frac{\vl\vl \tau\vl\vl^2}{2}\Big)\Big)\circ J\bigg)=-2i\bar\partial\partial f\Big(\frac{\vl\vl\tau\vl\vl^2}{2}\Big) \ .$$
\end{proof}
\begin{corollary}\label{cor:momentmaptilde}
    Let $\rho$ be a fixed area form on $\Sg$, and let $c:=\frac{2\pi\chi(\Sg)}{\mathrm{Area}(\Sg,\rho)}$. Then, the map \begin{align*}\widetilde\momentmap:&\pickg\longrightarrow B^2(\Sg)\subset\Lham(\Sg,\rho)^* \\ &(J,A)\longmapsto\momentmap(J,A)+2c\rho\end{align*} is a moment map for the action of $\Ham(\Sg,\rho)$ on $\big(\pickg,\ome_f\big)$.
\end{corollary}
\begin{proof}
    According to Lemma \ref{lem:Donaldsonausiliario}, the de-Rham cohomology class of the closed $2$-form $\momentmap(J,A)$ does not depend on the choice of the section, and the same is true for its integral over the surface. Hence, if $A=0$, by Gauss-Bonnet Theorem we get $$\int_\Sg\momentmap(J,0)=-2\int_\Sg K_J\rho=-4\pi\chi(\Sg) \ .$$ In particular, the integral of the $2$-form $\momentmap(J,A)+2c\rho$ is equal to zero. This implies that $\widetilde\momentmap$ takes values in the space of exact $2$-forms $B^2(\Sg)$, which is contained in $\Lham(\Sg,\rho)^*$ (see Section \ref{sec:3.1}). Finally, the properties $(i)$ and $(ii)$ in Definition \ref{def:momentmap} continue to hold for $\widetilde\momentmap$ since the additional term $2c\rho$ does not depend on the chosen section.
\end{proof}
In the remaining part of this section we show how, if we assume the additional hypothesis $\bar\partial_J\tau=0$, the moment map $\widetilde\momentmap$ we found is directly related to Wang's equation for hyperbolic affine spheres in $\R^3$ (see Section \ref{sec:2.1}). The idea of proof of the following result is similar to that used in \cite{trautwein2018infinite} for a slightly different moment map.
\begin{theorem}
    Let $(J,A)\in\pickg$ and suppose that $A=g_J^{-1}\Ree(\tau)$ with $\bar\partial_J\tau=0$, then $$\widetilde\momentmap(J,A)=-2e^{F\Big(\frac{\vl\vl\tau\vl\vl^2_J}{2}\Big)}\Big(K_h-\vl\vl \tau\vl\vl_h^2+1\Big)\rho \ , \qquad\text{where} \ \ h:=e^{F\Big(\frac{\vl\vl \tau\vl\vl_J^2}{2}\Big)}g_J \ .$$
\end{theorem}
\begin{proof} If $A=0$, the statement is immediate. Suppose that $A\neq 0$ and define $\lambda$ to be the function $\frac{\vl\vl\tau\vl\vl^2_J}{2}$. Then, outside the zeroes of $A$ it is easy to show that: $$-\frac{i}{\lambda}\bar\partial\lambda\wedge\partial\lambda=\vl\vl\partial\tau\vl\vl^2_J\rho \ , \qquad K_J\rho=-\frac{i}{3}\bar\partial\partial\log(\lambda) \ .$$Let us assume for a moment that the following identity holds: \begin{equation}\label{relationformomentmap}
        2i\bar\partial\bigg(\Big(f'(\lambda)-\frac{f(\lambda)}{3\lambda}\Big)\partial\lambda\bigg)=-\frac{2i}{3\lambda}f'(\lambda)\bar\partial\lambda\wedge\partial\lambda+2f(\lambda)K_J\rho+2i\bar\partial\partial f(\lambda) \ .
    \end{equation} According to Theorem \ref{thm:ourmomentmap} we can write $\widetilde\momentmap(J,A)$ as follows: $$\widetilde\momentmap(J,A)=-\frac{2i}{3\lambda}f'(\lambda)\bar\partial\lambda\wedge\partial\lambda+\frac{2i}{3}\big(1-f(\lambda)\big)\bar\partial\partial\log(\lambda)+2i\bar\partial\partial f(\lambda)+2c\rho$$ In particular, we obtain the following sequence of identities: \begin{align*}
        \widetilde\momentmap(J, A)&=-\frac{2i}{3\lambda}f'(\lambda)\bar\partial\lambda\wedge\partial\lambda+\frac{2i}{3}\big(1-f(\lambda)\big)\bar\partial\partial\log(\lambda)+2i\bar\partial\partial f(\lambda)+2c\rho \\ &=-\frac{2i}{3\lambda}f'(\lambda)\bar\partial\eta\wedge\partial\lambda+\frac{2i}{3}\big(1-f(\lambda)\big)\big(-\frac{3}{i}K_J\rho)+2i\bar\partial\partial f(\lambda)+2c\rho \tag{rel. $(\ref{relationformomentmap})$} \\ &=2i\bar\partial\bigg(\Big(\frac{f(\lambda)}{3\lambda}-f'(\lambda)\Big)\partial\lambda\bigg)-2(K_J-c)\rho \tag{Lemma \ref{lem:combinationoffunctionfandfprime}} \\ &=-2i\bar\partial F'(\lambda)\wedge\partial\lambda-2(K_J-c)\rho \\ &= -2i\bar\partial\partial F(\lambda)-2(K_J-c)\rho \\ &=-2\Big(K_J-\frac{1}{2}\Delta_{g_J}F(\lambda)-c\Big)\rho \ . 
    \end{align*} Now, if $h$ denotes the Riemannian metric on $\Sg$ conformal to $g_J$ with conformal factor equal to $e^{F(\lambda)}$, we get $$K_h=e^{-F(\lambda)}\Big(K_J-\frac{1}{2}\Delta_{g_J}F(\lambda)\Big) \ .$$ On the other hand, using the functional equation (\ref{functionalequationF}) satisfied by $F$, we have \begin{align*}
        \widetilde\momentmap(J,A)&=-2\Big(K_J-\frac{1}{2}\Delta_{g_J}F(\lambda)-c\Big)\rho \\ &=-2\Big(e^{F(\lambda)}K_h-c\Big)\rho \\ &=-2e^{F(\lambda)}\big(K_h-\vl\vl\tau\vl\vl^2_Je^{-3F(\lambda)}+1\big) \\ &=-2e^{F(\lambda)}\big(K_h-\vl\vl\tau\vl\vl^2_h+1\big)\rho \ .
    \end{align*} In order to finish the proof, it only remains to show that relation (\ref{relationformomentmap}) holds, which stems from the following identities:
    $$2fK_J\rho=\frac{2i}{3\lambda}f\Big(\frac{1}{\lambda}\bar\partial\wedge\partial\lambda-\bar\partial\partial\lambda\Big), \quad 2i\bar\partial\partial f=2i\Big(f''\bar\partial\lambda\wedge\partial\lambda+f'\bar\partial\partial\lambda\Big) \ .$$ This ends the proof outside the zeroes of $A$, which is a finite set in $\Sg$. Thus, the statement follows by continuity of the expression.
\end{proof}

\begin{corollary}\label{cor:zerolocusmomentmap}
    Let $(J,A)\in\pickg$. Then $\widetilde\momentmap(J,A)=0$ and $\mathrm d^\nabla A=0$ if and only if $(J,A)\in\haffrhozerotilde$.
\end{corollary}
\begin{proof}
    Recall from Section \ref{sec:2.3} that $\haffrhozerotilde$ is the space of pairs $(J,A)$ such that $\big(h=e^Fg_J, A\big)$ satisfies (\ref{Gausscodazzi}) (see also Remark \ref{rem:conformalinvarianceHS}). By Theorem \ref{thm:picktensor} we know that, up to contraction with the metric, $A$ is the real part of a $J$-holomorphic cubic differential $\tau$. Finally, the above theorem implies that $\widetilde\momentmap(J,A)=0$ if and only if $K_h-\vl\vl \tau\vl\vl^2_h=-1$.
\end{proof}
\subsection{The symplectic quotient}\label{sec:5.3}
Here we explain how the use of symplectic reduction allows us to determine (in part) the system of differential equations (\ref{differentialequations}) defining $W_{(J,A)}$, hence those characterizing the tangent space to the $\SL(3,\R)$-Hitchin component (see Remark \ref{rem:equationsHitchincomponent}). Following in parallel the construction done for Teichm\"uller space in Section \ref{sec:3.2}, the idea is to induce our symplectic form $\ome_f$ from the ambient space $\pickg$ to the quotient of $\widetilde\momentmap^{-1}(0)$ by the group $\Ham(\Sg,\rho)$. On the other hand, there are two major differences with the case of $\mathcal T(\Sg)$: the first is that the infinite-dimensional space $\haffrhozerotilde$ is cut by two equations (see Corollary \ref{cor:zerolocusmomentmap}) and, only one of them, has an interpretation as a moment map. In particular, we have to look at the space $\widetilde\momentmap^{-1}(0)\cap\mathcal{M}_{\mathrm C}$ modulo $\Ham(\Sg,\rho)$; the second is that once we induce the symplectic form on the quotient, the pairing $\ome_f(\i\cdot,\cdot)=\g_f$ gives rise to a pseudo-Riemannian metric, and this generates additional difficulties since one is intent to identify the space $W_{(J,A)}$ with the $\g_f$-orthogonal to the $\Ham(\Sg,\rho)$-orbit.\newline
Our moment map $\widetilde\momentmap$ has values in the space of exact $2$-forms on the surface, which is contained in the dual Lie algebra of the Hamiltonian group (see Corollary \ref{cor:momentmaptilde}). After recalling two technical lemmas, in Proposition \ref{prop:differentialourmomentmap} we compute a primitive of the differential of the moment map (still with values in the exact 2-forms). This will allow us, in Proposition \ref{prop:integrazioneperparti}, to perform a highly non-trivial integration by parts, which will be useful later in discussing the (possible) presence of degenerate vectors for the pseudo-metric away from the Fuchsian locus. Then, 
 with Proposition \ref{prop:Winsidehaffrhozero} we prove the inclusion of $W_{(J,A)}$ inside the tangent to $\haffrhozerotilde$, the discussion of which had been left hanging by Section \ref{sec:4.4}. Finally, inspired by the K\"ahler reduction of Teichm\"uller space, in Theorem \ref{thmG} we are able to characterize $W_{(J,A)}$ as the largest subspace in $T_{(J,A)}\haffrhozerotilde$ that is both $\g_f$-orthogonal to the orbit and invariant under the action of the complex structure $\i$. \newline The statements and the proofs of Proposition \ref{prop:differentialourmomentmap} and Proposition \ref{prop:integrazioneperparti} are inspired by the analogous counterparts in the anti-de Sitter case (\cite[Proposition 6.10 and 6.12]{mazzoli2021parahyperkahler}). Despite that, the presence of the $1$-form part in the tensor $A$ created additional problems during the development of the proofs, which will be highlighted throughout. We first recall two technical lemmas that will be useful further on.

\begin{lemma}[{\cite[Lemma 4.16]{mazzoli2021parahyperkahler}}]\label{lem:divergenzaendomorfismo}
    Let $B$ be a trace-less endomorphism of $T\Sg$, then $$(\nabla_XB)Y-(\nabla_YB)X=(\dive_gB)(Y)X-(\dive_gB)(X)Y \ .$$
\end{lemma}
\begin{lemma}[{\cite[Lemma 4.15]{mazzoli2021parahyperkahler}}]\label{lem:firstordervariationlevicivita}
    Let $\dot J\in T_J\almostg$ be an infinitesimal variation of a complex structure on $\Sg$. If $\dot\nabla$ denotes the first order variation of the Levi-Civita connection of $g_J=\rho(\cdot,J\cdot)$ along $\dot J$, then the following holds: \begin{equation}
        \dot\nabla_XY=-\frac{1}{2}\big((\dive\dot J)(X)JY+J(\nabla_X\dot J)Y\big) \ ,
    \end{equation}for every tangent vector fields $X,Y$ on $\Sg$.
\end{lemma}
\begin{proposition}\label{prop:differentialourmomentmap}
    For every $(J,A)\in\pickg$ such that $A=g_J^{-1}\Ree(\tau)$ with $\bar\partial_J\tau=0$, and for every tangent vector $(\dot J,\dot A)\in T_{(J,A)}\pickg$ we have \begin{equation}
        \mathrm d\widetilde\momentmap(\dot J,\dot A)=\mathrm d\Big((f-1)\dive_g\dot J+\mathrm df\circ\dot J+\mathrm d\dot f\circ J-\frac{f'}{6}\beta\Big) \ ,
    \end{equation}where $f,f',\dot f$ are evaluated at $\frac{\vl\vl\tau\vl\vl_J^2}{2}$ and $\beta$ is the $1$-form defined as $\beta(V):=\langle\dot A_0, (\nabla_V A)J\rangle$.
\end{proposition}
\begin{proof}
    The most intricate part of the proof is encompassed in showing the following identity: \begin{equation}\label{difficultidentitydmomentmap}
        \bigg(-\frac{2}{3}f'\big(\vl\vl\bar\partial\tau\vl\vl^2-\vl\vl\partial\tau\vl\vl^2\big)\rho\bigg)'=-\mathrm d\Big(\frac{f'}{6}\beta\Big)-2\dot fK_J\rho+\mathrm{d}(f-1)\wedge\dive_g\dot J \ ,
    \end{equation}where the derivative is taken with respect to $(J,A)$ along tangent directions $(\dot J,\dot A)$. Let us assume for a moment that (\ref{difficultidentitydmomentmap}) holds and let us prove the formula stated in the theorem. In fact, the other terms in $\widetilde\momentmap$ (see Theorem \ref{thm:ourmomentmap}) are easier to handle \begin{align*}\big(2(f-1)K_J\rho\big)'&=2\dot fK_J\rho+2(f-1)\mathrm dK_J(\dot J)\rho \tag{Proposition \ref{differentialofcurvature}} \\ &=2\dot fK_J\rho+(f-1)\mathrm d\big(\dive_g\dot J\big) \ .\end{align*} Moreover, $$\big(2i\bar\partial\partial f\big)'=-\big(\Delta_{g_J}f\big)'\rho=\mathrm d\big((\mathrm df\circ J)'\big)=\mathrm d\big(\mathrm d\dot f\circ J+\mathrm df\circ\dot J\big) \ .$$ Combining these formulas, we get the desired expression for the moment map\begin{align*}
        \mathrm d\widetilde\momentmap(\dot J,\dot A)&=-\mathrm d\Big(\frac{f'}{6}\beta\Big)-2\dot fK_J\rho+\mathrm{d}f\wedge\dive_g\dot J+\mathrm d\big(\mathrm d\dot f\circ J+\mathrm df\circ\dot J\big)+2\dot fK_J\rho+(f-1)\mathrm d\big(\dive_g\dot J\big) \\ &=\mathrm d\Big((f-1)\dive_g\dot J+\mathrm df\circ\dot J+\mathrm d\dot f\circ J-\frac{f'}{6}\beta\Big) \ .
    \end{align*}Now let us focus on proving relation (\ref{difficultidentitydmomentmap}). As was shown in Theorem \ref{thm:ourmomentmap}, we know that: \begin{equation}\label{symplecticformmomentmap}
        -\frac{2}{3}f'\big(\vl\vl\bar\partial\tau\vl\vl^2-\vl\vl\partial\tau\vl\vl^2\big)=(\widehat\omega_f)_{(J,A)}\big((0, \nabla_{e_1}A ), (0, \nabla_{e_2}A)\big) \ ,
    \end{equation}for any choice of a local frame $\{e_1,e_2\}$ such that $\rho(e_1,e_2)=1$. Therefore, we can compute the following derivative: \begin{align*}
        \bigg(-\frac{2}{3}f'\big(\vl\vl\bar\partial\tau\vl\vl^2-\vl\vl\partial\tau\vl\vl^2\big)\bigg)'&=\bigg((\widehat\omega_f)_{(J,A)}\big((0, \nabla_{e_1}A ), (0, \nabla_{e_2}A)\big)\bigg)' \\ &= \bigg(-\frac{f'}{6}\langle\nabla_{e_1}A, (\nabla_{e_2}A)J\rangle\bigg)' \ ,
    \end{align*}where in the second step we used that, for any $i=1,2$, the endomorphism part of $\nabla_{e_i}A$ is trace-less and $g_J$-symmetric. In order to simplify the computation of the derivative, let us make some preliminary observations. Since equation (\ref{symplecticformmomentmap}) is true for any unit volume local frame $\{e_1, e_2\}$, we can further assume that it is $g_J$-orthonormal and does not change as $J$ varies along tangent directions. Moreover, the terms corresponding to variations $\dot J$ make no contributions as $\tr(\nabla_{e_1}A\nabla_{e_2}A\dot J)=0$ (see (\ref{def:scalarproductpick2})). This allows us to reduce the study of the derivative to only the following terms: \begin{equation}\begin{aligned}\label{firstderivativemomentmap}
        \bigg(-\frac{2}{3}f'\big(\vl\vl\bar\partial\tau\vl\vl^2-\vl\vl\partial\tau\vl\vl^2\big)\bigg)'=&-\frac{f''}{24}\langle A,\dot A_0\rangle\langle\nabla_{e_1}A, (\nabla_{e_2}A)J\rangle+ \\ &-\frac{f'}{6}\langle(\nabla_{e_1}A)', (\nabla_{e_2}A)J\rangle -\frac{f'}{6}\langle\nabla_{e_1}A, (\nabla_{e_2}A)'J\rangle \ ,
    \end{aligned}\end{equation}and we expressed the first order variation of $f'$ as $\frac{f''}{4}\langle A,\dot A_0\rangle$ (see Lemma 3.22 in \cite{rungi2021pseudo}). At this point, using Lemma \ref{lem:firstordervariationlevicivita}, we can obtain an expression for $(\nabla_XA)'$. In fact, $$(\nabla_XA)'=\dot\nabla_XA+\nabla_XA'$$ and we can compute \begin{align*}
        (\dot\nabla_XA)(Y)Z&=\dot\nabla_X\big(A(Y)Z\big)-A(\dot\nabla_XY)Z-A(Y)\dot\nabla_XZ \\ &=\frac{1}{2}\Big(-(\dive\dot J)(X)JA(Y)Z-J(\nabla_X\dot J)A(Y)Z+(\dive\dot J)(X)A(JY)Z+ \\ & \ \ \ +A\big(J(\nabla_X\dot J)Y\big)Z+(\dive\dot J)(X)A(Y)JZ+A(Y)J(\nabla_X\dot J)Z\Big) \\ &=\frac{1}{2}\Big(3(\dive\dot J)(X)A(Y)JZ+A\big(J(\nabla_X\dot J)Y\big)Z+A(Y)J(\nabla_X\dot J)Z-J(\nabla_X\dot J)A(Y)Z\Big)
    \end{align*}where we used $A(JY)Z=A(Y)JZ$ and $A(Y)JZ=-JA(Y)Z$. As for the term involving the derivative of $A$, we first notice that $A'=J\dot JA+\dot A$, hence \begin{align*}
        (\nabla_XA')&=J\nabla_X\dot JA+J\dot J\nabla_XA+\nabla_X\dot A_0+\nabla_X\dot A_{\text{tr}} \\ &=J\nabla_X\dot JA+J\dot J\nabla_XA+\nabla_X\dot A_0+\frac{1}{2}\Big(\tr(\nabla_X\dot JJA)+\tr(\dot JJ\nabla_XA)\Big) \ .
    \end{align*}Now, choosing $X=e_1$ and observing that the two trace terms in $(\nabla_{e_1}A')$ and the four elements $AJ\nabla_{e_1}\dot J, J\nabla_{e_1}\dot JA, J\nabla_{e_1}\dot JA, J\dot J\nabla_{e_1}A$ are zero once they pair with $(\nabla_{e_2}A)J$ using the scalar product (\ref{def:scalarproductpick2}), we get \begin{align*}
        -\frac{f'}{6}\langle(\nabla_{e_1}A)',(\nabla_{e_2}A)J\rangle&=-\frac{f'}{4}(\dive\dot J)(e_1)\langle AJ,(\nabla_{e_2}A)J\rangle-\frac{f'}{6}\langle\nabla_{e_1}\dot A_0,(\nabla_{e_2}A)J\rangle \\ & \ \ \ -\frac{f'}{12}\langle A\big(\nabla_{e_1}\dot J\cdot\big)J, \nabla_{e_2}AJ\rangle \ .
    \end{align*} Moreover, since $\langle\nabla_{e_1}A, (\nabla_{e_2}A)'J\rangle=-\langle(\nabla_{e_1}A)J, (\nabla_{e_2}A)'\rangle$, performing a similar computation as above, we obtain \begin{align*}
        \frac{f'}{6}\langle(\nabla_{e_2}A)',(\nabla_{e_1}A)J\rangle&=\frac{f'}{4}(\dive\dot J)(e_2)\langle AJ,(\nabla_{e_1}A)J\rangle+\frac{f'}{6}\langle\nabla_{e_2}\dot A_0,(\nabla_{e_1}A)J\rangle \\ & \ \ \ +\frac{f'}{12}\langle A\big(\nabla_{e_2}\dot J\cdot\big)J, \nabla_{e_1}AJ\rangle \ .
    \end{align*}Combining everything together in (\ref{firstderivativemomentmap}), we have
    \begin{equation}\begin{aligned}\label{derivativestep2}
        \Big(-\frac{2}{3}f'\big(\vl\vl\bar\partial\tau\vl\vl^2-\vl\vl\partial\tau\vl\vl^2\big)\Big)'=&-\frac{f''}{24}\langle A,\dot A_0\rangle\langle\nabla_{e_1}A, (\nabla_{e_2}A)J\rangle+\frac{f'}{4}\Big((\dive\dot J)(e_2)\langle\nabla_{e_1}A, A\rangle \\ &-(\dive\dot J)(e_1)\langle AJ,(\nabla_{e_2}A)J\rangle\Big)-\frac{f'}{6}\Big(\langle\nabla_{e_1}\dot A_0,(\nabla_{e_2}A)J\rangle \\ &+\langle\nabla_{e_1}A, (\nabla_{e_2}\dot A_0)J)\rangle\Big)+\frac{f'}{12}\Big(\langle\nabla_{e_1}A, A\big(\nabla_{e_2}\dot J\cdot\big)\rangle \\ &-\langle A\big(\nabla_{e_1}\dot J\cdot\big), \nabla_{e_2}A\rangle\Big) \ ,
    \end{aligned}\end{equation} where we used, again, the symmetry and the compatibility of the scalar product with $J$ (see (\ref{rel:cpxstructurescalarprod1})). Regarding the divergence term found in (\ref{derivativestep2}), it can be elaborated as follows: \begin{align*}
        \frac{f'}{4}\Big((\dive\dot J)(e_2)\langle\nabla_{e_1}A, A\rangle-(\dive\dot J)(e_1)\langle A,\nabla_{e_2}A\rangle\Big)&=-\big(\dive\dot J\wedge\mathrm d f\big)(e_1, e_2) \\ &=\big(\mathrm d(f-1)\wedge\dive\dot J\big)(e_1,e_2) \ ,
    \end{align*}where we used $\mathrm d f=\frac{f'}{4}\langle A,\nabla_\bullet A\rangle$. Comparing relation (\ref{difficultidentitydmomentmap}) with (\ref{derivativestep2}), the proof is complete if we show that 
        \begin{align*}
          \boldsymbol{(i)} \  -\mathrm d\Big(\frac{f'}{6}\beta\Big)(e_1,e_2)=&-\frac{f''}{24}\langle A,\dot A_0\rangle\langle\nabla_{e_1}A, (\nabla_{e_2}A)J\rangle+2\dot fK_J+ \\ &-\frac{f'}{6}\Big(\langle\nabla_{e_1}\dot A_0,(\nabla_{e_2}A)J\rangle +\langle\nabla_{e_1}A, (\nabla_{e_2}\dot A_0)J)\rangle\Big) \ ,
        \end{align*}
       $$\boldsymbol{(ii)} \ \ \langle\nabla_{e_1}A, A\big(\nabla_{e_2}\dot J\cdot\big)\rangle-\langle A\big(\nabla_{e_1}\dot J\cdot\big), \nabla_{e_2}A\rangle=0 \ .\qquad\qquad\qquad\qquad\qquad$$\newline\textbf{Proof of relation $\boldsymbol{(i)}$}\newline First notice that if $A=0$ then the relation is clearly satisfied. Suppose $A$ is not identically zero, then $$-\mathrm d\Big(\frac{f'}{6}\beta\Big)=-\frac{1}{6}\mathrm df'\wedge\beta-\frac{f'}{6}\mathrm d\beta=-\frac{f''}{24}\langle A,\nabla_{\bullet}A\rangle\wedge\beta-\frac{f'}{6}\mathrm d\beta \ .$$ Regarding the differential of $\beta(\bullet)=\langle\dot A_0,(\nabla_\bullet A)J\rangle$ we get \begin{align*}
           \mathrm d\beta(e_1,e_2)&=e_1\cdot\big(\langle\dot A_0, (\nabla_{e_2}A)J\rangle\big)-e_2\cdot\big(\langle\dot A_0, (\nabla_{e_1}A)J\rangle\big)-\langle\dot A_0, (\nabla_{[e_1,e_2]}A)J\rangle \\ &=\langle\nabla_{e_1}\dot A_0,(\nabla_{e_2}A)J\rangle-\langle\nabla_{e_2}\dot A_0,(\nabla_{e_1}A)J\rangle+\langle\dot A_0,\big(\nabla_{e_1}\nabla_{e_2}A-\nabla_{e_2}\nabla_{e_1}A-\nabla_{[e_1,e_2]}A\big)J\rangle \\ &=\langle\nabla_{e_1}\dot A_0,(\nabla_{e_2}A)J\rangle-\langle\nabla_{e_2}\dot A_0,(\nabla_{e_1}A)J\rangle-3K_J\langle A,\dot A_0\rangle \ , 
       \end{align*}where the last equality follows from $R^\nabla(e_1,e_2)A=3K_JAJ$ since $R^\nabla(e_1,e_2)=\nabla_{e_1}\nabla_{e_2}-\nabla_{e_2}\nabla_{e_1}-\nabla_{[e_1,e_2]}$. Thus, \begin{equation}
           -\frac{f'}{6}\mathrm d\beta(e_1,e_2)=-\frac{f'}{6}\Big(\langle\nabla_{e_1}\dot A_0,(\nabla_{e_2}A)J\rangle-\langle\nabla_{e_2}\dot A_0,(\nabla_{e_1}A)J\rangle\Big)+2\dot fK_J \ .
       \end{equation}Concerning the other therm, we need to prove that $$\big(\langle A,\nabla_{\bullet}A\rangle\wedge\beta\big)(e_1, e_2)=\langle A,\dot A_0\rangle\langle\nabla_{e_1}A, (\nabla_{e_2}A)J\rangle \ .$$ Notice that, for any $p\in\Sg$ outside the zeroes of $A$, the elements $(A_1)_p:=\big(A(e_1)\big)_p$ and $(A_1J)_p:=\big(A(e_1)J\big)$ form a basis for the space of $g_J$-symmetric and trace-less endomorphisms of $T_p\Sg$. In particular, using the scalar product $\langle\cdot,\cdot\rangle$ we can write $$\dot A_0=\frac{1}{\vl\vl A\vl\vl^2}\Big(\langle A,\dot A_0\rangle A+\langle AJ,\dot A_0\rangle AJ\Big),\quad \nabla_{e_1}A=\frac{1}{\vl\vl A\vl\vl^2}\Big(\langle A,\nabla_{e_1}A\rangle A+\langle AJ,\nabla_{e_1}A\rangle AJ\Big) \ .$$ Replacing these identities in the previous equation, we obtain \begin{align*}
          \big(\langle A,\nabla_{\bullet}A\rangle\wedge\beta\big)(e_1, e_2)&=\langle A, \nabla_{e_1}A\rangle\langle\dot A_0,(\nabla_{e_2}A)J\rangle-\langle A,\nabla_{e_2}A\rangle\langle\dot A_0, (\nabla_{e_1}A)J\rangle \\ &=\frac{\langle A, \nabla_{e_1}A\rangle}{\vl\vl A\vl\vl^2}\Big(\langle\dot A_0, A\rangle\langle A,(\nabla_{e_2}A)J\rangle+\langle\dot A_0, AJ\rangle\langle AJ,(\nabla_{e_2}A)J\rangle\Big) \\ & \ \ \ -\frac{\langle A, \nabla_{e_2}A\rangle}{\vl\vl A\vl\vl^2}\Big(\langle\dot A_0, A\rangle\langle A,(\nabla_{e_1}A)J\rangle+\langle\dot A_0, AJ\rangle\langle AJ,(\nabla_{e_1}A)J\rangle\Big) \\ &=\frac{\langle\dot A_0, A\rangle}{\vl\vl A\vl\vl^2}\Big(\langle A,\nabla_{e_1}A\rangle\langle A,(\nabla_{e_2}A)J\rangle-\langle A, \nabla_{e_2}A\rangle\langle A,(\nabla_{e_1}A)J\rangle\Big) \\ &=\frac{\langle\dot A_0, A\rangle}{\vl\vl A\vl\vl^2}\cdot\big\langle\langle A,\nabla_{e_1}A\rangle A+\langle AJ,\nabla_{e_1}A\rangle AJ; (\nabla_{e_2}A)J\big\rangle \\ &=\langle A,\dot A_0\rangle\langle\nabla_{e_1}A, (\nabla_{e_2}A)J\rangle \ .
       \end{align*}Since the relation is true on the complement of a finite set in $\Sg$ (the zeroes of $A$), it extends on the whole surface by continuity of the expression.\vspace{0.1cm}\newline
       \textbf{Proof of relation $\boldsymbol{(ii)}$}\newline
       As explained at the beginning of the section, the presence of the $1$-form part in the tensor $A$ generates further difficulties. In fact, one has to deal with terms of the form $A\big(\nabla_{e_i}\dot J\cdot\big)$ which do not appear in the anti-de Sitter case. First of all notice that if $A$ is identically zero, then the relation is clearly satisfied. Hence, let us assume that this is not the case. In the following, we will use the notations introduced in the proof of Theorem \ref{thm:ourmomentmap}. Namely, $$\nabla_{e_1}A=(A_1)^1e^*_1+(A_2)^1e^*_2,\qquad \nabla_{e_2}A=(A_1)^2e^*_1+(A_2)^2e^*_2 \ ,$$ where $A_j:=A(e_j)$ for $j=1,2$ and since $A=g_J^{-1}C$, we have \begin{align*}& A_1=\begin{pmatrix}
       C_{111} & C_{112} \\ C_{112} & -C_{111}
       \end{pmatrix},\qquad A_2=\begin{pmatrix}
           C_{112} & -C_{111} \\ -C_{111} & -C_{112}
       \end{pmatrix}, \\[1em] &(A_1)^k:=\begin{pmatrix}
        (\nabla_kC)_{111} & (\nabla_kC)_{112} \\ (\nabla_kC)_{112} & -(\nabla_kC)_{111}
    \end{pmatrix},\quad (A_2)^k:=\begin{pmatrix}
        (\nabla_kC)_{112} & -(\nabla_kC)_{111} \\ -(\nabla_kC)_{111} & -(\nabla_kC)_{112}
    \end{pmatrix}, \ k=1,2 \\[1em] &(\nabla_kC)_{jlm}:=(\nabla_{e_k}C)(e_j,e_l,e_m)=g_J\big((\nabla_{e_k}A)(e_j)\cdot e_l,e_m\big) \ .\end{align*}By assumption, $C$ is the real part of a holomorphic cubic differential and, this is equivalent (see Theorem \ref{thm:picktensor}), to require that $(\nabla_{JX}A)(\cdot)=(\nabla_XA)(J\cdot)$ for any vector field $X$ on the surface. In particular, we obtain the following additional relations \begin{equation}
        (\nabla_2C)_{111}=(\nabla_1C)_{112},\qquad (\nabla_2C)_{112}=-(\nabla_1C)_{111} \ .
\end{equation}The next step is to write explicitly, in a similar way, the tensors $$A\big(\nabla_{e_1}\dot J\cdot\big)=(\widetilde A_1)^1e_1^*+(\widetilde A_2)^1e_2^*,\qquad A\big(\nabla_{e_2}\dot J\cdot\big)=(\widetilde A_1)^2e_1^*+(\widetilde A_2)^2e_2^* \ . $$ For any $p\in\Sg$ outside the zeroes of $A$, the elements $A_1$ and $A_2=A_1J$ form a basis for the space of $g_J$-symmetric and trace-less endomorphisms of $T_p\Sg$. In particular, both $\nabla_{e_1}\dot J$ and $\nabla_{e_2}\dot J$ can be written in this basis as \begin{align*}
    &\nabla_{e_1}\dot J=\frac{1}{\tr(A_1^2)}\Big(\tr(\nabla_{e_1}\dot JA_1)A_1+\tr(A_1J\nabla_{e_1}\dot J)A_2\Big) \\ &\nabla_{e_2}\dot J=\frac{1}{\tr(A_1^2)}\Big(\tr(\nabla_{e_2}\dot JA_1)A_1+\tr(A_1J\nabla_{e_2}\dot J)A_2\Big) \ .
\end{align*} This new form of the endomorphisms allows us to compute their values on the $g_J$-orthonormal basis of the tangent to the surface \begin{align*}
    \nabla_{e_1}\dot J\cdot e_1&=\frac{1}{\tr(A_1^2)}\Big(\tr(\nabla_{e_1}\dot JA_1)C_{111}+\tr(A_1J\nabla_{e_1}\dot J)C_{112}\Big)e_1 \\ & \ \ \ +\frac{1}{\tr(A_1^2)}\Big(\tr(\nabla_{e_1}\dot JA_1)C_{112}-\tr(A_1J\nabla_{e_1}\dot J)C_{111}\Big)e_2 
    \end{align*}
   \begin{align*} \nabla_{e_1}\dot J\cdot e_2&=\frac{1}{\tr(A_1^2)}\Big(\tr(\nabla_{e_1}\dot JA_1)C_{112}-\tr(A_1J\nabla_{e_1}\dot J)C_{111}\Big)e_1 \\ & \ \ \ -\frac{1}{\tr(A_1^2)}\Big(\tr(\nabla_{e_1}\dot JA_1)C_{111}+\tr(A_1J\nabla_{e_1}\dot J)C_{112}\Big)e_2 
\end{align*}and the same calculation can be done for $\nabla_{e_2}\dot J$. In particular, we obtain\begin{align*}&(\widetilde A_1)^k=\frac{1}{\tr(A_1^2)}\Big(\tr(\nabla_{e_k}\dot JA_1)\big(C_{111}A_1+C_{112}A_2\big)+\tr(A_1J\nabla_{e_k}\dot J)\big(C_{112}A_1-C_{111}A_2\big)\Big), \\[1em]  & (\widetilde A_2)^k=\frac{1}{\tr(A_1^2)}\Big(\tr(\nabla_{e_k}\dot JA_1)\big(C_{112}A_1-C_{111}A_2\big)-\tr(A_1J\nabla_{e_k}\dot J)\big(C_{111}A_1+C_{112}A_2\big)\Big) \ . \end{align*} To conclude, we notice that $\tr(A_1^2)=2\big(C_{111}^2+C_{112}^2\big)$, hence \begin{align*}
    \langle\nabla_{e_1}A, A\big(\nabla_{e_2}\dot J\cdot\big)\rangle-\langle A\big(\nabla_{e_1}\dot J\cdot\big), \nabla_{e_2}A\rangle&=\tr\Big((A_1)^1(\widetilde A_1)^2+(A_2)^1(\widetilde A_2)^2\Big) \\ & \ \ \ -\tr\Big((A_1)^2(\widetilde A_1)^1-(A_2)^2(\widetilde A_2)^1\Big) \\ &=0 \ .
\end{align*}Since the relation is true on the complement of a finite set in $\Sg$ (the zeroes of $A$), it extends on the whole surface by continuity of the expression.
\end{proof}

\begin{remark}\label{rem:fixedprimitive}
In analogy with what happens for the $\PSL(2,\R)\times\PSL(2,\R)$ case (see \cite[Remark 6.11]{mazzoli2021parahyperkahler}), we fix a primitive of $\mathrm{d}\widetilde\momentmap$ found in Proposition \ref{prop:differentialourmomentmap} and we consider the linear map $L_{(J,A)}:T_{(J,A)}\pickg\to\Omega^1(\Sg)/B^1(\Sg)\subset\Lsymp(\Sg,\rho)^*$ which associates to each tangent vector $(\dot J,\dot A)$ the above primitive (modulo exact $1$-forms). %It should be noted that, a priori, the kernel of $L_{(J,A)}$ is strictly contained in $T_{(J,A)}\widetilde\momentmap^{-1}(0)$ and may not coincide. 
With an abuse of notation we will denote this primitive by $\mathrm{d}\widetilde\momentmap(\dot J,\dot A)\equiv L_{(J,A)}(\dot J,\dot A)$. 
\end{remark}

\begin{proposition}\label{prop:integrazioneperparti}
    Let $(J,A)\in\haffrhozerotilde$, then for every $(\dot J,\dot A)\in T_{(J,A)}\pickg$ and for every symplectic vector field $V$, we have \begin{equation}\label{integrationbypart}
        \ome_f\big((\liederivative_VJ, g_J^{-1}\liederivative_VC); (\dot J,\dot A)\big)=-\langle\mathrm d\widetilde\momentmap(\dot J,\dot A) \ | \ V\rangle_{\Lsymp}
    \end{equation}
\end{proposition}
\begin{proof}Before we begin the proof of the formula stated in the proposition, let us make some preliminary remarks. For any vector field $X$ on the surface, let use define the operator $M_X:\Gamma(T\Sg)\to\Gamma(T\Sg)$ as $M_X(Y):=\nabla^g_YX$, where $\nabla^g$ is the Levi-Civita connection with respect to $g\equiv g_J=\rho(\cdot,J\cdot)$. The endomorphism $M_X$ can be decomposed as $$M_X=\frac{\tr(M_X)}{2}\mathds 1-\frac{\tr(JM_X)}{2}J+M_X^{\mathrm s} \ ,$$ where the first term is the trace part, the second one is the $g$-skew-symmetric part, and the third one is the $g$-symmetric and trace-less part. If $X=V$ is a $\rho$-symplectic vector field, then the trace part of $M_V$ vanishes. Since $J$ is $\nabla^g$-parallel, we have $M_{JV}=JM_V$ and its decomposition is  given by \begin{equation}M_{JV}=JM_V=\frac{\tr(JM_V)}{2}\mathds 1+0+JM_V^{\mathrm s} \ .\end{equation} In particular, the $g$-skew-symmetric part of $M_{JV}$ vanishes and $JM_V^{\mathrm s}=M_{JV}^{\mathrm s}$. Recall that (see (\ref{liederivative-g-minusone-C})) we found the following formula for the Lie derivative of $C$ expressed in terms of the tensor $A$ $$(g^{-1}\liederivative_VC)(\cdot)=(\nabla_VA)(\cdot)+A(M_V\cdot)+A(\cdot)M_V+M_V^*A(\cdot) \ ,$$ which can be re-written using the decomposition of $M_V$ found above \begin{equation}\label{liederivativeCsenzatraccia}(g^{-1}\liederivative_VC)(\cdot)=\underbrace{(\nabla_VA)(\cdot)-\frac{3}{2}\tr(JM_V)AJ+A\big(M_V^{\mathrm s}\cdot\big)}_{\text{trace-less}}+\underbrace{AM_V^{\mathrm s}+M_V^{\mathrm s}A}_{\text{trace part}} \ .\end{equation}At this point, we can compute the symplectic form \begin{align*}
    (\ome_f)\big((\liederivative_VJ,g^{-1}\liederivative_VC), (\dot J,\dot A)\big)&=\int_\Sg\Big((f-1)\langle\liederivative_VJ,J\dot J\rangle-\frac{f'}{6}\langle(g^{-1}\liederivative_VC)_0, (\dot AJ+A\dot J)_0\rangle \\ & \ \ \ +\frac{f'}{12}\langle(g^{-1}\liederivative_VC)_{\text{tr}},(\dot AJ+A\dot J)_{\text{tr}}\rangle\Big)\rho \\ &=\int_\Sg\Big((f-1)\langle\liederivative_VJ,J\dot J\rangle-\frac{f'}{6}\langle \nabla_VA-\frac{3}{2}\tr(JM_V)AJ,\dot A_0J\rangle \\ & \ \ \ -\frac{f'}{6}\big(\langle A\big(M_V^{\mathrm s}\cdot\big),\dot A_0J\rangle-\frac{1}{2}\langle AM_V^{\mathrm{s}}+M_V^{\mathrm{s}}A, (\dot AJ+A\dot J)_{\text{tr}}\rangle\big)\Big)\rho
\end{align*}In order to simplify the third and fourth term in the integral, we make us of the following identity which will be proven at the end \begin{equation}\label{symplecticform3-4zero}
    \langle A\big(M_V^{\mathrm s}\cdot\big),\dot A_0J\rangle-\frac{1}{2}\langle AM_V^{\mathrm{s}}+M_V^{\mathrm{s}}A, (\dot AJ+A\dot J)_{\text{tr}}\rangle=0 \ .
\end{equation}Regarding the first term in the symplectic form, we use Lemma \ref{traceliederivative} and we obtain \begin{align*}
    \int_\Sg(f-1)\langle\liederivative_VJ,J\dot J\rangle\rho &=\int_\Sg\Big((1-f)(\dive_g\dot J)(V)+(f-1)\dive_g(\dot JV)\Big)\rho \\ &=\int_\Sg\Big((1-f)(\dive_g\dot J)(V)-\mathrm df(\dot JV)+\dive_g\big((f-1)\dot JV\big)\Big)\rho \\ &=-\int_\Sg\Big((f-1)(\dive_g\dot J)(V)+\mathrm df(\dot JV)\Big)\rho \ .
\end{align*}Moving on to the second term in the symplectic form \begin{align*}
    -\int_\Sg\frac{f'}{6}\langle\nabla_VA-\frac{3}{2}\tr(JM_V)AJ,\dot A_0J\rangle\rho&=-\int_\Sg\Big(-\frac{f'}{6}\beta(V)-\dot f\dive_g(JV)\Big)\rho \\ &=-\int_\Sg\Big(-\frac{f'}{6}\beta(V)+\mathrm d\dot f(JV)-\dive_g(\dot fJV)\Big)\rho \\ &=-\int_\Sg\Big(-\frac{f'}{6}\beta(V)+\mathrm d\dot f(JV)\Big)\rho \ .
\end{align*} In the end, combining the above two relations with (\ref{symplecticform3-4zero}), we obtain \begin{align*}
    (\ome_f)\big((\liederivative_VJ,g^{-1}\liederivative_VC), (\dot J,\dot A)\big)&=-\int_\Sg\Big((f-1)(\dive_g\dot J)(V)+\mathrm df(\dot JV)-\frac{f'}{6}\beta(V)+\mathrm d\dot f(JV)\Big)\rho \\ &=\int_\Sg\iota_V\Big((f-1)\dive_g\dot J+\mathrm df\circ\dot J-\frac{f'}{6}\beta+\mathrm d\dot f\circ J\Big)\rho \\ &=\int_\Sg\Big((f-1)\dive_g\dot J+\mathrm df\circ\dot J-\frac{f'}{6}\beta+\mathrm d\dot f\circ J\Big)\wedge\iota_V\rho \\ &=-\langle\mathrm d\widetilde\momentmap(\dot J,\dot A) \ | \ V\rangle_\Lsymp \ ,
\end{align*} where in the third step we used equation (\ref{vectorfieldand1form}). \vspace{0.1cm}\newline \textbf{Proof of relation (\ref{symplecticform3-4zero})}\newline Once again, the presence of the $1$-form part in $A$ makes the analysis more difficult. In fact, there is an additional term which does not appear in the anti-de Sitter case. If $A=0$ the identity is clearly satisfied. Suppose $A\neq 0$, then for any $p\in \Sg$ outside the zeroes of $A$ the elements $A_1$ and $A_2=A_1J$ form a basis for the space of $g_J$-symmetric and trace-less endomorphisms of $T_p\Sg$. Let $\{e_1, e_2\}$ be a $g_J$-orthonormal basis and let $\{e_1^*, e_2^*\}$ be its dual. Following the approach used to prove Proposition \ref{prop:differentialourmomentmap}, we have \begin{align*}
    A\big(M_V^{\mathrm s}\cdot\big)&=\frac{1}{\tr(A_1^2)}\Big(\tr(M_V^\mathrm sA_1)\big(C_{111}A_1+C_{112}A_2\big)+\tr(A_1JM_V^\mathrm s)\big(C_{112}A_1-C_{111}A_2\big)\Big)e_1^* \\[0.7em] & \ \ \ +\frac{1}{\tr(A_1^2)}\Big(\tr(M_V^\mathrm sA_1)\big(C_{112}A_1-C_{111}A_2\big)+\tr(A_1JM_V^\mathrm s)\big(C_{111}A_1+C_{112}A_2\big)\Big)e_2^* \ , \\[0.7em] &
    AM_V^\mathrm s+M_V^\mathrm sA=\tr(M_V^\mathrm sA_1)\mathds 1 e_1^*+\tr(A_1JM_V^\mathrm s)\mathds 1e_2^*, \\[0.7em] & (\dot A J+A\dot J)_{\text{tr}}=\frac{1}{2}\tr(\dot A_2)\mathds 1 e_1^*-\frac{1}{2}\tr(\dot A_1)\mathds 1 e_2^*,\qquad \dot A_0J=(\dot A_1)_0Je_1^*+(\dot A_2)_0Je_2^*
\end{align*} In particular, we can write the two terms in (\ref{symplecticform3-4zero}) as follows: \begin{align*}
    &-\frac{1}{2}\langle AM_V^{\mathrm{s}}+M_V^{\mathrm{s}}A, (\dot AJ+A\dot J)_{\text{tr}}\rangle=-\frac{1}{2}\Big(\tr(\dot A_2)\tr(A_1M_V^\mathrm s)-\tr(\dot A_1)\tr(A_1M_{JV}^\mathrm s)\Big),
\end{align*}\begin{align*}\langle A\big(M_V^{\mathrm s}\cdot\big),\dot A_0J\rangle&=\frac{\tr(A_1M_V^\mathrm s)}{\tr(A_1^2)}\Big(C_{111}\tr(A_1(\dot A_1)_0J)+C_{112}\tr(A_2(\dot A_1)_0J)\Big) \\[0.7em] & \ \ \ + \frac{\tr(A_1M_V^\mathrm s)}{\tr(A_1^2)}\Big(C_{112}\tr(A_1(\dot A_2)_0J)-C_{111}\tr(A_2(\dot A_2)_0J)\Big) \\[0.7em] & \ \ \ +\frac{\tr(A_1M_{JV}^\mathrm s)}{\tr(A_1^2)}\Big(C_{112}\tr(A_1(\dot A_1)_0J)-C_{111}\tr(A_2(\dot A_1)_0J)\Big) \\[0.7em] & \ \ \ -\frac{\tr(A_1M_{JV}^\mathrm s)}{\tr(A_1^2)}\Big(C_{111}\tr(A_1(\dot A_2)_0J)+C_{112}\tr(A_2(\dot A_2)_0J)\Big) \ . 
\end{align*}Finally, writing $\dot A$ in term of the variations of the tensor $C$, namely $$\dot A_1=\begin{pmatrix}
    \dot C_{111} & \dot C_{112} \\ \dot C_{112} & \dot C_{122}
\end{pmatrix},\qquad\dot A_2=\begin{pmatrix}
    \dot C_{112} & \dot C_{122} \\ \dot C_{122} & \dot C_{222}
\end{pmatrix}$$and using that $\tr(A_i(\dot A_k)_0J)=\tr(A_i\dot A_kJ)$ for any $i,j=1,2$, a direct computation shows the desired equality. Since (\ref{symplecticform3-4zero}) holds on the complement of a finite set in $\Sg$, it holds everywhere by continuity of the expression.
\end{proof}
\begin{remark}
    It is crucial to emphasize the importance of the result just proved. From the general theory of moment maps (see Definition \ref{def:momentmap}) we know that (\ref{integrationbypart}) follows from Corollary \ref{cor:momentmaptilde} if $\mathrm d\widetilde\momentmap$ is paired with Hamiltonian vector fields. The point is that $\widetilde\momentmap$ can not be promoted to a moment map for the action of $\Symp(\Sg,\rho)$, which still preserves $\ome_f$. In particular, the formula showed above is far from being obvious when computed for a symplectic vector field, which decomposes as the sum of a harmonic and a Hamiltonian vector field (see (\ref{splittingsymplecticvectorfield})).
\end{remark}
\begin{lemma}\label{lem:linearizedcodazzi}
    Let $(J,A)\in\pickg$, then the kernel of the linearized Codazzi-like equation $\mathrm d^\nabla A=0$ is given by $$\big\{(\dot J,\dot A)\in T_{(J,A)}\pickg \ | \ \mathrm d^\nabla\dot A_0(\bullet,\bullet)-J(\dive_g\dot J\wedge A)(\bullet,\bullet)=0\big\} \ .$$
\end{lemma}
\begin{proof}
    Recall that, for any vector fields $X,Y,Z\in\Gamma(T\Sg)$, we have \begin{equation}\label{dnabla}(\mathrm d^\nabla A)(X,Y)Z=(\nabla_XA)(Y)Z-(\nabla_YA)(X)Z \ .\end{equation} Therefore, we need to compute the derivative of (\ref{dnabla}) with respect to variations of $(J,A)$. For instance, \begin{align*}\Big((\nabla_XA)(Y)Z-(\nabla_YA)(X)Z\Big)'=(\dot\nabla_XA)(Y)Z-(\dot\nabla_YA)(X)Z+(\mathrm d^\nabla A')(X,Y)Z \ ,\end{align*}where $A'=J\dot JA+\dot A$. The part involving the variation of the connection has already been computed in the proof of Proposition \ref{prop:differentialourmomentmap} \begin{align*}
        (\dot\nabla_XA)(Y)Z=\frac{1}{2}\Big(3(\dive\dot J)(X)A(Y)JZ+A\big(J(\nabla_X\dot J)Y\big)Z+A(Y)J(\nabla_X\dot J)Z-J(\nabla_X\dot J)A(Y)Z\Big).
    \end{align*}Subtracting the term $(\dot\nabla_YA)(X)Z$ from the last expression and using Lemma \ref{lem:divergenzaendomorfismo} on $A\big((\nabla_X\dot J)Y-(\nabla_Y\dot J)X\big)$, we get \begin{align*}
        (\dot\nabla_XA)(Y)Z-(\dot\nabla_YA)(X)Z&=J\Big((\dive\dot J)(Y)A(X)-(\dive\dot J)(X)A(Y)\Big)Z+\frac{1}{2}JA(X)(\nabla_Y\dot J)Z \\ & \ \ \ +\frac{1}{2}J\Big((\nabla_Y\dot J)A(X)-A(Y)(\nabla_X\dot J)-(\nabla_X\dot J)A(Y)\Big)Z \\ &=-J(\dive\dot J\wedge A)(X,Y)Z+\frac{1}{2}JA(X)(\nabla_Y\dot J)Z \\ & \ \ \ +\frac{1}{2}J\Big((\nabla_Y\dot J)A(X)-A(Y)(\nabla_X\dot J)-(\nabla_X\dot J)A(Y)\Big)Z \ .
    \end{align*}
    Regarding the term with the exterior covariant derivative of $A'$, we have \begin{equation}
        (\mathrm d^\nabla A')(X,Y)Z=\underbrace{(\mathrm d^\nabla (J\dot JA))(X,Y)Z}_{\text{term} \ (a)}+\underbrace{(\mathrm d^\nabla\dot A_{\text{tr}})(X,Y)Z}_{\text{term} \ (b)}+(\mathrm d^\nabla\dot A_0)(X,Y)Z \ .
    \end{equation} The term $(a)$ is easy to handle since $\nabla_\bullet J=0$ and $\mathrm d^\nabla A=0$, \begin{align*}
        (\mathrm d^\nabla (J\dot JA))(X,Y)Z&=\nabla_X(J\dot JA)(Y)Z-\nabla_Y(J\dot JA)(X)Z \\&=J\Big((\nabla_X\dot J)A(Y)-(\nabla_Y\dot J)A(X)\Big)Z  \ .
    \end{align*}As for the term $(b)$, recall that $\dot A_{\text{tr}}=\frac{1}{2}\tr(\dot JJA)\mathds 1$, hence \begin{align*}
        (\mathrm d^\nabla\dot A_{\text{tr}})(X,Y)Z&=(\nabla_X\dot A_{\text{tr}})(Y)Z-(\nabla_Y\dot A_{\text{tr}})(X)Z \\[0.7em] &=\frac{1}{2}\tr(\nabla_X\big(\dot JJA(Y)\big))Z-\frac{1}{2}\tr(\nabla_Y\big(\dot JJA(X)\big))Z \\[0.7em] &=\frac{1}{2}\tr((\nabla_X\dot J)JA(Y)-(\nabla_Y\dot J)JA(X))Z \ .
    \end{align*}We conclude if we show that \begin{align*}
        \frac{1}{2}\tr((\nabla_X\dot J)JA(Y)-(\nabla_Y\dot J)JA(X))Z&=-J(\nabla_X\dot J)A(Y)Z+J(\nabla_Y\dot J)A(X)Z \\ & \ \ \ -\frac{1}{2}JA(X)(\nabla_Y\dot J)Z-\frac{1}{2}J(\nabla_Y\dot J)A(X)Z \\ & \ \ \ +\frac{1}{2}JA(Y)(\nabla_X\dot J)Z+\frac{1}{2}J(\nabla_X\dot J)A(Y)Z \ ,
    \end{align*}which follows from the fact that the elements $JA(X)\nabla_Y\dot J-J(\nabla_Y\dot J)A(X)$ and $J(\nabla_X\dot J)A(Y)-JA(Y)\nabla_X\dot J$ are both trace-term, and they can be written as \begin{align*}&JA(X)\nabla_Y\dot J-J(\nabla_Y\dot J)A(X)=-\tr(J(\nabla_Y\dot J)A(X))\mathds 1,\\[0.7em] &J(\nabla_X\dot J)A(Y)-JA(Y)\nabla_X\dot J=-\tr(JA(Y)\nabla_X\dot J)\mathds 1 \ .\end{align*}
\end{proof}
\begin{proposition}\label{prop:Winsidehaffrhozero}
    Let $(J,A)\in\haffrhozerotilde$ and consider the space $W_{(J,A)}$ defined by the system of equations (\ref{differentialequations}). Then, $$W_{(J,A)}\subset T_{(J,A)}\haffrhozerotilde \ .$$
\end{proposition}
\begin{proof}
    According to Corollary \ref{cor:zerolocusmomentmap}, the infinite-dimensional space $\haffrhozerotilde$ can be seen as the intersection of $\widetilde\momentmap^{-1}(0)$ with $\mathcal{M}_\mathrm C:=\{(J,A)\in\pickg \ | \ \mathrm d^\nabla A=0\}$. In particular, the tangent space to the pre-image of the zero locus of the moment map is identified with $\mathrm{Ker}\big(\mathrm d\widetilde\momentmap\big)$. On the other hand, Proposition \ref{prop:differentialourmomentmap} and Lemma \ref{lem:linearizedcodazzi} together implies that $$(\dot J,\dot A)\in T_{(J,A)}\haffrhozerotilde\iff\begin{cases}
      \mathrm d\Big((f-1)\dive_g\dot J+\mathrm df\circ\dot J+\mathrm d\dot f\circ J-\frac{f'}{6}\langle\dot A_0,(\nabla_\bullet A)J\rangle\Big)=0  \\ \mathrm d^\nabla\dot A_0(\bullet,\bullet)-J(\dive_g\dot J\wedge A)(\bullet,\bullet)=0
    \end{cases}$$Looking again at the equations (\ref{differentialequations}) defining the space $W_{(J,A)}$, it is clear that $$W_{(J,A)}\subset T_{(J,A)}\haffrhozerotilde \ .$$
\end{proof}At this point, it must be noted that the subspace we are interested in can be described as \begin{equation}\label{charcaterizationofW}W_{(J,A)}=\left\{(\dot J,\dot A)\in T_{(J,A)}\pickg \ \bigg| \ \parbox{15em}{$(\dot J,\dot A),\i(\dot J,\dot A)\in\mathrm{Ker}(\mathrm d\widetilde\momentmap)$ \\ $\mathrm d^\nabla\dot A_0(\bullet,\bullet)-J(\dive_g\dot J\wedge A)(\bullet,\bullet)=0$}\right\}\end{equation}
which clarifies the connection of the first two equations in (\ref{differentialequations}) with symplectic reduction theory.
\begin{manualtheorem} H
For any $(J,A)\in\haffrhozerotilde$, the vector space $W_{(J,A)}$ is the largest subspace in $T_{(J,A)}\haffrhozerotilde$ that is: \begin{itemize}
    \item[$\bullet$] invariant under the complex structure $\i$;
    \item[$\bullet$] $\g_f$-orthogonal to the orbit $T_{(J,A)}\big(\Ham(\Sg,\rho)\cdot(J,A)\big)$.
\end{itemize}
\end{manualtheorem}
\begin{proof}
    Recall from Corollary \ref{cor:zerolocusmomentmap} that the space $\haffrhozerotilde$ can be identified with $\widetilde\momentmap^{-1}(0)\cap\mathcal M_\mathrm C$, where $\mathcal{M}_\mathrm C:=\{(J,A)\in\pickg \ | \ \mathrm d^\nabla A=0\}$. Let us denote with $\widetilde W$ the largest subspace in $T_{(J,A)}\haffrhozerotilde$ that is $\g_f$-orthogonal to $T_{(J,A)}\big(\Ham(\Sg,\rho)\cdot(J,A)\big)$ and $\i$-invariant. Suppose that $(\dot J,\dot A)\in \widetilde W$, hence the same is true for $\i(\dot J,\dot A)$ by $\i$-invariance.  In particular, both $(\dot{J}, \dot{A})$ and  $\i(\dot J,\dot A)$ lie in $\mathrm{Ker}(\mathrm d\widetilde\momentmap)$. We now note that $(\dot{J}, \dot{A})$ is $\g_f$-orthogonal to the $\Ham(\Sg,\rho)$-orbit if and only if $\i(\dot{J}, \dot A)$ lies in $\mathrm{Ker}(\mathrm d\widetilde\momentmap)$ thanks to the following computation \begin{align*}
        \g_f\big((\liederivative_VJ,g_J^{-1}\liederivative_V C),(\dot J,\dot A)\big)&=-\g_f\big((\liederivative_VJ,g_J^{-1}\liederivative_V C), \i^2(\dot J,\dot A)\big) \\ &=-\ome_f\big((\liederivative_VJ,g_J^{-1}\liederivative_V C), \i(\dot J,\dot A)\big) \\ &=\langle\mathrm d\widetilde\momentmap\big(\i(\dot J,\dot A)\big) \ | \ V\rangle_{\Lham},\quad \forall \ V\in\Lham(\Sg,\rho) \ . 
    \end{align*} 
    by definition of moment map. This implies that an element $(\dot J,\dot A)$ belongs to $\widetilde W$ if and only if $$\begin{cases}
        \mathrm d\widetilde\momentmap(\dot J,\dot A)=0 \\ \mathrm d\widetilde\momentmap\big(\i(\dot J,\dot A)\big)=0 \\ \mathrm d^\nabla\dot A_0(\bullet,\bullet)-J(\dive_g\dot J\wedge A)(\bullet,\bullet)=0 \ ,
    \end{cases}$$ which is equivalent to the system of partial differential equations (\ref{differentialequations}) defining the subspace $W_{(J,A)}$ (see Proposition \ref{prop:differentialourmomentmap}).
\end{proof}

\section{Final properties}
In this final section of the paper we discuss about the non-degeneracy of the pseudo-metric over the entire $\SL(3,\R)$-Hitchin component, and we study the relation of $\ome_f$ with Goldman's symplectic form on the Fuchsian locus, both along horizontal and vertical directions. 
\subsection{The finite-dimensional quotient}\label{sec:6.1}
 Although the main part of the results have been shown, it still remains to prove Theorem \ref{thmF}, namely the identification of $\hitc$ with the finite dimensional quotient $\defgtilde/H$, where $\defgtilde$ is the smooth manifold of real dimension $16g-16+2g$ isomorphic to the quotient of the space $\haffrhozerotilde$ by the group $\Ham(\Sg,\rho)$ (see Theorem \ref{thmE}), and $H:=\Symp_0(\Sg,\rho)/\Ham(\Sg,\rho)$ is isomorphic to $H^1_{\text{dR}}(\Sg,\R)$ (see Lemma \ref{lem:fluxisomorphism}). The tangent space $T_{[J,A]}\defgtilde$ is identified with the vector space $W_{(J,A)}$ which is defined as the space of solutions to the following system of differential equations 
 \begin{equation*}
\begin{cases}
\mathrm d\big(\divr\big((f-1)\dot J\big)+\mathrm d\dot f\circ J-\frac{f'}{6}\beta\big)=0  \\ \mathrm d\big(\divr\big((f-1)\dot J\big)\circ J+\mathrm d\dot{f}_0\circ J-\frac{f'}{6}\beta\circ J\big)=0 \\ \mathrm d^\nabla\dot A_0(\bullet,\bullet)-J(\divr\dot J\wedge A)(\bullet,\bullet)=0
\end{cases}\end{equation*}Let us denote with $\alpha_1$ and $\alpha_2$ the $1$-forms in the above system whose differential is zero and let us introduce the vector space \begin{equation}\label{definitionV}
    V_{(J,A)}:=\left\{(\dot J,\dot A)\in T_{(J,A)}\haffrhozerotilde \ \bigg| \ \parbox{15em}{ $\alpha_1+i\alpha_2$ is exact \\ $\mathrm d^\nabla\dot A_0(\bullet,\bullet)-J(\divr\dot J\wedge A)(\bullet,\bullet)=0$}\right\}
\end{equation}It is not difficult to see, following the lines of the proof of Lemma \ref{lem:symplectomorphisminvariancedistribution} and Lemma \ref{lem:invariance by the complex structure}, that $V_{(J,A)}$ is invariant under the action of $\Symp(\Sg,\rho)$ and the complex structure $\i$. In what follows, although we will use the term "symplectic form" to denote $\ome_f$, we do not yet know whether on the spaces we are considering $\ome_f$ is actually non-degenerate. In any case, with abuse of terminology, the results we are about to present still apply.
\begin{proposition}\label{prop:orthogonaldecompositionW} 
There is a $\ome_f$-orthogonal decomposition \begin{equation*}
    W_{(J,A)}=V_{(J,A)}\overset{\perp_{\ome_f}}{\oplus}S_{(J,A)} \ ,
\end{equation*}where $S_{(J,A)}:=\{\big(\liederivative_XJ,g_J^{-1}\liederivative_XC\big) \ | \ X\in\Gamma(T\Sg), \ \mathrm d(\iota_X\rho)=\mathrm d(\iota_{JX}\rho)=0\}\cong T_{(J,A)}\big(H\cdot (J,A)\big)$ is the tangent space to the harmonic orbit.
\end{proposition}\begin{proof}
    Recall that, according to (\ref{integrationbypart}), for any symplectic vector field $X$ on the surface and for any $(\dot J,\dot A)\in T_{(J,A)}\pickg$, we have $$\ome_f\big((\liederivative_XJ, g_J^{-1}\liederivative_XC); (\dot J,\dot A)\big)=-\langle\mathrm d\widetilde\momentmap(\dot J,\dot A) \ | \ X\rangle_{\Lsymp} \ ,$$ where $\mathrm{d}\widetilde\momentmap(\dot J,\dot A)$ denotes the primitive found in Proposition \ref{prop:differentialourmomentmap} (see also Remark \ref{rem:fixedprimitive}). In particular, if $(\dot J,\dot A)\in V_{(J,A)}$ such a primitive equals the $1$-form $\alpha_1$ considered in (\ref{definitionV}), hence it is exact. Using the non-degenerate symplectic pairing (\ref{symplecticpairing}), we get $$\ome_f\big((\liederivative_XJ, g_J^{-1}\liederivative_XC); (\dot J,\dot A)\big)=-\langle\mathrm d\widetilde\momentmap(\dot J,\dot A), X\rangle_{\Lsymp}=0 \ ,$$ for any symplectic vector field $X$ and for any $(\dot J,\dot A)\in V_{(J,A)}$. In other words, $V_{(J,A)}$ is $\ome_f$-orthogonal to the symplectic orbit and it coincides with the $\ome_f$-orthogonal to $S_{(J,A)}$ inside $W_{(J,A)}$. For this reason, we can conclude if we show that $$V_{(J,A)}\cap S_{(J,A)}=\{0\} \ .$$ Suppose there exists a harmonic vector field $X$ such that $(\liederivative_XJ,g_J^{-1}\liederivative_XC)\in V_{(J,A)}$. By definition of $V_{(J,A)}$, the $1$-form $$\widetilde\alpha_1:=\divr\Big((f-1)\liederivative_XJ\Big)+\mathrm d\dot f\circ J-\frac{f'}{6}\langle\big(g_J^{-1}\liederivative_XC\big)_0,(\nabla_\bullet A)J\rangle$$ is exact. Therefore, \begin{align*}
        \int_\Sg\widetilde\alpha_1\wedge \iota_U\rho &=-\langle\mathrm d\widetilde\momentmap\big(\liederivative_XJ,g_J^{-1}\liederivative_XC\big), U\rangle_{\Lsymp} \tag{rel. (\ref{vectorfieldand1form})} \\ &=0, \quad \forall U\in\Lsymp(\Sg,\rho) \ .
    \end{align*}Since $X$ is harmonic, we can choose $U=JX$ and obtain \begin{align*}
        0&=\int_\Sg\Big(\divr\Big((f-1)\liederivative_XJ\Big)+\mathrm d\dot f\circ J-\frac{f'}{6}\beta\Big)\wedge \iota_{JX}\rho \tag{rel. (\ref{vectorfieldand1form})}\\ &=\int_\Sg\Big(\divr\Big((f-1)\liederivative_XJ\Big)+\mathrm d\dot f\circ J-\frac{f'}{6}\beta\Big)\big(JX\big)\rho \\ &=\int_\Sg\Big(\divr\Big((f-1)\liederivative_XJ\Big)-\frac{f'}{6}\beta\Big)(JX)\rho-\int_\Sg(\mathrm d\dot f)(X)\rho \\ &=\int_\Sg\Big(\divr\Big((f-1)\liederivative_XJ\Big)-\frac{f'}{6}\beta\Big)(JX)\rho-\int_\Sg\big(\divr(\dot fX)-\dot f\divr(X)\big)\rho \\ &=\int_\Sg\Big(\divr\Big((f-1)\liederivative_XJ\Big)-\frac{f'}{6}\beta\Big)(JX)\rho \ . \tag{$X$ is harmonic}
    \end{align*}Now let us compute the term \begin{align*}
    \beta(JX)&=\langle\big(g_J^{-1}\liederivative_XC\big)_0,(\nabla_{JX}A)J\rangle \tag{Theorem \ref{thm:picktensor}} \\ &=-\langle\big(g_J^{-1}\liederivative_XC\big)_0,\nabla_{X}A\rangle \tag{rel. (\ref{liederivativeCsenzatraccia})} \\ &=-\langle\nabla_XA-\frac{3}{2}\tr(JM_X)AJ+A\big(M_X^{\mathrm s}\cdot\big),\nabla_XA\rangle \tag{$JX$ is symplectic} \\ &=-\langle\nabla_XA+A\big(M_X^{\mathrm s}\cdot\big),\nabla_XA\rangle \\ &=-\vl\vl\nabla_XA\vl\vl^2-\langle A\big(M_X^{\mathrm s}\cdot\big),\nabla_XA\rangle \tag{Theorem \ref{thm:picktensor}} \\ &=-\vl\vl\nabla_XA\vl\vl^2+\langle A\big(M_X^{\mathrm s}\cdot\big),(\nabla_{JX}A)J\rangle \tag{$\nabla_\bullet J=0$} \\ &=-\vl\vl\nabla_XA\vl\vl^2+\langle A\big(M_X^{\mathrm s}\cdot\big),(\nabla_{JX}A)J+A\nabla_{JX}J\rangle \ .
    \end{align*}Applying equation (\ref{symplecticform3-4zero}) to the last term with $\dot A_0=\nabla_{JX}A$ and $\dot J=\nabla_{JX}J$, we get \begin{align*}
        \beta(JX)&=-\vl\vl\nabla_XA\vl\vl^2+\langle A\big(M_X^{\mathrm s}\cdot\big),(\nabla_{JX}A)J+A\nabla_{JX}J\rangle \\ &=-\vl\vl\nabla_XA\vl\vl^2+\frac{1}{2}\langle AM_X^\mathrm s+M_X^\mathrm sA,\Big((\nabla_{JX}A)J+A\nabla_{JX}J\Big)_{\text{tr}}\rangle \\ &=-\vl\vl\nabla_XA\vl\vl^2 \ ,
    \end{align*} where we used that the endomorphism part of $(\nabla_{JX}A)J$
is trace-less. In order to study the divergence term, let us first make some preliminary observations. Let $L:\Gamma(T\Sg)\to\mathrm{End}_0(T\Sg,g_J)$ be the Lie derivative operator. It can be shown that its $L^2$-adjoint is $L^*(\dot J)=-J(\divr_{g_J}\dot J)^\#$ (\cite{tromba2012teichmuller}), where $\#:\Omega^1(\Sg)\to\Gamma(T\Sg)$ is the musical isomorphism induced by the metric $g_J$. Therefore, \begin{align*}
    \int_\Sg\Big(\divr\Big((f-1)\liederivative_XJ\Big)\Big)(JX)\rho&=\int_\Sg \langle\divr\Big((f-1)\liederivative_XJ\Big)^\#,JX\rangle\rho \\ &=-\int_\Sg \langle J\Big(\divr\Big((f-1)\liederivative_XJ\Big)\Big)^\#,X\rangle\rho \\ &=\int_\Sg\langle(f-1)\liederivative_XJ,\liederivative_XJ\rho\rangle \\ &=\int_\Sg(f-1)\vl\vl\liederivative_XJ\vl\vl^2\rho \ .
\end{align*}Referring back to the term we are interested in, we conclude $$\int_\Sg(f-1)\vl\vl\liederivative_XJ\vl\vl^2\rho+\frac{1}{6}\int_\Sg f'\vl\vl\nabla_XA\vl\vl^2\rho=0 $$ and, since $f,f'$ are both strictly negative, this is possible if and only if $\liederivative_XJ=\nabla_XA=0$. Given that on a Riemann surface $(\Sg, J)$ of genus $g\ge 2$ there are no non-zero biholomorphism isotopic to the identity, it follows that $X=0$.
\end{proof}
%According to what has been provedn in Section \ref{sec:4.3} and Section \ref{sec:4.4}, on any such $W_{(J,A)}$ we have a triple $(\ome_f,\i,\g_f)$ which satisfies all the properties of being a pseudo-K\"ahler structure, except being non-degenerate for the pseudo-metric whenever $A\neq 0$. \begin{proposition}\end{proposition}
\begin{lemma}\label{lem:complexsymplecticsubspace}The vector space $S_{(J,A)}$ is a complex-symplectic subspace of $\big(W_{(J,A)}, \i,\ome_f\big)$ isomorphic to $H^1_{\text{dR}}(\Sg,\R)$. \end{lemma}
\begin{proof}
Requiring $S_{(J,A)}$ to be a complex subspace of $\big(W_{(J,A)}, \i\big)$ is equivalent to say that it is preserved by the action of the complex structure. For instance, if $(\liederivative_XJ,g_J^{-1}\liederivative_XC)\in S_{(J,A)}$ then $\i\big(\liederivative_XJ,g_J^{-1}\liederivative_XC\big)=\big(-\liederivative_{JX}J,-g_J^{-1}\liederivative_{JX}\big)$ (see Lemma \ref{lem:complexstructureliederivative}). Since $X$ is harmonic, i.e. $X$ and $JX$ are symplectic vector field, the element $\big(-\liederivative_{JX}J,-g_J^{-1}\liederivative_{JX}\big)$ belongs to $S_{(J,A)}$ as $\mathrm d\big(\iota_{J^2X}\rho\big)=-\mathrm d\big(\iota_X\rho\big)=0$. Moreover, according to Proposition \ref{prop:orthogonaldecompositionW}, we have $$S_{(J,A)}\cap\big(S_{(J,A)}\big)^{\perp_{\ome_f}}=\{0\} \ ,$$ which implies that $S_{(J,A)}$ is a symplectic subspace of $\big(W_{(J,A)},\ome_f\big)$ endowed with the restricted symplectic form.
\newline
 Now if $(\liederivative_XJ,g_J^{-1}\liederivative_XC)\in S_{(J,A)}$, then $\mathrm{d}(\iota_X\rho)=\mathrm d(\iota_{JX}\rho)=0$. In particular, $$0=\mathrm d(\iota_{JX}\rho)=-\mathrm d(\iota_X\rho\circ J)$$ and since $\iota_X\rho\circ J=\ast_J(\iota_X\rho)$, we conclude that $\iota_X\rho$ is a harmonic $1$-form. This gives a well-defined map from $S_{(J,A)}$ to the space of harmonic $1$-forms on the surface, which is isomorphic to $H^1_{\text{dR}}(\Sg,\R)$ by Hodge theory. The map is an isomorphism since for any cohomology class $[\gamma]\in H^1_{\text{dR}}(\Sg, \R)$ there exists a unique harmonic representative, which is of the form $\iota_X\rho$, for some harmonic vector field $X$ on the surface (see Lemma \ref{lem:vectorfieldssurface}). \end{proof}
 \begin{remark}\label{rem:orthogonaldecompositionpseudometric}
    It should be noted that the decomposition of Proposition \ref{prop:orthogonaldecompositionW} is also orthogonal with respect to $\g_f$. In fact, $\g_f(\cdot,\cdot)=\ome_f(\i\cdot,\cdot)$ and using the $\i$-invariance of $S_{(J,A)}$ it follows that $$V_{(J,A)}=\big(S_{(J,A)}\big)^{\perp_{\ome_f}}=\big(S_{(J,A)}\big)^{\perp_{\g_f}}\subset W_{(J,A)} \ .$$
\end{remark}
In Section \ref{sec:3.2}, we discussed how to obtain Teichm\"uller space by means of symplectic reduction theory and we argued how the symplectic form is actually part of a K\"ahler metric. If $\mu$ denotes the moment map of Theorem \ref{thm:momentmapteich}, the quotient space $\widetilde{\mathcal{T}}(\Sg)=\mu^{-1}(0)/\Ham(\Sg,\rho)$ is a smooth manifold of dimension $6g-6+2g$ with a natural $H$-action. In particular, since the action is free and proper, the quotient map $p:\widetilde{\mathcal{T}}(\Sg)\to\mathcal{T}(\Sg)$ is an $H$-principal bundle. On the other hand, there is a $\mathrm{MCG}(\Sg)$-equivariant projection map $\widetilde\pi:\defgtilde\to\widetilde{\mathcal{T}}(\Sg)$ which allows us to lift the $H$-action from $\widetilde{\mathcal{T}}(\Sg)$ to $\defgtilde$. By a standard argument, the $H$-action on $\defgtilde$ is free and proper as well (see \cite[Proposition 6.3.3]{labourie2008cross}).
In the end, the quotient $\defgtilde/H$ results in an identification with $\defg$ so that the following diagram commutes \begin{equation*} \begin{tikzcd} {} \defgtilde \arrow[r, "\widetilde\pi"] \arrow[d, "p'"'] & \widetilde{\mathcal{T}}(\Sg) {} \arrow[d, "p"] \\
\defg \arrow[r, "\pi"]             & \mathcal{T}(\Sg)  {}              \end{tikzcd}\end{equation*}where $\pi:\defg\to\mathcal T(\Sg)$ is the $\mathrm{MCG}(\Sg)$-equivariant holomorphic vector bundle map given by Theorem \ref{teoloftinlabourie}, and $p':\defgtilde\to\defg$ is the quotient projection. According to Proposition \ref{prop:orthogonaldecompositionW} and Lemma \ref{lem:complexsymplecticsubspace}, the orbits $H$-orbits in $\defgtilde$ are complex-symplectic submanifolds, therefore there is a well-defined complex structure $\i$ and symplectic form $\ome_f$ on the quotient (see \cite[Lemma 4.4.9]{trautwein2018infinite}), giving rise to a pseudo-K\"ahler metric on the $\SL(3,\R)$-Hitchin component. In other words, we proved the following
\begin{manualtheorem}G 
The $H$-action on $\defgtilde$ is free and proper, with complex and symplectic $H$-orbits. Moreover, the pseudo-K\"ahler structure $(\g_f,\i,\ome_f)$ descend to the quotient which is identified with $\hitc$. Finally, the complex structure $\i$ induced on the $\SL(3,\R)$-Hitchin component coincides with the one found by Labourie and Loftin.\end{manualtheorem} 

\subsection{The pseudo-metric is non-degenerate on the orbit}\label{sec:6.2}
Here we want to study the set $\mathcal{M}_\mathrm C=\{(J,A)\in\pickg \ | \ \mathrm d^\nabla A=0\}$, namely the subspace of $\pickg$ where the Codazzi-like equation for hyperbolic affine spheres (see (\ref{Gausscodazzi})) is satisfied. The main result of this section is contained in Corollary \ref{cor:nondegeneratemetriconVandMc}.
\begin{lemma}\label{lem:difforbitinsidecodazzi}
    Let $(J,A)$ be a point in $\mathcal M_\mathrm C$, then $$T_{(J,A)}\big(\Diff(\Sg)\cdot (J,A)\big)\subset T_{(J,A)}\mathcal{M}_\mathrm C \ .$$ Moreover, the tangent space $T_{(J,A)}\mathcal M_\mathrm C$ admits the following decomposition: \begin{equation*}
        V_{(J,A)}\overset{\perp_{\g_f}}{\oplus} S_{(J,A)}\overset{\perp_{\g_f}}{\oplus} T_{(J,A)}\big(\Ham(\Sg,\rho)\cdot(J,A)\big)\overset{\perp_{\g_f}}{\oplus}\i\Big(T_{(J,A)}\big(\Ham(\Sg,\rho)\cdot(J,A)\big)\Big) \ .
    \end{equation*}
\end{lemma}
\begin{proof}
    If $(J,A)\in\mathcal M_\mathrm C$, then $\mathrm d^\nabla A=0$ where $\nabla$ is the Levi-Civita connection with respect to $g_J=\rho(\cdot,J\cdot)$. In particular, $A=g_J^{-1}C=g_J^{-1}\Ree(q)$ where $q$ is a $J$-complex cubic differential on $(\Sg, J)$ so that equation $\mathrm d^\nabla A=0$ is equivalent to $\bar\partial_J q=0$ (see Theorem \ref{thm:picktensor}). Now let $X\in\Gamma(T\Sg)$ and consider its flow $\{\phi_t\}\subset\Diff(\Sg)$, namely $X=\frac{\mathrm d}{\mathrm dt}\phi_t|_{t=0}$ and $\phi_0=\mathrm{Id}$. Let us define $$J_t:=\mathrm d\phi_t^{-1}\circ J\circ\mathrm d\phi_t,\qquad C_t:=C(\mathrm d\phi_t\cdot,\mathrm d\phi_t\cdot,\mathrm d\phi_t\cdot),\qquad q_t:=\phi_t^*q \ .$$ It is not difficult to show that $q$ is holomorphic with respect to $J$ if and only if $q_t$ is holomorphic with respect to $J_t$. Therefore, to conclude the proof of the first part of the statement, we only need to show that $\Ree(C_t)=q_t$. This last identity can be proven with a computation in coordinates. In fact, let $\{x,y\}$ be isothermal coordinates on the surface, so that $g_J=e^u(\mathrm dx^2+\mathrm dy^2)$ and $q=(P+iQ)\mathrm dz^3$, with $P+iQ$ a $J$-holomorphic function. Then, we get $$C=P\mathrm dx^3-3P\mathrm dx\odot\mathrm dy^2-3Q\mathrm dx^2\odot\mathrm dy+Q\mathrm dy^3 \ ,$$ where $\odot$ denotes the symmetric product. Plugging in the action of the flow $(\phi_t)$ on the expressions above for $q$ and $C$ gives the claim. Regarding the decomposition, we already know by Lemma \ref{lem:vectorfieldssurface} that $T_{(J,A)}\big(\Diff(\Sg)\cdot (J,A)\big)$ splits as a direct sum $$T_{(J,A)}\big(H\cdot (J,A)\big)\oplus T_{(J,A)}\big(\Ham(\Sg,\rho)\cdot(J,A)\big)\oplus\i\Big(T_{(J,A)}\big(\Ham(\Sg,\rho)\cdot(J,A)\big)\Big) \ ,$$ where $H:=\Symp_0(\Sg\,\rho)/\Ham(\Sg,\rho)$. In particular, by Lemma \ref{lem:complexsymplecticsubspace} the tangent to the harmonic orbit is identified with $S_{(J,A)}$. Let $U$ be a Hamiltonian vector field on the surface. The $\g_f$-orthogonality follows from the following computation: \begin{align*}
\g_f\big((\liederivative_UJ,g^{-1}\liederivative_UC);\i(\liederivative_UJ,g^{-1}\liederivative_UC)\big)=\ome_f\big((\liederivative_UJ,g^{-1}\liederivative_UC);(\liederivative_UJ,g^{-1}\liederivative_UC)\big)=0 \ ,
    \end{align*}and by $\i$-invariance of $S_{(J,A)}$, which is contained in the largest subspace in $T_{(J,A)}\haffrhozerotilde$ that is $\g_f$-orthogonal to the Hamiltonian orbit (see Theorem \ref{thmG}). Finally, $V_{(J,A)}$ is $\g_f$-orthogonal to the symplectic orbit by Proposition \ref{prop:integrazioneperparti} and to the space $\i\Big(T_{(J,A)}\big(\Ham(\Sg,\rho)\cdot(J,A)\big)\Big)$ by $\i$-invariance.
\end{proof}
\begin{proposition}\label{prop:metricnondegenerateontheorbit}
    Let $(J,A)$ be a point in $\mathcal M_\mathrm C$. Then, the pseudo-metric $\g_f$ is non-degenerate when restricted to the following subspaces: $$S_{(J,A)}, \ \ T_{(J,A)}\big(\Ham(\Sg,\rho)\cdot(J,A)\big), \ \ \i\Big(T_{(J,A)}\big(\Ham(\Sg,\rho)\cdot(J,A)\big)\Big). $$
\end{proposition}\begin{proof}
    The pseudo-metric $\g_f$ is non-degenerate on the Hamiltonian orbit as a consequence of Lemma \ref{lem:Wisomorphictofinitedimensionalquotient} and Theorem \ref{thmG}, indeed they imply together the following condition $$T_{(J,A)}\big(\Ham(\Sg,\rho)\cdot(J,A)\big)\cap\Big(T_{(J,A)}\big(\Ham(\Sg,\rho)\cdot(J,A)\big)\Big)^{\perp_{\g_f}}=\{0\} \ .$$ Moreover, the same is true on the Hamiltonian orbit after applying the complex structure $\i$ since $\g_f(\i\cdot,\i\cdot)=\g_f(\cdot,\cdot)$. Regarding the subspace $S_{(J,A)}$, we get the thesis directly from the proof of Lemma \ref{lem:complexsymplecticsubspace} (see also Remark \ref{rem:orthogonaldecompositionpseudometric}).
\end{proof}
\begin{corollary}\label{cor:nondegeneratemetriconVandMc}
    Let $(J,A)$ be a point in $\mathcal M_\mathrm C$. Then, the following are equivalent:\newline
        $\bullet$ $\g_f$ is non-degenerate on $T_{(J,A)}\mathcal M_\mathrm C$;\newline
        $\bullet$ $\g_f$ is non-degenerate when restricted to $V_{(J,A)}$
\end{corollary}
\begin{proof}
    The tangent space $T_{(J,A)}\mathcal M_\mathrm C$ decomposes in the $\g_f$-orthogonal direct sum of four subspaces (Lemma \ref{lem:difforbitinsidecodazzi}). Thanks to Proposition \ref{prop:metricnondegenerateontheorbit} we know that the metric $\g_f$ is non-degenerate on three out of four spaces, and the one not counted is exactly $V_{(J,A)}$. Using that the decomposition is $\g_f$-orthogonal, the thesis follows directly.
\end{proof}

\subsection{A new result regarding Goldman's symplectic form}
In \cite{Goldman_affine}, Goldman gave a gauge theoretic formula for a symplectic form $\ome_G$ on the Hitchin component, which carries his name: if we identify $\Hit_3(\Sigma)$ with the moduli space of flat $\mathrm{SL}(3,\R)$-connections, by associating to $\rho \in \Hit_3(\Sigma)$ the Blaschke connection $\nabla$ of the unique $\rho$-equivariant affine sphere, we can interpret tangent vectors to the Hitchin component as variations $\dot{\nabla}$ of the connection. Then the pairing 
\[
         \ome_G(\dot{\nabla}, \dot{\nabla}') = \int_\Sigma \tr(\dot{\nabla} \wedge \dot{\nabla}') -\frac{1}{3} \tr(\dot{\nabla}) \wedge \tr(\dot{\nabla}')
\]
induces a non-degenerate symplectic form on $\Hit_3(\Sigma)$. The techniques we developed in Section \ref{sec:5.2} allows us to explicitly compute this form at the Fuchsian locus along \textit{all} variations, tangent or transverse to the Fuchsian locus.\\

\noindent Recall that the Blaschke connection $\nabla$ of the affine sphere with embedding data $(J, A)$ can be written as $\nabla=\nabla^h+\sqrt{2}A$, where $\nabla^h$ denotes the Levi-Civita connection of the Blaschke metric $h$. See also Remark \ref{rem:osservazionedifferenzialecubicoriscalatowangequation}  and Remark \ref{rm:rescaling} that justify the presence of the factor $\sqrt{2}$. By the discussion in Section \ref{sec:2.3}, we can write $h=e^{F(\|A\|_0^2)} g_J$, where $F$ is the function introduced in Lemma \ref{lem:functionFef}. Therefore, by standard Riemannian geometry, we have
\[
        \nabla^{h}_X Y = \nabla^J_X Y + \frac{1}{2}\left(X(F)Y+Y(F)X-g_J(X,Y)\mathrm{grad}_J(F) \right) \ ,
\]
where $\nabla^J$ and $\mathrm{grad}_J$ denote the Levi-Civita connection and the gradient for the metric $g_J$, respectively. On the Fuchsian locus $A=0$, the Blaschke metric $h$ is the unique hyperbolic metric in the conformal class of $J$ by Equation (\ref{Gausscodazzi}) and, by our choice of the area form $\rho$ on $\Sigma$ with volume $-2\pi\chi(\Sigma)$, $F$ is constantly equal to $0$ on $\Sigma$. Thus for any $(\dot{J}, \dot{A}) \in V_{(J,0)}$ the induced variation on the Blaschke connection is
\begin{align*}
        \dot{\nabla}_X Y &=\dot{\nabla}^J_X Y +\frac{1}{2}\big( X(\dot{F})Y  + Y(\dot{F})X-\dot{g}_J(X,Y)\mathrm{grad}_J(F) \\
        & \ \ \ \ - g_J(X,Y)\dot{\mathrm{grad}}_J(F)-g_J(X,Y)\mathrm{grad}_J(\dot{F})\big)+ \dot{A}(X)Y \\ 
        & =   \dot{\nabla}^J_X Y + \sqrt{2}\dot{A}(X)Y \tag{$\dot{F}=0$, see Lemma 3.22 in \cite{rungi2021pseudo}} \\
        & = -\frac{1}{2}\big((\dive\dot J)(X)JY+J(\nabla^J_X\dot J)Y\big) + \sqrt{2}\dot{A}(X)Y \tag{by Lemma \ref{lem:firstordervariationlevicivita}} \\
        & = -\frac{1}{2}J(\nabla^J_X\dot J)Y + \sqrt{2}\dot{A}_0(X)Y     
\end{align*}
because the trace part of $\dot{A}$ vanishes identically (see Lemma \ref{lem:tangentpickform}) and the tangent vector $(\dot{J}, \dot{A})$ is a solution to the system (\ref{equationHitchincomponent}), which on the Fuchsian locus simply says that $\dot{J}$ is divergence free and $\dot{A}_0$ is the real part of a holomorphic cubic differential. Since on the Fuchsian locus $\nabla^{J}$ is equal to the Blaschke connection $\nabla$, we will suppress the superscript $J$ from now on.

\begin{proof}[Proof of Theorem \ref{thmGoldman}] We show that $\ome_{G}$ and $\ome_f$ coincide at a point on the Fuchsian locus $(J,0)$ along any pair of tangent vectors. The reader may refer to Equation (\ref{symplecticform}) for the explicit formula of $\ome_f$ that we will use throughout the proof. Let us start with vertical variations, i.e. tangent vectors of the form $(0, \dot{A}_{0}), (0, \dot{A}_{0}') \in V_{(J,0)}$. From the discussion above, the corresponding variations of the Blaschke connection are $\dot{\nabla}=\sqrt{2}\dot{A}_{0}$ and $\dot{\nabla}'=\sqrt{2}\dot{A}_{0}'$. Therefore,
\[
    \ome_{G}(\dot{\nabla}, \dot{\nabla}')= 2\int_{\Sigma} \tr(\dot{A}_{0} \wedge \dot{A}_{0}')= 2\langle \dot{A}_{0}, \dot{A}_{0}'J\rangle
\]
because, since $\dot{A}_{0}'$ is the real part of a holomorphic cubic differential, in an orthonormal basis for $g_{J}$, we have $\dot{A}_{0}'=(\dot{A}_{0})_{1}e_{1}^{*}+(\dot{A}_{0})_{1}Je_{2}^{*}$. The last expression in the above formula equals exactly $(\ome_f)_{(J,0)}((0, \dot{A}_{0}), (0, \dot{A}_{0}'))$ because $f'(0)=-3$ by direct computation. \\
Let us now consider the pairing between a horizontal variation $(\dot{J}, 0)\in V_{(J,0)}$ and a vertical one $(0, \dot{A}_{0})$. We need to check that, if we denote again by $\dot{\nabla}$ and $\dot{\nabla}'$ the induced variations of the Blaschke connection, we have
\[
    \ome_G(\dot{\nabla}, \dot{\nabla}')=(\ome_f)_{(J,0)}((\dot{J},0), (0,\dot{A}_{0}'))=0 \ .
\]
Now, we know that $\dot{\nabla}=-\frac{1}{2}J(\nabla \dot{J})$ and $\dot{\nabla}'=\sqrt{2}\dot{A}_{0}'$, therefore 
\begin{align*}
    \ome_G(\dot{\nabla}, \dot{\nabla}') &= -\frac{\sqrt{2}}{2}\int_{\Sigma} \tr(J\nabla\dot{J} \wedge \dot{A}_{0}') \\
    &=-\frac{\sqrt{2}}{2}\int_{\Sigma} \tr(\nabla J\dot{J} \wedge \dot{A}_{0}') \tag{$\nabla J=0$} \\
    &= \frac{\sqrt{2}}{2} \int_{\Sigma} \tr(J\dot{J} \wedge d^{\nabla}\dot{A}_{0}') \tag{integrating by parts} \\
    &= 0 \tag{Equation (\ref{equationHitchincomponent}) with $A=0$} \ .
\end{align*}
We are left with the case of two horizontal tangent vectors $(\dot{J}, 0), (\dot{J}',0) \in V_{(J,0)}$, which induce the variations of Blaschke connection $\dot{\nabla}=-\frac{1}{2}J(\nabla \dot{J})$ and $\dot{\nabla}'=-\frac{1}{2}J(\nabla \dot{J}')$. We compute
\begin{align*}
    \ome_{G}(\dot{\nabla}, \dot{\nabla}')&= \frac{1}{4} \int_{\Sigma} \tr (J\nabla \dot{J} \wedge J\nabla \dot{J}') \\ 
    &=\frac{1}{4} \int_{\Sigma} \tr (\nabla J\dot{J} \wedge \nabla J\dot{J}') \tag{$\nabla J=0$} \\
    &= -\frac{1}{4} \int_{\Sigma} \tr (J\dot{J} \wedge d^{\nabla}(\nabla J\dot{J}')) \tag{integrating by parts}\\
    &= -\frac{1}{4} \int_{\Sigma} \tr(J\dot{J} \wedge [-R^{\nabla}, J\dot{J}'])
\end{align*}
because $d^{\nabla}\circ \nabla$ is the curvature operator of the connection on the bundle $\mathrm{End}(T\Sigma)$, which  coincides with the adjoint action of the Riemann tensor $R^{\nabla}$. The minus sign in the above formula comes from our conventions in the definition of the Riemann tensor (see page 52). Since $\nabla$ is the Levi-Civita connection of the hyperbolic metric compatible with $J$, we have $R^{\nabla}=-J\otimes \rho$, hence
\begin{align*}
    \ome_{G}(\dot{\nabla}, \dot{\nabla}')&= -\frac{1}{4} \int_{\Sigma} \tr(J\dot{J} \wedge [-R^{\nabla}, J\dot{J}']) \\
    &= -\frac{1}{4} \int_{\Sigma} \tr(J\dot{J}[J, J\dot{J}']) \rho\\
    & = \frac{1}{2} \int_{\Sigma} \tr(J\dot{J}\dot{J}') \rho \tag{$J\dot{J}'=-\dot{J}'J$} \\
    & = -\frac{1}{2} \int_{\Sigma} \tr(\dot{J}J\dot{J}') \rho \\
    & = (\ome_{f})_{(J,0)}((\dot{J},0),(\dot{J}',0)) \ . 
\end{align*}
Note that this last computation is compatible with the fact that Goldman symplectic form on Teichm\"uller space is equal to four times the Weil-Petersson symplectic form. \\
By bilinearity, we conclude that $\ome_{G}$ coincides with $\ome_{f}$ on the Fuchsian locus along all pairs of tangent vectors. 
\end{proof}

\noindent Theorem \ref{thmGoldman} implies that the pairing between Goldman symplectic form on the Hitchin component and the Labourie-Loftin complex structure cannot give rise to a positive definite K\"ahler structure. Moreover, it is reasonable to conjecture that Goldman symplectic form coincides with our $\ome_{f}$ at all points.

\emergencystretch=1em

\printbibliography

\end{document}